\documentclass[twoside,centertags,reqno,openany]{amsart}

\usepackage[english]{babel}

\usepackage[T1]{fontenc}			

\usepackage{cite}

\usepackage{hyperref}
\usepackage[dvipsnames]{xcolor}
\definecolor{darkbrown}{rgb}{0.7,0.2,0.1}
\definecolor{darkgreen}{rgb}{0.2,0.6,0.2}
\colorlet{linkcolor}{red!50!black}
\colorlet{citecolor}{red!50!black}
\colorlet{urlcolor}{red!50!black}
\hypersetup{
  colorlinks,
  linkcolor=Magenta,
  citecolor=Plum,
  urlcolor=Thistle
}
\usepackage{amsmath}
\usepackage{amsthm}
\usepackage[nameinlink]{cleveref}
\usepackage{amssymb}
\usepackage{mathtools}
\usepackage{microtype}
\usepackage{enumitem}
\usepackage{xcolor}
\usepackage{mathrsfs,stmaryrd,color}
\usepackage{relsize}
\usepackage{relsize}
\usepackage[bbgreekl]{mathbbol}
\usepackage{amsfonts}
\usepackage{graphicx}
\graphicspath{ {./images/} }

\usepackage{tikz}\usetikzlibrary{cd} 		
\usetikzlibrary{babel}	
\usetikzlibrary{decorations.pathmorphing}
\usetikzlibrary{decorations.markings}
\makeatletter
\tikzcdset{
open/.code={\tikzcdset{hook, circled};},
closed/.code={\tikzcdset{hook, slashed};},
open'/.code={\tikzcdset{hook', circled};},
closed'/.code={\tikzcdset{hook', slashed};},
circled/.code={\tikzcdset{markwith={\draw (0,0) circle (.375ex);}};},
slashed/.code={\tikzcdset{markwith={\draw[-] (-.4ex,-.4ex) -- (.4ex,.4ex);}};},
markwith/.code={
\pgfutil@ifundefined{tikz@library@decorations.markings@loaded}%
{\pgfutil@packageerror{tikz-cd}{You need to say %
\string\usetikzlibrary{decorations.markings} to use arrow with markings}{}}{}%
\pgfkeysalso{/tikz/postaction={/tikz/decorate,
/tikz/decoration={
markings,
mark = at position 0.5 with
{#1}}}}},
}
\makeatother

\numberwithin{equation}{section}
\numberwithin{equation}{subsection}
\newtheoremstyle{resultStyle}
	{}
	{}
	{\itshape}
	{}
	{\bfseries}
	{.}
	{.4em}
	{}
	
\newtheoremstyle{defStyle}
  {}
  {}
  {\upshape}
  {}
  {\bfseries}
  {.}
  {.3em}
  {}
  
\newtheoremstyle{remStyle}
{}
{}
{\upshape}
{}
{\itshape}
{.}
{.3em}
{}

\addto\captionsenglish{} 

\theoremstyle{resultStyle}
\newtheorem{thm}{Theorem}[subsection]
\newtheorem*{nonothm}{Theorem}
\newtheorem{prop}[thm]{Proposition}
\newtheorem{cor}[thm]{Corollary}
\newtheorem*{nonocor}{Corollary}

\theoremstyle{defStyle}
\newtheorem{defn}[thm]{Definition}
\newtheorem{lem}[thm]{Lemma}
\newtheorem{constr}[thm]{Construction}
\newtheorem{non}[thm]{}

\newtheorem{vari}[thm]{Variant}

\theoremstyle{remStyle}
\newtheorem{ex}[thm]{Example}
\newtheorem{rem}[thm]{Remark}
\newtheorem{nota}[thm]{Notation}

\makeatletter
\def\blfootnote{\xdef\@thefnmark{}\@footnotetext}

\DeclareSymbolFontAlphabet{\mathbb}{AMSb} 
\DeclareSymbolFontAlphabet{\mathbbl}{bbold}

\newcommand{\notehelper}[3]{\textcolor{#3}{$\blacksquare$}\marginpar{\ifodd\thepage\raggedright\else\raggedleft\fi\color{#3}\tiny \textbf{#2:} #1}}

\newcommand{\iso}{\simeq}
\newcommand{\from}{\colon}
\newcommand{\isoto}{\xrightarrow{\sim}}

\DeclareMathOperator{\colimit}{colim}
\DeclareMathOperator{\limit}{lim}

\newcommand{\Solid}{\scriptscriptstyle{\square}}

\newcommand{\Prism}{{\mathlarger{\mathbbl{\Delta}}}}
\newcommand{\Nyg}{\mathcal{N}}
\newcommand{\DR}{\text{dR}}
\newcommand{\Witt}{W}
\newcommand{\QCoh}{\mathcal{D}}
\newcommand{\DSolid}{\text{Solid}}
\newcommand{\Perf}{\text{Perf}}
\newcommand{\QCohF}{\mathcal{DF}}
\newcommand{\Mod}{\text{Mod}}
\newcommand{\CAlg}{\text{CAlg}}
\newcommand{\Sheaves}{\text{Sh}}
\newcommand{\Normal}{\mathcal{N}}
\newcommand{\Oh}{\mathcal{O}}
\newcommand{\E}{\mathcal{E}}
\newcommand{\F}{\mathcal{F}}
\newcommand{\C}{\mathcal{C}}
\newcommand{\BS}{\mathcal{B}}
\newcommand{\Line}{\mathcal{L}}
\newcommand{\Rees}{\mathcal{R}}
\DeclareMathOperator{\CoTan}{\mathcal{L}}

\newcommand{\Znumb}{\mathbf{Z}}
\newcommand{\Qnumb}{\mathbf{Q}}
\newcommand{\Nnumb}{\mathbf{N}}
\newcommand{\IntProjP}{\mathbf{P}}
\newcommand{\Unit}{\textbf{1}}
\newcommand{\DualSh}{\text{D}}
\newcommand{\CoDualSh}{\text{P}}
\newcommand{\Fp}{\mathbf{F}_{p}}
\DeclareMathOperator{\IntHom}{\underline{Hom}}
\DeclareMathOperator{\Hom}{\text{Hom}}
\DeclareMathOperator{\Fun}{\text{Fun}}
\newcommand{\Spf}{\text{Spf}}
\newcommand{\Spec}{\text{Spec}}
\newcommand{\Spa}{\text{Spa}}
\newcommand{\Spv}{\text{Spv}}
\newcommand{\SmSpec}{\text{SmSpec}}
\newcommand{\A}{\mathbf{A}}
\newcommand{\Disc}{\mathbf{D}}

\newcommand{\Th}{\mathbf{Th}}
\newcommand{\ProS}{\mathbf{P}}
\newcommand{\Bl}{\mathbf{Bl}}
\newcommand{\Coh}{\mathbf{C}}
\newcommand{\V}{\mathbf{V}}
\newcommand{\BV}{B\mathbf{V}}

\DeclareMathOperator{\GaHodge}{\mathbf{G}_{a}^{Hodge}}
\DeclareMathOperator{\GadR}{\mathbf{G}_{a}^{dR}}
\DeclareMathOperator{\Gasharp}{\mathbf{G}_{a}^{\#}}
\DeclareMathOperator{\GaPrism}{\mathbf{G}_{a}^{\Prism}}
\DeclareMathOperator{\GaPrismA}{(\mathbf{G}_{a}/A)^{\Prism}}
\DeclareMathOperator{\GaHT}{\mathbf{G}_{a}^{HT}}
\DeclareMathOperator{\GaNyg}{\mathbf{G}_{a}^{\Nyg}}
\DeclareMathOperator{\BGasharp}{B\mathbf{G}_{a}^{\#}}
\DeclareMathOperator{\Gm}{\mathbf{G}_{m}}
\DeclareMathOperator{\Ga}{\mathbf{G}_{a}}
\DeclareMathOperator{\Gln}{\mathbf{Gl}_{n}}
\DeclareMathOperator{\Grn}{\mathbf{Gr}_{n}}
\DeclareMathOperator{\BGm}{B\mathbf{G}_{m}}
\DeclareMathOperator{\Vect}{\mathbf{Vect}}
\DeclareMathOperator{\BGln}{B\mathbf{Gl}_{n}}
\DeclareMathOperator{\BHom}{B\text{Hom}}
\DeclareMathOperator{\AHadmodGm}{\widehat{\mathbf{A}^{1}}/\Gm}
\DeclareMathOperator{\AmodGm}{\mathbf{A}^{1}/\Gm}
\newcommand{\fSch}{\text{fSch}}
\newcommand{\Sm}{\text{Sm}}
\newcommand{\Top}{\text{Top}}
\newcommand{\SemPerfd}{\text{SemPerfd}}
\newcommand{\AnStack}{\text{AnStack}}
\newcommand{\MS}{\text{MS}}
\newcommand{\FGauge}{\text{F-Gauge}}
\newcommand{\Span}{\text{Span}}
\newcommand{\PrL}{\text{Pr}^{L}}

\newcommand{\Ainf}{\mathbb{A}_{\text{inf}}}
\newcommand{\Kernels}{\mathcal{K}}
\newcommand{\Et}{\text{\'Et}}
\newcommand{\Smoo}{\text{Sm}}

\title{A six-functor formalism for syntomic cohomology}
\author{Niklas Kipp}
\address{CNRS and Laboratoire de Math\'ematiques d’Orsay}
\email{niklas.kipp (at) universite-paris-saclay.fr}

\begin{document}

\maketitle

We construct a six-functor formalism for syntomic cohomology of $p$-adic formal schemes. 
This, in particular, generalizes Poincaré duality to general smooth morphisms. 
 
\tableofcontents

\section*{Introduction}

Grothendieck suggested that a cohomology theory should come together with its counterparts, homology, compactly 
supported cohomology, and Borel-Moore homology. Furthermore, these invariants are supposed to interact with one 
another in terms of certain maps and module structures, and all of this structure should have a fibered interpretation. 
Nowadays this structure has a precise interpretation and runs under the name \emph{six-functor formalism} (see for example 
\cite{HeyerMann6FF},\cite{Scholze6FF} and \cite{CisinskiTriangulatedCO}). In this text we construct this structure 
for syntomic cohomology of $p$-adic formal schemes as defined in \cite{Bhatt2019PrismsAP} \cite{Bhatt2018TopologicalHH} 
and \cite{BhattLurieAPC}. 

\begin{nonothm}[\textbf{A} see \ref{Main theorem}]
    There is a six-functor formalism 
       $$X\mapsto \FGauge_{\Prism}^{\Solid}(X)$$
    on the category of (derived) $p$-adic formal schemes satisfying the following:
       \begin{itemize}
        \item[(A)] Morphisms locally of finite type are 
                   $!$-able. 
        \item[(B)] \'Etale morphisms are cohomologically \'etale. 
        \item[(C)] Any finite composition of closed immersions and proper morphisms\footnote{In this note proper morphisms are locally 
                   of finite presentation.} is weakly cohomologically proper. 
        \item[(D)] The functor $(\_)^{\text{syn},\Solid}$ preserves étale covers and sends Zariski covers to 
                   open covers of analytic stacks. In particular the functors $\FGauge_{\Prism}^{\Solid}(\_)^{\ast}$ and 
                   $\FGauge_{\Prism}^{\Solid}(\_)^{!}$ are \'etale sheafs. 
        \item[(E)] It admits Tate twists. That is for any $p$-adic formal scheme $X$ 
                   the object 
                      $$\Oh_{X^{\text{syn}}}(-1):= cof(\Oh_{X^{\text{syn}}} \to f_{\ast}\Oh_{(\ProS^{1})^{\text{syn}}})$$
                   is $\otimes$-invertible and its inverse identifies with 
                      $$\Oh_{X^{\text{syn}}}(1)\iso \Oh_{X^{\text{syn}}}\{1\}[-2]$$
                   a shift of the Breuil-Kisin twist. 
        \item[(F)] Any smooth morphism is cohomologically smooth. Furthermore, for such a smooth 
                   morphism $f\from X\to S$, we have an identification 
                      $$f^{!}\Oh_{S^{\text{syn}}}:= \omega_{f} \iso \Oh_{X^{\text{syn}}}(d)\iso \Oh_{X^{\text{syn}}}\{d\}[-2d]$$
                    of the dualizing sheaf, where $d$ is the relative dimension 
                    of $f$. 
        \item[(G)] The dualizable\footnote{There is also a version containing the full category of prismatic F-Gauges \ref{Classsical syntomification theorem}} objects identify 
                      $$(\FGauge_{\Prism}^{\Solid}(\_))^{\text{dual}}\iso \Perf((\_)^{\text{syn}})$$
                    with perfect F-Gauges as defined in \cite{FGauges}[6.1].
                   In particular, there is a functorial identification 
                      $$R\Gamma_{\text{syn}}^{\text{BMS}}(\_,\Znumb_{p}(n))\iso \Hom_{(\_)^{\text{syn}}}(\Oh_{(\_)^{\text{syn}}},\Oh_{(\_)^{\text{syn}}}\{n\})$$
                   of the mapping spectrum with the syntomic cohomology of $p$-adic formal schemes as defined in \cite{BhattLurieAPC}.
       \end{itemize}
\end{nonothm}

One reason such a theorem is interesting is that, for any smooth morphism $f\from X\to S$ of 
relative dimension $d$, it produces a relative compactly supported cohomology 
   $$f_{!}\Oh_{X^{\text{syn}}} \in \FGauge_{\Prism}^{\Solid}(S)$$
which gives a Poincaré duality pairing. That is, we obtain the following corollary.

\begin{nonocor}
   The identification in (F) induces a trace map 
      $$tr \from f_{!}\Oh_{X^{\text{syn}}}\{d\}[-2d] \to \Oh_{S^{\text{syn}}}$$
    such that the composition 
\[\begin{tikzcd}
	{f_{\ast}\Oh_{X^{\text{syn}}}\otimes_{S}f_{!}\Oh_{X^{\text{syn}}}\{d\}[-2d]} & {f_{!}\Oh_{X^{\text{syn}}}\{d\}[-2d]} & {\Oh_{S^{\text{syn}}}}
	\arrow["\cup", from=1-1, to=1-2]
	\arrow["tr", from=1-2, to=1-3]
\end{tikzcd}\]
   is adjoint to an isomorphism 
      $$f_{\ast}\Oh_{X^{\text{syn}}} \iso (f_{!}\Oh_{X^{\text{syn}}}\{d\}[-2d])^{\vee}$$
   in $\FGauge_{\Prism}^{\Solid}(S)$. 
\end{nonocor}

Another reason this theorem is interesting lies in the nature of Syntomic cohomology, which we now recall very briefly.

One way to understand syntomic cohomology is, following Bhatt-Lurie \cite{BhattLurieAPC} \cite{BhattLuriePrismatisationofpadicformalschemes} and Drinfeld \cite{Drinfeld2020Prismatization}, 
via the stacky approach. That is, to a $p$-adic formal scheme $X$, one associates a stack $X^{\text{syn}}$, such that 
quasi-coherent sheaves on this stack recover syntomic cohomology. 

In order to obtain our theorem, we interpret these stacks in the world of analytic 
stacks introduced by Clausen-Scholze \cite{AStacksLecture}. 
That is to a $p$-adic formal scheme $X$ 
we now associate an analytic stack 
   $$X^{\text{syn},\Solid} \in \AnStack,$$
which we call the \emph{solid Syntomification} of $X$, such that its category of quasi-coherent 
sheaves recovers syntomic cohomology. This has (among others) two advantages: 

One is that on quasi-coherent sheaves of analytic stacks it is easier\footnote{Maybe possible would be the better wording here.} to 
define a well-behaved compactly supported cohomology. This is crucial to obtain the above theorem. 

Another advantage is that the category of analytic stacks is quite huge and contains a priori very 
different looking stacks. In particular, in this world, it is easier to find relations between different looking geometries. 

This latter advantage could be particularly useful for the syntomification, as it philosophically 
is supposed to parametrize well-behaved $p$-adic cohomology theories\footnote{See for example \cite{Lahoti2025CohomologyTI} for some precise statement.}
and \textbf{Theorem A} gives a way to apply this idea. Namely, given a map 
   $$\BS \to \Znumb_{p}^{\text{syn},\Solid}$$
of analytic stacks. Then it is not hard to see that assigning to a $p$-adic formal scheme $X$
the category 
  $$\QCoh(\BS\times_{\Znumb_{p}^{\text{syn},\Solid}}X^{\text{syn},\Solid})$$
can be extended to a six-functor formalism that satisfies (A)-(F). We record one example of this. 

\begin{nonothm}[\textbf{B} see \ref{Etale locus theorem}]
   There is a six-functor formalism 
      $$X \mapsto \QCoh(X_{\eta}^{\text{ét},\Solid})$$
   on $p$-adic formal schemes over $\Znumb_{p}^{\text{cyc}}$, which satisfies 
   analogous assertions to (A)-(F). Furthermore, the dualizable objects naturally identify 
      $$\QCoh((\_)_{\eta}^{\text{ét},\Solid})^{\text{dual}}\iso \QCoh_{\text{lisse}}^{\text{b}}((\_)_{\eta},\Znumb_{p})$$
   with the category of lisse étale sheaves on the generic fiber. 
\end{nonothm}

The way we prove the Poincaré duality is somewhat disconnected from these geometric ideas, in the 
sense that we give a general strategy to prove this duality in six-functor formalisms. 
This should be applicable in many other situations and is highly motivated by the work of 
Zavyalov \cite{ZavyalovPoincareDI}, Tang \cite{Tang2026TheG} \cite{TangGysin} and Annala-Hoyois-Iwasa 
\cite{AnnalaAlgebraicCob} \cite{AnnalaAtiyahDual}. That is, we explain the following theorem. 

\begin{nonothm}[\textbf{C} see \ref{dualizing complex theorem}]
    Given a six-functor formalism $\QCoh$ on the category of schemes, which satisfies the 
    following (and admits the following structures):
       \begin{itemize}
        \item All étale morphisms of schemes are cohomologically étale.
        \item All proper morphisms are (weakly) cohomologically proper.
        \item $\QCoh^{\ast}$ is a Nisnevich sheaf. 
        \item It admits an additive orientation. 
        \item It satisfies elementary blow-up excision. 
       \end{itemize}
    Then any smooth morphism $f\from X\to S$ of relative dimension $d$ will be cohomologically smooth and we have a 
    canonical identification 
       $$f^{!}(\Unit_{S}) \iso \Unit_{X}(d)$$
    of the dualizing sheaf with the Tate twist\footnote{Here the Tate twist is given by the homology of pointed $\ProS^{1}$. In the text we will switch to $(d)[-2d]$.}. 
\end{nonothm}

Here an additive orientation is given by a theory of first Chern classes together with a projective bundle formula isomorphism. 
This in practice reduces to a simple computation of cohomology of projective spaces. Also elementary 
blow-up excision reduces to a computation of cohomology of blow-ups of zero-sections in affine space and in 
practice should often even reduce to the projective bundle formula. 

In the text \textbf{Theorem C} is not only formulated for six-functor formalisms on the category 
of schemes. Instead, we tried to axiomatize what properties smooth morphisms need to satisfy to make the 
theorem work. In particular, this should logically be applicable to geometric settings like 
complex analytic spaces, rigid spaces or Berkovich spaces as well.

\subsection*{Organization of the Text}

In \ref{Remark on Smoothness} we prove \textbf{Theorem C} and give the necessary 
axiomatizations and constructions. 

In \ref{formal} we explain how to understand formal stacks in terms of analytic stacks. 
Most importantly, we give a criterion for properness of analytic stacks coming from 
formal stacks \ref{From formal to analytic prop}. 

In \ref{solid Syntomic} we recall some aspects of the stacky approach to 
syntomic cohomology and define the solid Syntomification. 

In \ref{theSIXFF} we collect the proof of \textbf{Theorem A}.

\subsection*{Notations and conventions}

We will use the following notations and conventions. 

\begin{itemize}
    \item We will use the term \emph{category} to mean what, often in the literature, is called an $\infty$-category.
    \item A ring will be an animated commutative ring.
    \item Many topoi in this text are "big topoi". That is, they are formally not topoi because of size issues. 
          We will systematically ignore this issue, as we will never use that they are topoi but just that 
          they practically behave like those in the following sense. To deal with this issue, one defines a $\kappa$-small 
          version of this topos for each uncountable strong limit cardinal $\kappa$ and then takes 
          the colimits over those $\kappa$ in categories. In particular, most arguments will happen in the $\kappa$-small subcategory 
          for some $\kappa$ and thus in some topos. 
    \item We will use the homological conventions to speak about (co)homology groups, $t$-structures and suspensions.  
\end{itemize}

\subsection*{Acknowledgments} A version of this paper was handed in as my PhD thesis, and first and most of all 
I would like to thank my advisor, Marc Hoyois, for (among other things) his incredible talent for teaching mathematics. 
I also would like to thank the mathematics department at the University of Regensburg and everybody who was involved 
in the amazing time I had there. Finally, I would like to thank Johannes Anschütz, Ko Aoki,
Tess Bouis, Denis-Charles Cisinski, Bastiaan Cnossen, Benjamin Dünzinger, Elden Elmanto, Markus Fuchs, Ryomei Iwasa, Han-Ung Kufner,
Sil Linskens, Deven Manam, Matthew Morrow, Vova Sosnilo, Sebastian Wolf and Bogdan Zavyalov for helpful discussion, which one 
way or another found their way into this paper. During the writing process, I was supported by the SFB 1085 "Higher Invariants", and this project has
received funding from the European Research Council (ERC) under the
European Union’s Horizon 2020 research and innovation program (grant
agreement No. 101001474).
\section{A remark on smoothness in six-functor formalisms}\label{Remark on Smoothness}

\subsection{Recollections on six-functor formalisms}

We recall some facts and definitions in the context of six-functor formalisms. 
More detailed discussions can be found in \cite{HeyerMann6FF} \cite{GaitsgoryDAG1} \cite{CisinskiTriangulatedCO} \cite{Scholze6FF}.

\begin{non}
   Consider a pair $(\C,\C_{E})$ where $\C$ is a category and $\C_{E}\subset \C$
   a wide subcategory satisfying the following: 
     \begin{itemize}
        \item[(a)] Morphisms in $\C_{E}$ are closed under pullbacks along morphisms in $\C$.
        \item[(b)] $\C_{E}$ admits pullbacks and the inclusion $\C_{E}\subset \C$ preserves those.
     \end{itemize}
   Such data is called a \emph{geometric setup} and given such data 
   we can construct the category 
      $$\Span(\C,\C_{E})$$
    which informally can be described as follows (see \cite{HeyerMann6FF}[2.2] for an honest 
    construction):
       \begin{itemize}
        \item Objects are given by the objects in $\C$.
        \item A morphism from $X$ to $Y$ is given by a span 
\[\begin{tikzcd}
	& Z \\
	X && Y
	\arrow[from=1-2, to=2-1]
	\arrow[from=1-2, to=2-3]
\end{tikzcd}\]
             where the right leg lives in $\C_{E}$. To compose such spans, one 
             takes fiber products of the inner cospan. 
       \end{itemize} 
    Assuming that $\C$ admits finite products, the category $\Span(\C,\C_{E})$ can be equipped 
    with a symmetric monoidal structure induced by the cartesian product in $\C$. 
\end{non}

The following definition is taken from \cite{HeyerMann6FF}.

\begin{defn}
    Given a geometric setup $(\C,\C_{E})$, such that $\C$ admits finite products. A \emph{six-functor formalism} 
    on $(\C,\C_{E})$ is a lax symmetric monoidal 
    functor 
       $$\QCoh \from \Span(\C,\C_{E}) \to \PrL.$$
    Where $\PrL$ denotes the category of presentable categories with colimit-preserving functors as morphisms and equipped with the Lurie-tensor product.  
\end{defn}

\begin{rem}
    Using \cite{HeyerMann6FF}[3.4] we can and will assume that $\C$ is a topos, such that $\QCoh^{\ast}$
    defines a sheaf on this topos and $\QCoh^{!}$ descends along effective epimorphisms 
    in $\C_{E}$.
\end{rem}

\begin{non}[\textbf{The category of Kernels}]
    Given a six-functor formalism and an object $X\in \C$, we obtain a restricted six-functor 
    formalism 
       $$\QCoh\from \Span(\C(X)) \to \PrL$$
    where we write $\Span(\C(X)):= \Span((C_{E})_{/X},(C_{E})_{/X})$.
    Furthermore, $\QCoh$ gives a commutative algebra object in 
    $\Fun(\Span(\C(X)),\PrL)$ where the later category is equipped with the day convolution 
    symmetric monoidal structure. We then write 
       $$\Kernels_{X} := \Mod_{\QCoh}(\Fun(\Span(\C(X)),\PrL))$$
    and call it the \emph{category of kernels} of $X$. 
    
    This can be interpreted as a $2$-category, where any object $X\in \C_{E}$
    give rise to an object $[X]\in \Kernels_{\ast}$, $1$-morphisms $[X]\to [Y]$
    are described by object in $\QCoh(X\times Y)$ and $2$-morphism between those are described by 
    morphisms in $\QCoh(X\times Y)$.
\end{non}

\begin{rem}
    In this text, the category of kernels will only be used implicitly. For a more detailed 
    description and construction, we refer to \cite{HeyerMann6FF}[4] \cite{Scholze6FF}[V].
\end{rem}

We will use the following notions \cite{HeyerMann6FF}[4] \cite{AnalyticDeRahm}[3.1.]. 

\begin{defn}
   Let $f\from X\to S$ be a morphism in $\C_{E}$ and $\E\in \QCoh(X)$.
     \begin{itemize}
        \item[(a)] $\E$ is called \emph{f-suave}, if $\E\from [X]\to [S]$ is left adjoint in $\Kernels_{S}$.
        \item[(b)] $\E$ is called \emph{f-prim}, if $\E\from [S]\to [X]$ is left adjoint in $\Kernels_{S}$.
        \item[(c)] $f$ is called \emph{cohomologically smooth}, if $\Unit_{X}$ is $f$-suave and $\DualSh_{f}(\Unit_{X})=f^{!}(\Unit_{S})$
                   is $\otimes$-invertible.
        \item[(d)] $f$ is called \emph{cohomologically co-smooth}, if $\Unit_{X}$ is $f$-prim and $p_{2,\ast}\Delta_{!}(\Unit_{X})=\CoDualSh_{f}(\Unit_{X})$
                   is $\otimes$-invertible.
        \item[(e)] $f$ is called \emph{cohomologically étale}, if it is $n$-truncated for some finite $n$, $\Unit_{X}$ is $f$-suave and the 
                   diagonal $\Delta_{f}$ is cohomologically étale or an isomorphism.  
        \item[(f)] $f$ is \emph{weakly cohomologically proper}, if $\Unit_{X}$ is $\Delta_{f}$-prim, there exist some (non canonical) 
                   isomorphism $\CoDualSh_{\Delta_{f}}(\Unit_{X})\iso \Unit_{X}$ and $\Unit_{X}$ is $f$-prim. 
     \end{itemize} 
\end{defn}

\begin{rem}
    All of the above notions admit dual versions, which will not appear in this text. 
    In particular, all "étale" morphisms we will consider will be truncated. In contrast, almost no "proper" morphisms we will 
    consider will be truncated. 
\end{rem}

\begin{rem}
    All the above notions are stable under base change and composition. 
\end{rem}

We will use the following notions of coverings. Recall the notion of descendable 
algebra from \cite{MathewGaloisgroupsofstablehomotopy}[3.3.].

\begin{defn}
   Consider a morphism $f\from X\to S$ in $\C_{E}$. Then we will say:
      \begin{itemize}
        \item[(a)] $f$ is a smooth cover, if it is cohomologically smooth and $f^{\ast}$
                   is conservative. 
        \item[(b)] $f$ is a descendable cover, if it is cohomologically co-smooth and 
                     $$f_{\ast}\Unit_{X} \in \QCoh(S)$$
                   is descendable.
        \item[(c)] $f$ is a proper descendable cover, if it is a descendable cover and weakly cohomologically proper.
      \end{itemize} 
\end{defn}

\begin{rem}
    For all the above covers, one gets that $\QCoh^{\ast}$ as well as 
    $\QCoh^{!}$ descend along them \cite{HeyerMann6FF}[4.7.1, 4.7.4].  
\end{rem}

To end the section, we record the following stability properties.

\begin{prop}\label{stability properties of covers in SIXF}
    We have the following.
       \begin{itemize}
        \item[(a)] Being cohomologically smooth is $\QCoh^{\ast}$-local on the target and smooth 
                   local on the source. 
        \item[(b)] Being cohomologically étale is $\QCoh^{\ast}$-local on the target and étale
                   local on the source. 
        \item[(c)] Being weakly cohomologically co-smooth is $\QCoh^{\ast}$-local on the target and 
                   proper descendable local on the source.
        \item[(d)] Being a smooth cover is stable under composition, base change, and $\QCoh^{\ast}$-local on the target. 
        \item[(e)] Being a descendable cover is stable under composition, base change, and descendable local on the target. 
       \end{itemize}
\end{prop}

\begin{proof}
    Assertion $(a)$ is explained in \cite{HeyerMann6FF}[4.4.8+4.5.8(i)] (see also \cite{HeyerMann6FF}[4.5.12] and \cite{HA}[4.4.1.11]).
    Assertion $(b)$ is \cite{HeyerMann6FF}[4.6.3](ii).  
    Assertion $(c)$ is \cite{AnalyticDeRahm}[3.1.14] or \cite{HeyerMann6FF}[4.4.8] and \cite{HeyerMann6FF}[4.7.5](ii).
    The other assertions are now clear if one has \cite{MathewGaloisgroupsofstablehomotopy}[3.24] in mind.
\end{proof}

\begin{rem}\label{weakly proper from diagonal and cosmooth}
    Note that, if we, for a co-smooth map $f\from X\to S$, know that the diagonal is weakly cohomologically
    proper. Then also $f$ is weakly cohomologically proper.
\end{rem}

\begin{rem}\label{weakly proper from decomposition}
    If the six-functor formalism is constructed from a suitable decomposition of $E$ \cite{HeyerMann6FF}[3.3.2]. 
    That is, the construction comes with two classes $(I,P)$ of étale and proper maps, which are cohomologically étale 
    resp. cohomologically proper by construction. Then if a map $f\from X\to S$ locally on $S$
    for a $\QCoh^{\ast}$-cover lies in $P$, it is weakly cohomologically proper.
    This is explained in \cite{AnalyticDeRahm}[3.1.21].
\end{rem}

\subsection{Additively oriented six-functor formalisms}
In what follows, we will incorporate some geometry into our six-functor formalism. 
Most importantly, we implement the so-called Projective bundle formula.
For this, we fix a six-functor formalism $\QCoh$.

\begin{defn}
    We will say a geometric setup $(\C,\C_{E})$ is \emph{geometrized}, if it comes with 
    functor 
       $$\Sm^{\text{sep}}_{B} \to \C_{E}$$
    from the category of separated smooth schemes over some (possibly derived) scheme $B$, which 
    preserves those finite limits, which are also finite limits in the category of all (derived) schemes.

    Furthermore, a six-functor formalism $\QCoh$ on a geometrized geometric setup 
    is called \emph{geometric}, if:
       \begin{itemize}
        \item For any object $S \in \C$ the category $\QCoh(S)$ is stable.
        \item Any \'etale morphism in $\Sm_{B}$ is cohomologically \'etale with respect 
              to $\QCoh$.
        \item Any proper morphism in $\Sm_{B}$ is weakly cohomologically proper with respect to 
              $\QCoh$.
        \item  Any Nisnevich covering gets sent to a $\QCoh^{\ast}$-cover. 
       \end{itemize}
\end{defn}

\begin{rem}
    We will normally not write the functor $\Sm_{B} \to \C$. That is, for example, we 
    will write $\A^{n}$ for the $n$-dimensional affine space seen as an object 
    in $\C$ via this functor. We can also define the affine space over an 
    arbitrary object $X\in \C$ via base change.
\end{rem}

\begin{rem}
    The Nisnevich topology is famously used in motivic homotopy theory and sits 
    in between the Zariski and the \'etale topology. The \'etale topology would be enough for all applications in this text. We just chose this definition for the 
    sake of generality. 
\end{rem}

For the rest of the section, we assume that the fixed six-functor formalism 
is geometric\footnote{This is of course structure and not a property.}.

\begin{defn}
    Let us write $f\from \ProS^{1}\to B$ for the projection. 
    We say $\QCoh$ admits a \emph{Tate twists}, if the 
    object 
       $$\Unit_{B}(-1) := cof(\Unit_{B} \to f_{\ast}\Unit_{\ProS^{1}_{B}})$$
    is $\otimes$-invertible. In that case, we will write 
       $$\Unit_{B}(1)$$
    for its $\otimes$-inverse and call it the \emph{Tate twist}.
\end{defn}

\begin{rem}
    Using proper base change, the Tate twist defines a 
    cartesian section for $\QCoh^{\ast}$. That is, we can define 
       $$\Unit_{S}(1)$$
    for an arbitrary $S\in \C$, either by the above construction over $S$ or via base change. 
\end{rem}

In the following, we will use the topology on $\C$ generated by families of cohomologically étale monomorphisms 
which jointly form a smooth cover (with respect to $\QCoh$) and refer to it as the \emph{open topology}.
The following definition is taken from \cite{ZavyalovPoincareDI}[5.2.4]. 

\begin{defn}
    We will say that $\QCoh$ admits a \emph{theory of first chern classes}, if
    it admits Tate twists, and it comes with a natural transformation 
       $$c_{1}\from R\Gamma_{\text{open}}(\_,\Gm)[-1] \to \Hom(\Unit_{(\_)},\Unit_{(\_)}(1))$$
    of sheaves of spectra on $\C$.
\end{defn}

\begin{rem}
    Note that there is a natural transformation 
       $$R\Gamma_{\text{Zar}}(\_,\Gm)[-1] \to R\Gamma_{\text{open}}(\_,\Gm)[1]$$
    of sheaves of spectra on $\Sm^{\text{sep}}_{B}$. Furthermore, $\pi_{0}$ of 
    the left-hand side computes the Picard group. In particular, any line 
    bundle $\Line$ on a smooth separated scheme $S$ over $B$ gives rise to a map 
       $$c_{1}(\Line)\from \Unit_{S} \to \Unit_{S}(1).$$
    In practice, $\pi_{0}$ of the right-hand side will essentially compute 
    the group of line bundles on an object $S\in \C$. 
\end{rem}

\begin{rem}
    The assignment of first Chern classes is compatible with base change 
    and we have the formula 
       $$c_{1}(\Line_{1}\otimes \Line_{2})\iso c_{1}(\Line_{1})+c_{1}(\Line_{2}).$$
\end{rem}

\begin{constr}[\textbf{The Projective bundle formula}]\label{Projective bundle formula morphism Construction}
    Consider a morphism $f\from X\to S$ of separated smooth schemes over $B$ 
    and a line bundle $\Line$ on $X$. Then, by adjunction, the morphism 
    $c_{1}(\Line)$ induces a morphism 
       $$c_{1}(\Line) \from \Unit_{S} \to f_{\ast}\Unit_{X}(1)$$
    which we will denote the same way. Furthermore, taking $\otimes$-powers iteratively 
    of this morphism with itself, we obtain a map 
       $$c_{1}(\Line)^{d} \from \Unit_{S} \to f_{\ast}\Unit_{X}(d).$$
    
    Now, if we specialize to $f\from \ProS_{S}(\F) \to S$ being the projection 
    from the projective space associated to a vector bundle $\F$ of rank $d+1$ on $S$, we can use 
    these constructions to obtain a morphism 
        $$\sum_{i=0}^{d} c_{1}(\Oh(1))^{i}(d-i)\from \bigoplus_{i=0}^{d}\Unit_{S}(d-i) \to f_{\ast}\Unit_{\ProS_{S}(\F)}(d)$$
    where $\Oh(1)$ denotes the universal line bundle and $c_{1}(\Oh(1))^{0}$ by 
    adjunction corresponds to the identity. 
\end{constr}

The following definition is taken from \cite{ZavyalovPoincareDI}[5.2.8]. 

\begin{defn}
    We will say that a theory of first Chern classes for $\QCoh$ is 
    an \emph{additive orientation}, if for each $d\ge 1$ the map 
       $$\sum_{i=0}^{d} c_{1}(\Oh(1))^{i}(d-i)\from \bigoplus_{i=0}^{d}\Unit_{B}(d-i) \to f_{\ast}\Unit_{\ProS^{d}_{B}}(d)$$
    is an isomorphism. 
\end{defn}

\begin{rem}
    Note that the Projective bundle formula isomorphism base changes to its 
    counterpart over an arbitrary object $S\in \C$. In particular, the Projective Bundle formula holds over any object in $\C$ if we have an additive orientation. 
    As we can check isomorphisms locally, we also obtain a Projective Bundle 
    formula for an arbitrary projective bundle over a smooth separated scheme 
    over $S$. 
\end{rem}

\begin{rem}
    For any object $S\in \C$ and any sheaf $\E\in \QCoh(S)$ we obtain an 
    isomorphism
       $$\sum_{i=0}^{d} c_{1}(\Oh(1))^{i}(d-i)\from \bigoplus_{i=0}^{d}\E(d-i) \to f_{\ast}f^{\ast}\E(d)$$
    by tensoring the sheaf on the Projective Bundle formula isomorphism.
\end{rem}

\subsection{Smooth morphisms}

In this subsection, we will axiomatize what we want a class of smooth morphisms 
to satisfy in order to compute their dualizing sheaf. For this, we fix a geometric 
six-functor formalism $\QCoh$. 

Recall that we referred to the topology generated by families of cohomologically 
étale monomorphisms, which generated an étale covering, the \emph{open topology}. 

\begin{defn}
    We will call a morphism $Z\to X$ in $\C$ a \emph{regular immersion}, if it locally 
    on the target in the open topology sits in a Cartesian square 
\[\begin{tikzcd}
	Z & X \\
	{\{0\}} & {\A^{1}_{B}}
	\arrow[from=1-1, to=1-2]
	\arrow[from=1-1, to=2-1]
	\arrow[from=1-2, to=2-2]
	\arrow[from=2-1, to=2-2].
\end{tikzcd}\]
\end{defn}

\begin{rem}
    We chose the wording regular immersion, as the above Cartesian square should 
    be thought of as a derived Cartesian square. 
\end{rem}

\begin{defn}\label{definition of geometrically smooth morphisms}
    Given two subcategories $(\Et,\Smoo)$ of $\C$, we will refer 
    to the morphisms in $\Et$ as geometrically étale and to morphisms in $\Smoo$ as geometrically smooth, if they satisfy 
    the following: 
       \begin{itemize}
        \item[(a)] Both classes are stable under base change along arbitrary morphisms between objects in $\Smoo$ and $\Et \subset \Smoo$.
        \item[(b)] Any morphism in $\Et$ is cohomologically étale.
        \item[(c)] $\Sm_{B}^{\text{sep}} \subset \Smoo$.
        \item[(d)] Any morphism in $\Smoo$, locally on source and target in the open topology, factors as 
\[\begin{tikzcd}
	X & {\A^{n}_{S}} & S
	\arrow["e", from=1-1, to=1-2]
	\arrow["pr", from=1-2, to=1-3]
\end{tikzcd}\]
                   where $e \in \Et$ and $pr$ denotes the canonical projection.
        \item[(e)] Given a morphism $f\from X\to S$ in $\Smoo$,
                    then any section $s\from S\to X$ factors as 
\[\begin{tikzcd}
	S & U & X
	\arrow["i", from=1-1, to=1-2]
	\arrow["j", from=1-2, to=1-3].
\end{tikzcd}\]
                   where $j$ is an \'etale monomorphism and $i$ a regular immersion.  
        \item[(f)] For any commutative triangle 
\[\begin{tikzcd}
	Z & X \\
	& S
	\arrow["i", from=1-1, to=1-2]
	\arrow["g"', from=1-1, to=2-2]
	\arrow["f", from=1-2, to=2-2]
\end{tikzcd}\]
                    where $g$ and $f$ are geometrically smooth and $i$ a regular immersion. Locally 
                    on $X$ in the open topology, we can find cartesian squares 
\[\begin{tikzcd}
	Z & Z & Z \\
	X & U & {\A^{n}_{Z}}
	\arrow["i"', from=1-1, to=2-1]
	\arrow[equal, from=1-2, to=1-1]
	\arrow[equal, from=1-2, to=1-3]
	\arrow["t"', from=1-2, to=2-2]
	\arrow["0", from=1-3, to=2-3]
	\arrow[from=2-2, to=2-1]
	\arrow[from=2-2, to=2-3]
\end{tikzcd}\]
                   with the horizontal morphisms \'etale. 
        \item[(g)] For any commutative triangle 
\[\begin{tikzcd}
	Z & X \\
	& S
	\arrow["i", from=1-1, to=1-2]
	\arrow["g"', from=1-1, to=2-2]
	\arrow["f", from=1-2, to=2-2]
\end{tikzcd}\]
                 where $g$ and $f$ are geometrically smooth and $i$ is a regular immersion. There exists a geometrically smooth
                 morphism $\Bl_{Z}(X) \to X$, which restricted to the open locus on $X$, where we have a cartesian square 
\[\begin{tikzcd}
	Z & X \\
	{\{0\}} & {\A^{n}_{B}}
	\arrow[from=1-1, to=1-2]
	\arrow[from=1-1, to=2-1]
	\arrow[from=1-2, to=2-2]
	\arrow[from=2-1, to=2-2]
\end{tikzcd}\]
                 is given by $X\times_{\A^{n}_{B}}\Bl_{\{0\}}(\A^{n})$. Here we write $\Bl_{\{0\}}(\A^{n}_{B})$
                 for the Blow-up from derived algebraic geometry \cite{KhanVirtualCD} \cite{Tang2026TheG}[A].
            \end{itemize}
\end{defn}

Let us make some remarks on this definition. 

\begin{rem}
   Assertion $(f)$ in the definition of geometrically smooth morphisms, in practice, 
   follows from the Jacobian criterion together with the fact that the diagonal of an \'etale 
   morphism is an open embedding and that the image of the complement of an 
   open along a Zariski closed immersion has an open complement \cite{MorelA!homotopy}[Section 3. Lemma 2.28].
\end{rem}

\begin{rem}\label{Nisnevich descent remark}
   Given a Zariski closed immersion $Z\to X$ then locally on $X$, $Z$ admits a 
   complement by pulling back the complement of the zero section. As these are monomorphisms 
   we can glue those local complements to a global complement $U \subset X$. Furthermore for any 
   etale morphism $V\to X$, such that the square 
\[\begin{tikzcd}
	Z & V \\
	Z & X
	\arrow[from=1-1, to=1-2]
	\arrow[equal, from=1-1, to=2-1]
	\arrow[from=1-2, to=2-2]
	\arrow[from=2-1, to=2-2]
\end{tikzcd}\]
   is Cartesian, we see that the pair 
      $$V\amalg U \to X$$
   is a cohomologically étale covering, since this holds locally. Note also that as $U\subset X$ is a 
   monomorphism, satisfying descent for this d\'et cover, is equivalent to sending the 
   square 
\[\begin{tikzcd}
	{U\times_{X}V} & V \\
	U & X
	\arrow[from=1-1, to=1-2]
	\arrow[from=1-1, to=2-1]
	\arrow[from=1-2, to=2-2]
	\arrow[from=2-1, to=2-2]
\end{tikzcd}\]
   to a pullback. 
\end{rem}

\begin{rem}
    Condition $(h)$ is a property.
\end{rem}

In order to prove the cohomologically smoothness of a geometrically smooth morphism
we will use the following criterion from \cite{ZavyalovPoincareDI}[3.2.4.]. 

Consider a morphism $f\from X\to S$ in $\C_{E}$. Then we consider the commutative
diagram 
\[\begin{tikzcd}
	X \\
	& {X\times_{S}X} & X \\
	& X & S
	\arrow["\Delta", from=1-1, to=2-2]
	\arrow["id", bend left, from=1-1, to=2-3]
	\arrow["id"', bend right, from=1-1, to=3-2]
	\arrow["{p_{2}}", from=2-2, to=2-3]
	\arrow["{p_{1}}"', from=2-2, to=3-2]
	\arrow["f", from=2-3, to=3-3]
	\arrow["f"', from=3-2, to=3-3].
\end{tikzcd}\]

\begin{defn}
    A \emph{trace-cycle theory} on a morphism $f\from X\to S$ in $\C_{E}$ consists 
    of a triple $(\omega_{f},\text{tr}_{f},\text{cl}_{\Delta})$ of:
      \begin{itemize}
        \item[(a)] A $\otimes$-invertible object $\omega_{f}$ in $\QCoh(X)$.
        \item[(b)] A \emph{trace morphism} $\text{tr}_{f} \from f_{!}\omega_{f} \to \Unit_{S}$
                   in $\QCoh(S)$.
        \item[(c)] A \emph{cycle} morphism $\text{cl}_{\Delta}\from \Delta_{!}\Unit_{X}\to (p_{2})^{\ast}\omega_{f}$
                   in $\QCoh(X\times_{S}X)$.
      \end{itemize}
    Such that the following hold:
       \begin{itemize}
        \item[(1)] The composition 
\[\begin{tikzcd}
	{\Unit_{X}} & {(p_{1})_{!}\Delta_{!}\Unit_{X}} & {(p_{1})_{!}(p_{2})^{\ast}\omega_{f}} & {\Unit_{X}}
	\arrow["\iso", from=1-1, to=1-2]
	\arrow["{(p_{1})_{!}(\text{cl}_{\Delta})}", from=1-2, to=1-3]
	\arrow["{\text{tr}_{p_{1}}}", from=1-3, to=1-4]
\end{tikzcd}\]
                  is the identity. Where we write $\text{tr}_{p_{1}}\iso f^{\ast}(\text{tr}_{f})$.
        \item[(2)] The composition 
\[\begin{tikzcd}
	{\omega_{f}} & {(p_{2})_{!}(p_{1}^{\ast}\omega_{f}\otimes\Delta_{!}\Unit_{X})} && {(p_{2})_{!}(p_{1}^{\ast}\omega_{f}\otimes p_{2}^{\ast}\omega_{f})\iso(p_{2})_{!}p_{2}^{\ast}\omega_{f}\otimes\omega_{f}} & {id\otimes\omega_{f}\iso\omega_{f}}
	\arrow["\iso", from=1-1, to=1-2]
	\arrow["{(p_{2})_{!}(id\otimes\text{cl}_{\Delta})}", from=1-2, to=1-4]
	\arrow["{\text{tr}_{f}\otimes id}", from=1-4, to=1-5]
\end{tikzcd}\]
                  is the identity. 
       \end{itemize}
\end{defn}

\begin{rem}
    The definition unwinds what it means to have a unit and counit of an adjunction 
       $$\Unit_{X}\dashv \omega_{f}$$
    in $\Kernels_{S}$.
\end{rem}

Then section $3$ in \cite{ZavyalovPoincareDI} implies the following. 

\begin{thm}[\textbf{Zavyalov}]\label{Bogdans criterium}
    Consider a geometrized geometric setup $(\C,\C_{E})$ and a geometric 
    six-functor formalism $\QCoh$ on it. Then a morphism $f\from X\to S$ in $\C_{E}$ is cohomologically 
    smooth, if and only if it admits a trace-cycle theory, and in that case, we can 
    compute the dualizing sheaf as 
       $$f^{!}\Unit_{S}\iso \omega_{f}.$$
    Furthermore, for any class of geometrically smooth morphisms, the following 
    are equivalent: 
       \begin{itemize}
        \item[(a)] Any geometrically smooth morphism is cohomologically smooth.
        \item[(b)] The morphism $\A^{1} \to B$ is cohomologically smooth.
        \item[(c)] The morphism $\ProS^{1}\to B$ is cohomologically smooth. 
       \end{itemize}
\end{thm}

\begin{proof}
    Using conditions $(b)$ and $(d)$ this follows as in \cite{ZavyalovPoincareDI}[3.3.3.]. 
\end{proof}

\subsection{Vector bundles and Thom spaces}

In this section, we will define the Thom space associated with a vector bundle and 
compute its cohomology. For this, we fix an additively oriented six-functor formalism $\QCoh$
on $\C$ and recall that we have a functor 
   $$\Sm^{\text{sep}}_{B} \to \C$$
which preserves those finite limits, which are also finite limits in the category of all schemes. 

\begin{non}
    For a morphism $f\from X\to S$ in $\C$ and an object $\E_{S}\in \QCoh(S)$, we will write 
       $$\Coh^{\ast}(X,\E_{S}) := f_{\ast}f^{\ast}\E_{S} \in \QCoh(S)$$
    for its cohomology. Furthermore if $f$ lives in $\C_{E}$, we will write 
       $$\Coh_{\ast}(X,\E_{S}):= f_{!}f^{!}\E_{S} \in \QCoh(S)$$
    for its homology. We will omit the coefficients in case $\E_{S}\iso \Unit_{S}$.
\end{non}

We now also fix a class of geometrically smooth and geometrically \'etale morphisms.

\begin{non}
    Given an object $S\in \C$, we will write 
       $$(\Sm_{S}^{\text{geom}})_{\text{open}}$$
    for the topos generated by geometrically smooth objects over $S$ equipped 
    with the topology generated by families of cohomologically \'etale monomorphisms, which 
    jointly induce a $\QCoh^{\ast}$-covering. Clearly, cohomology and homology induce functors on this 
    topos.  
\end{non}

We can now consider the following constructions.

\begin{non}[\textbf{Vector bundles}]
    Given an object $S\in \C$, we will write  
       $$\BGln_{S} \in (\Sm_{S}^{\text{geom}})_{\text{open}}$$
    for the stack of rank $n$ vector bundles on $S$ and 
      $$\Vect_{S} := \coprod_{n}\BGln_{S} \in (\Sm_{S}^{\text{geom}})_{\text{open}}$$
    for the stack of all vector bundles on $S$. 

    That is a map $\langle \F \rangle \from X\to \BGln_{S}$ corresponds to a $\Gln$-torsor 
       $$\V_{X}(\F) \to X$$
    which trivializes locally in the open topology. 
\end{non}

\begin{non}[\textbf{Projective bundles}]
   Given a rank $\ge 1$ vector bundle $\V_{X}(\F)\to X$, we have a zero section $0\from X\to \V_{X}(\F)$. 
   Furthermore using the diagonal map $\Gm \to \Gln$, we obtain an $\Gm_{X}$-action on 
   $\V_{X}(\F)$, which restricts to $\V_{X}(\F)-\{0\}$. 

   From this we define the \emph{projective bundle} associated to $\V_{X}(\F)$ to be 
      $$\ProS_{X}(\F) := (\V_{X}(\F)-\{0\})/\Gm \in (\Sm_{S}^{\text{geom}})_{\text{open}}.$$
   
   Note that there is a canonical map $\ProS_{X}(\F) \to \BGm$, which defines 
   a line bundle $\Oh(1)$ on the projective bundle. This line bundle has a first Chern class. 
\end{non}

\begin{rem}
   Any projective bundle is weakly cohomologically proper, as locally on the base, this is true 
   by assumption. 
\end{rem}

\begin{non}[\textbf{Thom spaces}]
   Given a rank $\ge 1$ vector bundle $\V_{X}(\F)$, we define the its \emph{Thom space} to be 
   the cofibre 
      $$\Th_{X}(\F):= cof(\ProS_{X}(\F) \to \ProS_{X}(\F\oplus \Oh_{X})) \in (\Sm_{S}^{\text{geom}})_{\text{open}}.$$
\end{non}

We now have the following classical observation. 

\begin{prop}\label{thom spaces are trivialized}
   Given an object $S\in \C$ together with a vector bundle $\V_{S}(\F)\to S$ of rank $n\ge 1$.
   Then we have a canonical isomorphism 
      $$t(\F)\from \Unit_{S}(-n) \isoto \Coh^{\ast}(\Th_{S}(\F))$$ 
   in $\QCoh(S)$. In particular, the cohomology of those Thom spaces is $\otimes$-invertible.
\end{prop}

\begin{proof}
   See for example \cite{AnnalaAlgebraicCob}[6]. We consider the following diagram, which commutes by naturality 
   of first Chern classes.
\[\begin{tikzcd}
	{\Coh^{\ast}(\ProS_{S}(\F\oplus \Oh_{S}))} & {\Coh^{\ast}(\ProS_{S}(\F))} \\
	{\bigoplus_{i=0}^{n}\Unit_{S}(-i)} & {\bigoplus_{i=0}^{n-1}\Unit_{S}(-i)}
	\arrow[from=1-1, to=1-2]
	\arrow["\iso", from=2-1, to=1-1]
	\arrow["\iso"', from=2-2, to=1-2]
	\arrow[from=2-2, to=2-1].
\end{tikzcd}\]
   Here, the vertical isomorphisms come from the Projective bundle formula.
   Thus, we obtain an isomorphism between the cofibre of the lower horizontal and the fiber of the 
   upper horizontal map. 
\end{proof}

\subsection{Tang's construction of cycle classes}

To construct a cycle class theory for a geometrically smooth morphism, 
we will make use of a construction of Longke Tang \cite{Tang2026TheG}. 

Let us now fix a geometric six-functor formalism $\QCoh$, which comes equipped with an 
additive orientation. Furthermore, we fix a class of geometrically smooth and geometrically 
étale morphisms.

\begin{non}\label{blowup squares}
   For $S\in \Smoo$ we write $\Sm^{\text{geom}}_{S}$ for the category of geometrically
   smooth morphisms over $S$. Furthermore, we will call a blow-up square 
\[\begin{tikzcd}
	E & {\Bl_{Z}(X)} \\
	Z & X
	\arrow[from=1-1, to=1-2]
	\arrow[from=1-1, to=2-1]
	\arrow["p", from=1-2, to=2-2]
	\arrow["i"', from=2-1, to=2-2]
	\arrow["f", from=1-1, to=2-2]
\end{tikzcd}\]
   associated to a regular closed immersion $i\from Z\to X$ between objects in $\Sm^{\text{geom}}_{S}$ a \emph{smooth 
   blowup square}. 
\end{non}

We will also use the following minor assumption on our six-functor formalism. 

\begin{defn}\label{blowup excision definition}
	We will say that $\QCoh$-cohomology satisfies \emph{blow-up excision}
	if for every $n\ge 1$ the blow-up square 
\[\begin{tikzcd}
	\ProS_{B}^{n-1} & \V_{\ProS^{n-1}_{B}}(\Oh(1)) \\
	B & \A^{n}_{B}
	\arrow[from=1-1, to=1-2]
	\arrow[from=1-1, to=2-1]
	\arrow["p", from=1-2, to=2-2]
	\arrow["0"', from=2-1, to=2-2]
	\arrow["f", from=1-1, to=2-2]
\end{tikzcd}\]
    induces a (co)cartesian square
\[\begin{tikzcd}
	{\Unit_{\A^{n}_{B}}} & {0_{\ast}0^{\ast}\Unit_{\A^{n}_{B}}} \\
	{p_{\ast}p^{\ast}\Unit_{\A^{n}_{B}}} & {f_{\ast}f^{\ast}\Unit_{\A^{n}_{B}}}
	\arrow[from=1-1, to=1-2]
	\arrow[from=1-1, to=2-1]
	\arrow[from=1-2, to=2-2]
	\arrow[from=2-1, to=2-2]
\end{tikzcd}\]
   in $\QCoh(\A^{n}_{B})$. 
\end{defn}

From this definition, we can obtain the following general form of blow-up excision.

\begin{prop}\label{blowup excision prop}
	Consider an object $S\in \Sm$. Then any blow-up square in $\Sm_{S}^{\text{geom}}$
	as in \ref{blowup squares} induces a (co)cartesian square 
\[\begin{tikzcd}
	{\Unit_{X}} & {i_{\ast}i^{\ast}\Unit_{X}} \\
	{p_{\ast}p^{\ast}\Unit_{X}} & {f_{\ast}f^{\ast}\Unit_{X}}
	\arrow[from=1-1, to=1-2]
	\arrow[from=1-1, to=2-1]
	\arrow[from=1-2, to=2-2]
	\arrow[from=2-1, to=2-2]
\end{tikzcd}\]
   in $\QCoh(X)$.
\end{prop}

\begin{proof}
	By pulling back along the structure map (using proper base change), we can deduce the 
	assertion for blow-up squares of the form 
\[\begin{tikzcd}
	\ProS_{Z}^{n-1} & \V_{\ProS^{n-1}_{Z}}(\Oh(1)) \\
	Z & \A^{n}_{Z}
	\arrow[from=1-1, to=1-2]
	\arrow[from=1-1, to=2-1]
	\arrow["p", from=1-2, to=2-2]
	\arrow["0"', from=2-1, to=2-2]
	\arrow["f", from=1-1, to=2-2]
\end{tikzcd}\]
   for general $Z\in \Sm$. For a general smooth blow-up square, we can use assertion $(f)$ in the 
   definition of geometrically smooth morphisms to Zariski locally on $X$ find cartesian squares 
\[\begin{tikzcd}
	Z & Z & Z \\
	X & U & {\A^{n}_{Z}}
	\arrow["i"', from=1-1, to=2-1]
	\arrow[from=1-2, to=1-1]
	\arrow[from=1-2, to=1-3]
	\arrow[from=1-2, to=2-2]
	\arrow["0", from=1-3, to=2-3]
	\arrow["j", from=2-2, to=2-1]
	\arrow["e"', from=2-2, to=2-3].
\end{tikzcd}\]
   Let us write $Q_{X}, Q_{U}$ and $Q_{\A^{n}_{Z}}$ for the cohomology squares induced by the corresponding blow-up squares. 
   Then by open descent, to check that $Q_{X}$ is cartesian, it suffices to see that 
      $$j^{\ast}(Q_{X})\iso Q_{U} \iso e^{\ast}(Q_{\A^{n}_{Z}})$$
	is cartesian. But this we did see in the first half of the argument. 
\end{proof}

\begin{rem}
	If $\QCoh$-cohomology over $S$ satisfies blow-up excision, by taking duals, one sees that 
	also the square 
\[\begin{tikzcd}
	{f_{!}f^{!}\Unit_{X}} & {p_{!}p^{!}\Unit_{X}} \\
	{i_{!}i^{!}\Unit_{X}} & {\Unit_{X}}
	\arrow[from=1-1, to=1-2]
	\arrow[from=1-1, to=2-1]
	\arrow[from=1-2, to=2-2]
	\arrow[from=2-1, to=2-2]
\end{tikzcd}\]
   associated to a blow-up square is (co)cartesian in $\QCoh(X)$ (note that all maps appearing are weakly 
   cohomologically proper).  
\end{rem}

\begin{rem}
	The author does not know how to formally deduce blowup-excision from the projective bundle formula. 
	In practice, though, it will normally be possible to make a reduction of this type. Essentially, because the following 
	is true. 
\end{rem}

\begin{prop}\label{Elementary blowup corollary}
   For any $S \in \Smoo$ and any object $\E_{S} \in \QCoh(S)$, the functor 
       $$\Coh_{\ast}(\_,\E_{S})\from\Sm^{\text{geom}}_{S} \to \QCoh(S)$$
    as well as the functor 
       $$\Coh^{\ast}(\_,\E_{S})\from (\Sm^{\text{geom}}_{S})^{\text{op}} \to \QCoh(S)$$
    sent smooth blow-up squares to pushout and pullback squares.
\end{prop}

\begin{proof}
   We first prove the claim for cohomology. As we assume that cohomology is a 
   Nisnevich sheaf, by \cite{AnnalaAlgebraicCob}[2.2] it suffices to check that for any geometrically smooth $X$ 
   a Blow-up squares of the form 
\[\begin{tikzcd}
	{\ProS_{X}^{d-1}} & {\V_{\ProS^{d-1}_{X}}(\Oh(1))} \\
	X & {\A^{d}_{X}}
	\arrow[from=1-1, to=1-2]
	\arrow[from=1-1, to=2-1]
	\arrow[from=1-2, to=2-2]
	\arrow["0"', from=2-1, to=2-2]
\end{tikzcd}\]
   get send to (co)cartesian squares by $\E_{S}$-cohomology (here we are using assertion $(f)$ in the definition 
   of geometrically smooth morphisms). 

   Now consider the commutative diagram 
\[\begin{tikzcd}[column sep=small, row sep=small]
	{\ProS^{d-1}_{X}} && {\V_{\ProS^{d-1}_{X}}(\Oh(1))} && W \\
	& {\ProS^{d-1}_{X}} && {\ProS_{\ProS^{d-1}_{X}}(\Oh(1)\oplus\Oh)} && U \\
	{X} && {\A^{d}_{X}} && W \\
	& {X} && {\ProS^{d}_{X}} && U
	\arrow[from=1-1, to=1-3]
	\arrow[from=1-1, to=2-2, equal]
	\arrow[from=1-1, to=3-1]
	\arrow[from=1-3, to=2-4]
	\arrow[from=1-3, to=3-3]
	\arrow[from=1-5, to=1-3]
	\arrow[from=1-5, to=2-6]
	\arrow[from=1-5, to=3-5, equal]
    \arrow[from=3-1, to=3-3]
	\arrow[from=2-2, to=2-4, crossing over]
	\arrow[from=2-2, to=4-2, crossing over]
	\arrow[from=2-6, to=2-4, crossing over]
	\arrow[from=2-6, to=4-6, equal]
	\arrow[from=3-1, to=4-2, equal]
	\arrow[from=3-5, to=3-3]
	\arrow[from=3-5, to=4-6]
    \arrow[from=2-4, to=4-4, crossing over]
    \arrow[from=4-2, to=4-4]
	\arrow[from=4-6, to=4-4]
    \arrow[from=3-3, to=4-4]
\end{tikzcd}\]
   where we write $U$ and $W$ for the respective complements of the zero 
   sections. Now, by Zariski descent, the upper and lower squares in the right cube get 
   send to (co)cartesian squares by $\E_{S}$-cohomology. From this one sees that on $\E_{S}$-cohomology 
   the square in the back of the left cube becomes (co)cartesian if and only if the 
   square in the front of the left cube becomes (co)cartesian.

   Thus, it suffices to check the 
   claim for squares of the form 
\[\begin{tikzcd}
	{\ProS^{d-1}_{X}} & {\ProS_{\ProS^{d-1}_{X}}(\Oh(1)\oplus\Oh)} \\
	{X} & {\ProS^{d}_{X}} 
	\arrow[from=1-1, to=1-2]
	\arrow[from=1-1, to=2-1]
   \arrow["h", from=1-1, to=2-2]
	\arrow["f", from=1-2, to=2-2]
	\arrow["0"', from=2-1, to=2-2].
\end{tikzcd}\]
   For the proof, let us refer to such squares as a projective Blow-up 
   squares. Then we first claim the following: 
     \begin{itemize}
      \item[($\ast$)] The proposition holds for projective Blow-up squares with $X=S$. 
     \end{itemize}
   Let us write $p\from \ProS^{d}_{S} \to S$ for the projection. Then we have to check that the 
   square 
\[\begin{tikzcd}
	{p_{\ast}p^{\ast}\E_{S}} & {p_{\ast}0_{\ast}0^{\ast}p^{\ast}\E_{S}} \\
	{p_{\ast}f_{\ast}f^{\ast}p^{\ast}\E_{S}} & {p_{\ast}h_{\ast}h^{\ast}p^{\ast}\E_{S}}
	\arrow[from=1-1, to=1-2]
	\arrow[from=1-1, to=2-1]
	\arrow[from=1-2, to=2-2]
	\arrow[from=2-1, to=2-2]
\end{tikzcd}\]
    is (co)cartesian. Using the Projective Bundle formula, we see that this square 
    identifies with the square 
\[\begin{tikzcd}
	{\oplus_{i=0}^{d}\E_{S}(-i)} & \E_{S} \\
	{\oplus_{i=0}^{d-1}\E_{S}(-i)\oplus \oplus^{d}_{i=1}\E_{S}(-i)} & {\oplus_{i=0}^{d-1}\E_{S}(-i)}
	\arrow[from=1-1, to=1-2]
	\arrow[from=1-1, to=2-1]
	\arrow[from=1-2, to=2-2]
	\arrow[from=2-1, to=2-2]
\end{tikzcd}\]
   which is easily seen to be (co)cartesian. To see that the map 
   induced on the fibers is the identity one uses that 
      $$f^{\ast}\Oh(1)\iso \Oh(1).$$
   
   Now, for a general projective Blow-up square, let us write $g\from X\to S$ 
   for the structure map. Then applying $\ast$ for $X=S$ with coefficients 
   in $g^{\ast}\E_{S}$ and using that $g_{\ast}$ preserves (co)cartesian squares, we win. 
   This finishes the argument for cohomology. To see the claim for homology, we reduce the 
   claim to projective Blow-up squares over the base in the same way, and for those, one uses that for 
   a proper map $f\from X\to S$, we have an identification 
      $$f_{!}f^{!}\E_{S}\iso \IntHom_{S}(f_{\ast}\Unit_{X},\E_{S})$$ 
   such that the claim follows from the case of cohomology applied to the unit.
\end{proof}

In the following, we will use the notion of deflatability, which first appeared in \cite{ColliotThlne1996TheBT} (see also \cite{Bouis2025connectivityOM}).
For this, we consider the following commutative diagram, which we can associate with an object $S \in \C$. 
\[\begin{tikzcd}
	{\A^{1}_{S}} & {\ProS^{1}_{S}} & S \\
	& S
	\arrow["j", from=1-1, to=1-2]
	\arrow["\pi"', from=1-1, to=2-2]
	\arrow["{\overline{\pi}}", from=1-2, to=2-2]
	\arrow["0"', from=1-3, to=1-2]
	\arrow[equal, from=1-3, to=2-2].
\end{tikzcd}\]

\begin{prop}\label{deflatability proposition}
   Given an object $S\in \C$, we have an identification of the morphisms 
\[\begin{tikzcd}
	{\Coh^{\ast}(\ProS^{1}_{S})} && {\Coh^{\ast}(\A^{1}_{S})}
	\arrow[bend right=10, "{(\pi\circ 0)^{\ast}}"', ""{name=U, below}, from=1-1, to=1-3]
	\arrow[bend left=10, "{j^{\ast}}", ""{name=B, above}, shift left=3, from=1-1, to=1-3]
	\arrow[equal, from=B, to=U]
\end{tikzcd}\]
   in $\QCoh(S)$. This construction is functorial in $S$.
\end{prop}

\begin{proof}
   We follow the proof given in \cite{ColliotThlne1996TheBT}[5.4.3]. That is, we consider the 
   following diagram 
\[\begin{tikzcd}
	{\Coh^{\ast}(\ProS^{1}_{S})} & {\Coh^{\ast}(\A^{1}_{S})} \\
	{\Coh^{\ast}(S)\oplus\Coh^{\ast}(S)(-1)} & {\Coh^{\ast}(S)}
	\arrow["{j^{\ast}}", from=1-1, to=1-2]
	\arrow["{0^{\ast}}", from=1-1, to=2-2]
	\arrow["\iso", from=2-1, to=1-1]
	\arrow["{\pi^{\ast}}"', from=2-2, to=1-2]
\end{tikzcd}\]
   where the left vertical isomorphism comes from the projective bundle formula. 
   We want to see that the triangle on the right commutes. To check this, we can precompose with the 
   just explained isomorphism. On the factor to the left, this follows from the functoriality 
   of cohomology, and on the factor on the right, both compositions are given by the first Chern class 
   of $\Oh_{\A^{1}_{S}}$.
\end{proof}

We are now ready to recall the construction of cycle class maps from \cite{Tang2026TheG}.

\begin{non}[\textbf{Cycle class map}]\label{cycle class map constr}
   We follow the construction given in \cite{Tang2026TheG}[3.1]. Consider a regular closed immersion $i\from Z\to X$ between two geometrically 
   smooth objects over $S$. 
   Then we can produce the following blowup squares 
\[\begin{tikzcd}
	{\ProS_{Z}(\Normal_{i})} & {\Bl_{Z}(X)} && {\ProS_{Z}(\Normal\oplus \Oh)} & {\Bl_{Z}(X\times\ProS^{1})} \\
	Z & X && {Z\times\{0\}} & {X\times\ProS^{1}}
	\arrow[from=1-1, to=1-2]
	\arrow[from=1-1, to=2-1]
	\arrow[from=1-2, to=2-2]
	\arrow[from=1-4, to=1-5]
	\arrow[from=1-4, to=2-4]
	\arrow[from=1-5, to=2-5]
	\arrow[from=2-1, to=2-2]
	\arrow[from=2-4, to=2-5].
\end{tikzcd}\]
   Note that there is a canonical map of squares from the left to the right. As we assume blow-up 
   excision \ref{blowup excision prop}, this map induces a map between (co)cartesian squares 
\[\begin{tikzcd}
	{\Coh^{\ast}(\ProS^{1}_{X})} & {\Coh^{\ast}(Z)} & {\Unit_{X}} & {\Coh^{\ast}(Z)} \\
	{\Coh^{\ast}(\Bl_{Z}(\ProS^{1}_{X}))} & {\Coh^{\ast}(\ProS^{1}_{Z}(\Normal_{i}\oplus \Oh))} & {\Coh^{\ast}(\Bl_{Z}(X))} & {\Coh^{\ast}(\ProS_{Z}(\Normal_{i}))}
	\arrow[from=1-1, to=1-2]
	\arrow[from=1-1, to=2-1]
	\arrow[from=1-2, to=2-2]
	\arrow[from=1-3, to=1-4]
	\arrow[from=1-3, to=2-3]
	\arrow[from=1-4, to=2-4]
	\arrow[from=2-1, to=2-2]
	\arrow[from=2-3, to=2-4]
\end{tikzcd}\]
   in $\QCoh(X)$. Thus, taking fibers at each term of this map, we obtain a (co)fiber sequence 
      $$\Coh^{\ast}(\ProS^{1}_{X}/X\times \{0\}) \to Q \to \Coh^{\ast}(\Th_{Z}(\Normal_{1}))$$.
	Furthermore the one section $X\times\{1\} \to \Bl_{Z}(\ProS^{1}_{X})$ induces a map 
	   $$Q \to \Unit_{X}$$
	who, using \ref{deflatability proposition}, identifies with the zero map after precomposing to 
	$\Coh^{\ast}(\ProS^{1}_{X}/X\times \{0\})$. This induces a map 
	   $$\Coh^{\ast}(\Th_{Z}(\Normal_{i})) \to \Unit_{X}$$
	and using \ref{thom spaces are trivialized}, we obtain a map 
	  $$\text{cl}_{i} \from i_{\ast}\Unit_{Z} \to \Unit_{X}(n)$$
	in $\QCoh(X)$, where $n$ is the codimension of $Z$ in $X$.
\end{non}

\begin{defn}
	For a regular closed immersion $i\from Z\to X$ between two objects in $\Smoo^{\text{geom}}_{S}$, 
	we will refer to the map 
	   $$\text{cl}_{i}\from i_{\ast}\Unit_{Z} \to \Unit_{X}(n)$$
	constructed in \ref{cycle class map constr} as the \emph{cycle class map} associated to 
	$i$. 
\end{defn}

The rest of this section is devoted to proving the following. The statement asserts 
that on $\Sm^{\text{sep}}_{B}$ our theory of first Chern classes underlies a theory of cycle classes in 
the sense of \cite{ZavyalovPoincareDI}[5.3.3.]. The cycle classes we use are the ones constructed above in the 
case that the regular immersion is given by an effective Cartier divisor.

\begin{thm}\label{chern classes underly cycle classes}
	Consider an effective cartier divisor $i\from D\to S$ in $\Sm^{\text{sep}}_{B}$. 
	Then the composition 
	   $$\text{cl}_{i}\circ \text{ad}_{\ast}^{\ast}\from \Unit_{S} \to i_{\ast}\Unit_{Z} \to \Unit_{S}(1)$$
	canonically identifies with the first Chern class associated to the line bundle $\Oh(D)$.
\end{thm}

We will start by recalling some computations (mainly taken from \cite{Tang2026TheG} and \cite{TangGysin}). 

\begin{non}[\textbf{Cohomology of $\ProS^{\infty}$}]\label{cohomology of Pinfty}
   Consider the canonical diagram 
\[\begin{tikzcd}
	S & {\ProS_{S}^{1}} & {\ProS_{S}^{2}} & {\ProS_{S}^{3}} & \dots \\
	&& S
	\arrow[from=1-1, to=1-2]
	\arrow[equal, from=1-1, to=2-3]
	\arrow[from=1-2, to=1-3]
	\arrow["{p_{1}}", from=1-2, to=2-3]
	\arrow[from=1-3, to=1-4]
	\arrow["{p_{2}}"{description}, from=1-3, to=2-3]
	\arrow[from=1-4, to=1-5]
	\arrow["{p_{3}}", from=1-4, to=2-3]
\end{tikzcd}\]
   over some separated smooth scheme $S$ over $B$ and let us write $\ProS_{S}^{\infty}$ for the colimit. We can compute the 
   cohomology of $\ProS_{S}^{\infty}$ by the formula 
      $$\limit_{n}(p_{n})_{\ast}(p_{n})^{\ast}\Unit_{S}\iso\prod^{\infty}_{i=0}\Unit_{S}(-i)\in \QCoh(S).$$
\end{non}

By \ref{deflatability proposition}, we see that cohomology identifies 
$\ProS^{1}$-homotopic maps and thus also (non-trivially) weighted $\A^{1}$-homotopic maps 
(see \cite{AnnalaAlgebraicCob}[4.6.]). We will use this in the following computations.

\begin{non}\label{weighted homotopy invariace computation}
   An example of a weighted homotopy equivalence is the map 
      $$V/\BGm \to \BGm$$
	where $V$ denotes a vector bundle over a smooth and separated $B$-scheme and 
	$\Gm$ acts on $V$ with non-zero weight.  
   From this, one also sees that the map 
      $$cof((V/\BGm)-(\BGm) \to V/\BGm) \to cof(\Oh(1)\from \ProS(V) \to \BGm)$$
   is a weighted homotopy equivalence.
   In particular, we see that we obtain an isomorphism
      $$\Coh^{\ast}(\BGm/S) \isoto \Coh^{\ast}((V/\Gm)/S)$$
   (and similar for the second map) in $\QCoh(S)$.
\end{non}

\begin{non}[\textbf{Cohomology of $\BGm$}]\label{cohomology of BGm}
   Consider a vector bundle $\F$ on $S$ for which we can find a surjective map 
   $\F \to \Oh$. Then using $\ProS^{1}$-homotopies \ref{deflatability proposition} and elementary blow up excision 
   \ref{Elementary blowup corollary}
   one proves the same way as in \cite{AnnalaAlgebraicCob}[5.3] that the projection 
      $$\Coh^{\ast}(\BGln/S) \isoto \Coh^{\ast}(\Grn(\F^{\infty})/S)$$
   from the infinite Grassmannian is an isomorphism in $\QCoh(S)$. 

   A special case of this computation is that the map 
      $$\ProS^{\infty} \to \BGm$$
   induces an isomorphism on cohomology. In particular, using \ref{cohomology of Pinfty}
   we see that we have a canonical identification 
      $$\Coh^{\ast}(\BGm/S)\iso \prod_{i=0}^{\infty}\Unit_{S}(-i)$$
    in $\QCoh(S)$. 
\end{non}

\begin{non}[\textbf{Description of first chern classes}]\label{description of first chern classes}
   Using \ref{cohomology of BGm} one sees that given a line bundle $\Line$ over $S$ corresponding 
   to a map 
      $$i_{\Line} \from S\to \BGm.$$
   The first Chern class of $\Line$ is given by the composition
      $$\Unit_{S} \to \prod_{i=0}^{\infty}\Unit_{S}(1-i)\iso \Coh^{\ast}(\BGm/S)(1) \to \Unit_{S}(1)$$
	where the first map is the inclusion into the first factor, and the second map is 
   induced by $i_{\Line}$ on cohomology.
\end{non}

\begin{non}[\textbf{Normalisation of cycle class maps}]\label{normalisation of cycle class}
   Using $\ProS^{1}$-homotopies \ref{deflatability proposition} and elementary blow up excision
   \ref{Elementary blowup corollary}, as in the proof of \cite{Tang2026TheG}[3.10], on sees the following. 
   Given a vector bundle $V$ over a smooth separated $B$-scheme, the cycle class map 
   of the zero section $S\to V$ is refined by a map 
      $$\Coh^{\ast}(\Th_{S}(V)/V) \to \Coh^{\ast}(V/V-\{0\}/V)$$ 
    which admits a section. This section is functorial in $S$ and $V$ (see \cite{Tang2026TheG}[2.17]).
\end{non}

\begin{proof}[Proof of \ref{chern classes underly cycle classes}]
	Note first that the assertion is stable under base change, so it suffices to check the claim 
	for the universal effective Cartier divisor 
\[\begin{tikzcd}
	{G=\BGm} & {V=\A^{1}/\Gm} \\
	& {G=\BGm}
	\arrow["0", from=1-1, to=1-2]
	\arrow["id"', from=1-1, to=2-2]
	\arrow["p", from=1-2, to=2-2].
\end{tikzcd}\]
   That is, we have to check that the composition 
      $$\text{cl}_{0}\circ\text{ad}_{\ast}^{\ast}\from \Unit_{V} \to 0_{\ast}\Unit_{G} \to \Unit_{V}(1)$$
	identifies with the first Chern class associated to the line bundle $\Line^{\text{univ}}$ corresponding to the map 
	   $$p\from V\to G.$$
	By adjunction, this map corresponds to the map 
	  $$\Unit_{S}(-1) \to \Coh^{\ast}(\Th_{G}(\Line^{\text{univ}}_{|G})/S) \to \Coh^{\ast}(V/S)$$
	in $\QCoh(S)$, where the first map comes from the structure morphism and the second 
	from \ref{cycle class map constr}. As we are in the situation of \ref{normalisation of cycle class}, we see that the second map 
	of this composition identifies with the map 
	  $$\Coh^{\ast}(\Th_{G}(\Line_{|G}^{\text{univ}})/S)\to \Coh^{\ast}(V/V-\{0\}/S)\to \Coh^{\ast}(V/S)$$
	where the first map is the section from \ref{normalisation of cycle class} and the second map 
	is induced by the canonical projection. 
	Finally, using the identification from \ref{weighted homotopy invariace computation} and \ref{cohomology of BGm}
	we see that the map in question can be described as the composition 
	   $$\Unit_{S}(-1) \to \prod_{i=1}^{\infty} \iso \Coh^{\ast}(G/\ProS^{1}/S) \to \Coh^{\ast}(V/S)$$
	where the first map is the inclusion into the first factor and the second map is induced by $p$. 
	This identifies with the first Chern class of $\Line^{\text{univ}}$ by \ref{description of first chern classes}. 
\end{proof}

\subsection{Cohomological smoothness}

In this section, we will argue that any geometrically smooth morphism will be cohomologically 
smooth through the eyes of our six-functor formalism. We will also compute the dualizing sheaf 
of such a morphism. 

For this, we fix an additively oriented six-functor formalism $\QCoh$ on $\C$, for which cohomology 
satisfies elementary blow-up excision. Furthermore, we fix a class of geometrically smooth and geometrically étale morphisms. 

We start with the construction of a trace morphism for the morphism 
$f\from \ProS^{1}_{B}\to B$. For this, we follow \cite{ZavyalovPoincareDI}[5.6.1]. 

\begin{constr}\label{construction of the trace morphism}
    As a dualizing sheaf, we want to consider the Tate twist. That is, we have to construct 
    a morphism 
       $$\text{tr}_{f} \from f_{\ast}\Unit_{\ProS^{1}_{B}}(1) \to \Unit_{B}.$$
    Such a morphism is given by the composition 
       $$f_{\ast}\Unit_{\ProS_{B}^{1}}(1)  \to \Unit_{B}\oplus \Unit_{B}(1) \to \Unit_{B}.$$
    The first map is the inverse of the Projective Bundle formula isomorphism, 
    and the second is the projection to the first factor. 
\end{constr}

Using the work \cite{ZavyalovPoincareDI}, we obtain the following. 

\begin{cor}\label{smoothnes without the dualizing sheaf}
    Any geometrically smooth morphism $f\from X\to S$ is cohomologically smooth for $\QCoh$. 
    Furthermore, if the morphism factors as 
       $$f=p\circ e \from X\to \A^{n}_{S} \to S$$
    where $e$ is cohomologically étale and $p$ the projection, we have an identification 
       $$f^{!}(\Unit_{S}) \iso \Unit_{X}(n)$$
    in $\QCoh(X)$. 
\end{cor}

\begin{proof}
   By \ref{Bogdans criterium}, to deduce the first claim, it suffices to check 
   the following:
      \begin{itemize}
         \item[($\ast$)] The triple $(\Unit_{\ProS_{B}^{1}}(1),\text{tr}_{f},\text{cl}_{\Delta})$
                         constructed in \ref{construction of the trace morphism} and \ref{cycle class map constr}
                         gives a trace-cycle theory for the projection $f\from \ProS_{B}^{1}\to B$.
      \end{itemize}
   But by \ref{chern classes underly cycle classes} we see that our theory of 
   Chern classes underlie our theory of cycle classes, allowing us to apply \cite{ZavyalovPoincareDI}[5.6.6]. 
   The computation of the dualizing sheaf now easily follows by base change and the 
   fact that for a cohomologically \'etale morphism $f\from U\to S$, we have 
       $$f^{\ast}\iso f^{!}.$$
\end{proof}

Recall the category of motivic spectra $\MS(B)$ from \cite{AnnalaAlgebraicCob}. Then we also obtain the following 
corollary, which allows us to use the assertions from \cite{Tang2026TheG}. 

\begin{cor}\label{functor from MS}
    Associating to a smooth separated $B$-scheme $f\from X\to B$ its homology 
       $$\Coh_{\ast}(X/B)\in \QCoh(B)$$
    induces a symmetric monoidal functor 
       $$\MS(B) \to \QCoh(B).$$
\end{cor}

\begin{proof}
    As any such morphism $f\from X\to B$ is cohomologically smooth for $\QCoh$ by \ref{smoothnes without the dualizing sheaf},
    taking homology is symmetric monoidal. So the assertion follows from the universal property 
    of $\MS(B)$ using \ref{Elementary blowup corollary} and \ref{thom spaces are trivialized}. 
\end{proof}

\begin{constr}\label{construction cycle class map general}
   Let us consider a geometrically smooth morphism $f\from X\to S$, then we can produce the 
   following diagram:
\[\begin{tikzcd}
	X \\
	& {X\times_{S}X} & X \\
	& X & S
	\arrow["\Delta", from=1-1, to=2-2]
	\arrow[bend left=20, equal, from=1-1, to=2-3]
	\arrow[bend right=20, equal, from=1-1, to=3-2]
	\arrow["p", from=2-2, to=2-3]
	\arrow["q"', from=2-2, to=3-2]
	\arrow["f", from=2-3, to=3-3]
	\arrow["f"', from=3-2, to=3-3].
\end{tikzcd}\]
   Furthermore by assertion $(e)$ in \ref{definition of geometrically smooth morphisms}, there is a factorization 
      $$\Delta \iso j \circ \tilde{\Delta} \from X\to U \to X\times_{S}X$$
   into a Zariski closed immersion followed by an open immersion. 
   Thus, using \ref{cycle class map constr}, we obtain a cycle class map 
      $$\text{cl}_{\tilde{\Delta}} \from \tilde{\Delta}_{!}\Unit_{X} \to \Unit_{U}(d).$$
   Furthermore, applying $j_{!}$ and using the counit, we obtain a map 
      $$\text{cl}_{\Delta} \from \Delta_{!}\Unit_{X} \to \Unit_{X\times_{S}X}(d)$$
   which we will also refer to as \emph{cycle class map}. 
\end{constr}

We now obtain the following theorem. 

\begin{thm}\label{dualizing complex theorem}
   Given a geometrically smooth morphism $f\from X\to S$, then the
   data 
     \begin{itemize}
        \item[(a)] $\Unit_{X}(d)$
        \item[(b)] $\text{cl}_{\Delta} \from \Delta_{!}\Unit_{X}\to \Unit_{X\times_{S}X}(d)$ 
     \end{itemize} 
    is part of a trace-cycle theory. Where $d$ denotes the relative dimension of $f$.  
\end{thm}

\begin{proof}
   We claim that the cycle class map is the unit of an adjunction $\Unit_{X} \dashv \Unit_{X}(d)$.
   Note that the existence of a counit for this adjunction is a property. Namely the property, that 
   for any $T\to S \in (\C_{E})_{/S}$ and any two objects $\E \in \QCoh(T\times_{S}X)$ and 
   $\F \in \QCoh(T)$ the canonical map 
      $$\Hom_{\QCoh(T)}(p_{!}\E,\F) \to \Hom_{\QCoh(T\times_{S}X)}(\E,p^{\ast}\F(d))$$
    induced by the cycle class map is an isomorphism (see \cite[\href{https://kerodon.net/tag/02CU}{Tag 02CU}]{Kerodon}). Here we write $p\from T\times_{S}X \to T$
    for the projection. This can be checked cohomologically étale locally on $X$ and $S$. 

    So by \ref{definition of geometrically smooth morphisms}$(d)$ we can assume that $f$ factors 
    as an étale map followed by a projection from an affine space. As a cycle class theory can be pulled back 
    along a cohomologically étale morphism on the domain, we can further assume that $f$ is given by 
    a projection from an affine space and as trace-cycle theories are stable under base change, we can even assume 
    that $f$ is given by the map $\A^{n}_{B} \to B$.

    Now note that a regular closed immersion locally factors as a composition of effective Cartier divisors. So using the compatibility 
    of cycle class maps with composition \cite{Tang2026TheG}[3.9, 3.17] (here we make use of \ref{functor from MS}) and again stability of base change of trace-cycle theories, 
    we can assume $f$ is given by the projection $\A^{1}_{B} \to B$. Now it suffices to construct a 
    trace-cycle theory for the map $\ProS^{1}_{B} \to B$, which we did in \ref{smoothnes without the dualizing sheaf}. 
\end{proof}

\begin{cor}\label{dualizing complex corollary}
    For any geometrically smooth morphism $f\from X \to S$, we obtain a canonical identification 
        $$f^{!}(\Unit_{S}) \iso \Unit_{X}(d)$$
    in $\QCoh(X)$, where $d$ denotes the relative dimension of $f$.
\end{cor}
\section{Formal schemes as analytic stacks}\label{formal}

\subsection{Recollections on analytic stacks}

In this text, we will make use of the theory of analytic stacks introduced by 
Clausen and Scholze \cite{AStacksLecture}. We will now recall the examples of those
we will use. 

\begin{non}
  Recall that an affine analytic stack is represented by a pair 
     $$(A,\QCoh(A))$$
  where $A$ is a (light) condensed ring and $\QCoh(A)$ a full subcategory, stable under limits and colimits, 
  of the category of modules over $A$ in (derived) condensed abelian groups.
  This subcategory is such that it inherits a closed symmetric monoidal structure and 
  a $t$-structure from those $A$-modules. 
\end{non}

\begin{ex}
  We will write $\Spa(\Znumb)$ for the analytic stack represented by the 
  pair $(\Znumb,\DSolid(\Znumb))$ where the subcategory is given by solid abelian groups 
  \cite{AStacksLecture}[Lecture 5] \cite{CondensedMath}. All our analytic stacks will live
  over $\Spa(\Znumb)$. 
\end{ex}

\begin{ex}
   An important object is given by the following construction.
   Consider the profinite set 
       $$\Nnumb\cup \{\infty\}:= \limit_{n} \{0,\dots,n,\infty\}$$
   where the transition maps send the highest number to $\infty$. To this we can associate its free (light) 
   condensed derived abelian group $\Znumb[\Nnumb\cup \{\infty\}]$ and we will write $\IntProjP$
   for the cofibre 
      $$\Znumb[\{\infty\}] \to \Znumb[\Nnumb\cup \{\infty\}] \to \IntProjP$$
   of the canonical inclusion. Note that this is a split exact sequence. An important fact 
   of this object is that it is internally compact in derived condensed abelian groups \cite{AStacksLecture}[Lecture 3].
\end{ex}

\begin{non}\label{P produced analytic ring structures}
   One can use $\IntProjP$ to construct new affine analytic stacks out of old ones. Namely, if we consider an analytic ring $A$
   internally compactness of $\IntProjP$ formally implies that 
       $$\IntProjP_{A}:= \IntProjP\otimes_{\Znumb}A \in \QCoh(A)$$
   is internally compact as well. Furthermore if we have given any collection $W$ of endomorphisms of $\IntProjP_{A}$ 
   and write 
     $$\widehat{(\_)} \from \QCoh(A) \to \QCoh(A)[W^{-1}]$$
   for the Bousfield localization obtained by inverting maps in $W$. Then it is easy to check that 
     $$(A,\QCoh(A)[W^{-1}])$$
   represents a new affine analytic stack. All affine analytic stacks in this text will arise in this way. 
\end{non}

\begin{non}\label{Huber pairs as analytic stacks}
   The main example we will consider is the following. Consider an animated ring $A$
   together with some finitely generated ideal $I$ in $A$ and a subring $A^{+}$ of $A$. 
   Then we will write 
     $$\Spa(A_{\widehat{I}},A^{+})$$
    for the following construction. The (derived) $I$-adic completion of $A$ as a condensed ring 
    defines a solid ring, and we can invert the maps 
      $$f_{\Solid}:= \text{id}- f\cdot \text{shift} \from \IntProjP_{A_{\widehat{I}}} \to \IntProjP_{A_{\widehat{I}}}$$
    for all $f\in A^{+}$ in $\Mod_{A_{\widehat{I}}}(\DSolid(\Znumb))$ to obtain $\QCoh(\Spa(A_{\widehat{I}},A^{+}))$.
\end{non}

\begin{rem}
   The analytic stack $\Spa(A,A^{+})$ does not change if one adjoints, in $A$, topological nilpotent 
   or elements to $A^{+}$. It also does not change after adjointing elements which are integral over $A^{+}$. 
\end{rem}

\begin{rem}
    In the very end of this text, we will also consider a version of \ref{Huber pairs as analytic stacks}, but inverting 
    topological nilpotent elements in $A$ (i.e. consider derived Tate-Huber pairs). 
\end{rem}

\begin{rem}
  In the case $A=A^{+}$, we will just write $\Spa(A)$.
\end{rem}

\begin{non}\label{topological spaces of a Huber pair}
    Given a triple as in \ref{Huber pairs as analytic stacks}, we will write 
       $$|\Spa(A,A^{+})|$$
    for the space of continuous valuations on $\pi_{0}A$, which value all functions in $A^{+}$
    as $\le 1$. This can be understood as a quasi-compact topological space \cite{WedhornAdicSpacesLect}[4.7]. 
    And by \cite{AStacksLecture}[Lecture 9] there is a map of locales 
       $$\Sm(\QCoh(\Spa(A,A^{+}))) \to |\Spa(A,A^{+})|$$
      where on the left we consider the smashing spectrum of a stable symmetric monoidal category \cite{Complexgeometry}[5.3].
\end{non}

\begin{non}\label{sixfunctors on analytic stacks}
  By \cite{AStacksLecture}[Lecture 16], there is a six-functor formalism on the category 
  of affine analytic stacks. One can then use this six-functor formalism to implement some aspects of 
  geometry into affine analytic rings. 
     \begin{itemize}
      \item A map of $j\from U\to X$ of affine analytic stacks is called an open immersion if 
                $$j^{\ast}\from \QCoh(X)\to \QCoh(U)$$
            admits a left adjoint $j_{!}$ and $j_{!}\Oh_{U}$ defines an idempotent coalgebra.
      \item A map $f\from X\to S$ of affine analytic stacks is called proper if 
               $$\QCoh(X) \iso \Mod_{f_{\ast}\Oh_{X}}(\QCoh(S)).$$
      \item A map of affine analytic stacks is called $!$-able, if it factors as an open immersion 
            followed by a proper map.  
     \end{itemize}
  Using this six-functor formalism, one can define a (Grothendieck) topology on affine analytic stacks. This topology is 
  such that for a $!$-able covering $U\to X$ the functor $\QCoh(\_)^{!}$ (which implies the analogous statement for $\QCoh(\_)^{\ast}$) 
  descends along this covering. Using the techniques of \cite{HeyerMann6FF} one then obtains a 
  six-functor formalism on the category of analytic stacks. We will write 
     $$\AnStack$$
  for the latter. 
\end{non}

Instead of going into the details of this defining topology, let us give the main examples, which 
will appear in this text. 

\begin{ex}
   A cohomologically étale morphism of affine analytic stacks defines a covering if $\QCoh(\_)^{\ast}$
   descends along this morphism. 
\end{ex}

\begin{ex}\label{descendable cover of analytic stacks}
   A proper morphism $f\from X\to S$ of affine analytic stacks gives a covering if 
   the algebra 
      $$f_{\ast}\Oh_{X} \in \QCoh(S)$$
    is descendable in the sense of \cite{MathewGaloisgroupsofstablehomotopy}[3.18].
\end{ex}

\begin{ex}\label{integral affine maps are proper}
   Consider a map of affine analytic stacks $f\from \Spa(B,B^{+}) \to \Spa(A,A^{+})$. 
   Then $f$ is proper if and only if the map $A^{+} \to B^{+}$ is integral up to topological nilpotent elements.
\end{ex}

\subsection{Formal stacks and the éid-topology}

In this section, we will clarify what we mean by a formal stack and introduce the 
main topology used in this text. 

\begin{constr}
   Consider an animated ring $A$ and a finitely generated ideal $I\subset \pi_{0}A$, such that 
   $A$ is derived $I$-adically complete. We 
   define a presheaf on the category of affine derived schemes $\Spf(A)$ by the 
   cartesian square
\[\begin{tikzcd}
	{\Spf(A)} & {\Spec(A)} \\
	{\Hom_{\text{cont}}(\pi_{0}A,\_)} & {\Spec(\pi_{0}A)}
	\arrow[from=1-1, to=1-2]
	\arrow[from=1-1, to=2-1]
	\arrow[from=1-2, to=2-2]
	\arrow[from=2-1, to=2-2]
\end{tikzcd}\]
   where we equip $\pi_{0}A$ with the $I$-adic topology. 
\end{constr}

\begin{defn}
    We will call a presheaf of the form $\Spf(A)$ an (derived)
    affine formal scheme. 
\end{defn}

\begin{rem}
   Note that $\Spf(A)$ just depends on the $I$-adic topology on $\pi_{0}A$
   not on the ideal itself. Any ideal defining the same topology will be 
   called an \emph{ideal of definition}. 
\end{rem}

\begin{rem}
    The category $\QCoh(\Spf(A))$ of quasi-coherent sheaves on $\Spf(A)$ is computed by the full subcategory 
       $$\QCoh_{I\text{-comp}}(A)\subset \QCoh(A)$$
    of those $A$-modules, which are derived $I$-adically complete. 
\end{rem}

\begin{non}
    Many types of maps between formal stacks can be defined by declaring 
    them to be representable by schemes and then referring to the corresponding 
    notion for schemes. We will use the following notions in this sense:
       \begin{itemize}
        \item[(a)] \emph{Affine open immersions}.
        \item[(b)] \emph{Affine étale morphisms}. 
        \item[(c)] \emph{Integral morphisms}.  
       \end{itemize}
\end{non}

To define the topology we want to work with, we will make use 
of the notion of a descendable algebra from \cite{MathewGaloisgroupsofstablehomotopy}. 

\begin{defn}
    A map of affine formal schemes $\Spf(B) \to \Spf(A)$, which is 
    representable by affine schemes, is called \emph{descendable},
    if 
       $$B \in \QCoh(\Spf(A))$$
    is a descendable algebra. 
\end{defn}

\begin{constr}\label{EIDtopos}
    Consider a presheaf $S$ on affine derived schemes. Then let us write 
       $$\fSch^{\text{aff}}_{/S}$$
    for the full subcategory of those presheaves over $S$, which are given 
    by an affine formal scheme 
      $$\Spf(A)\to S,$$
    such that the structure map as well as the diagonal $\Spf(A)\to \Spf(A)\times_{S}\Spf(A)$
    are representable by affine schemes. 

    We will equip this category with a (pre)-topology called the \emph{éid}-topology. 
    This topology is the union of the étale topology with the topology 
    generated by integral descendable maps. The induced topos will be denoted by 
       $$S_{\text{éid}}.$$ 
\end{constr}

\begin{vari}\label{flat EID topos}
    For later use, let us also record the following variant of \ref{EIDtopos}. 
    Assume $S$ is a static $p$-nilpotent ring. Then we consider the category of static $S$-algebras and 
    equip this category with the intersection of the flat topology with 
    the éid-topology. We will write 
       $$S_{\text{éfid}}$$
    for the induced topos.
\end{vari}

\begin{rem}
    Let us say a map $\Spf(B)\to \Spf(A)$ is called \emph{adic}, if there exists 
    an ideal of definition $I_{A}$ for $\Spf(A)$, such that 
       $$I_{A}\pi_{0}B \subset \pi_{0}B$$
    gives an ideal of definition for $\Spf(B)$. Then a map of affine formal schemes 
    is adic if and only if it is representable by affine schemes. 
\end{rem}

\begin{rem}\label{Fibre products of adic maps Remark}
    The ultimate reason we fix a base is the following: 
    Consider a cospan of representable morphisms 
       $$\Spf(B)\to \Spf(A) \leftarrow \Spf(C).$$
    Then the solid tensor product 
       $$C_{\widehat{I}}\otimes^{\Solid}_{A_{\widehat{I}}}B_{\widehat{I}}$$
    is derived $I$-adically complete \cite{BoscoRationalH}[A.3]. This is not true if 
    we do not assume the maps to be representable. 
    
    The generality of our base is chosen, as the objects we want to consider as a base will 
    be honest stacks. 
\end{rem}

\begin{rem}
    Any (derived) formal scheme $X$ admitting a map to $S$, 
    which locally on $X$ in the Zariski topology is representable by affines, defines an 
    object in $S_{\text{éid}}$. We will say a map of formal stacks 
    over $S$ is \emph{étale} if it is representable by étale maps 
    of formal schemes.  
\end{rem}

\begin{rem}
   Note that associating to an affine formal stack $\Spf(A)$ its category of quasi-coherent 
   sheaves $\QCoh(\Spf(A))$ defines a sheaf on $S_{\text{éid}}$. In particular, we 
   obtain a category of quasi-coherent sheaves $\QCoh(X)$ for any formal stack 
   $X$ over $S$.
\end{rem}

We will also use the following notion of covering. 

\begin{defn}\label{descendable cover of formal stacks}
   We will say a map $f\from X\to S$ of formal stacks is a descendable cover if 
     $$f_{\ast}\Oh_{X} \in \QCoh(S)$$
   is descendable.   
\end{defn}

We will use the following notions of maps of formal stacks and formal schemes. 

\begin{defn}
    We will say a map of affine formal schemes $\Spf(B) \to \Spf(A)$
    is of \emph{\textsuperscript{$+$}finite type}, if it can be factored 
       $$\Spf(B) \to \A^{n}_{\Spf(A)} \to \Spf(A)$$
    as an integral map followed by a projection from an affine space.

    We will say a map $X\to S$ of $p$-adic formal schemes is:
        \begin{itemize}
            \item \emph{\textsuperscript{+}Proper}, if the following conditions hold.
                  \begin{itemize}
                      \item[(a)] Zariski locally on target and domain it is of \textsuperscript{+}finite type.
                      \item[(b)] It is represented by separated and universally closed morphisms of schemes. 
                  \end{itemize}  
            \item \emph{Proper}, if it is \textsuperscript{+}proper and locally of finite presentation.
         \end{itemize}
    
    Furthermore, we will say a map of formal stacks $X \to S$
    is \emph{locally of \textsuperscript{$+$}finite type}, if, locally in the éid-topology on source and target, 
    it is of \textsuperscript{+}finite type. 
\end{defn}

\subsection{From formal stacks to analytic stacks} 

We will construct a functor from a suitable 
category of formal stacks to analytic stacks, and explain why this 
functor behaves well with respect to the geometries on both sides. 

\begin{rem}
   In the following, we will give criteria for proving that maps of formal stacks 
   induce $!$-able or proper maps on the side of analytic stacks. For this, it will 
   be convenient to work over an affine base. Note that there is no harm in doing so, 
   as we can check $!$-ability as well as properness locally on the target. 
\end{rem}

\begin{non}\label{compactification map on affines}
    To an affine formal scheme $\Spf(A)$, we associate 
    the analytic stacks 
       $$\Spa(A) \to \Spa(A,\widetilde{\Znumb})$$
    together with the canonical map coming from solidifying the 
    points in $A$. This construction is natural in $\Spf(A)$.   
\end{non}

\begin{thm}\label{From formal to analytic prop}
   The following hold:
      \begin{itemize}
         \item[\textbf{A}] There exist left exact and colimit preserving functors 
                  $$(\_)^{\Solid}\to (\_)^{\Solid\text{lp}}\from S_{\text{éid}} \to \AnStack_{/S^{\Solid}}$$
               together with a natural transformation. This datum is uniquely 
               determined by the assertion that on affine formal schemes it is 
               given by \ref{compactification map on affines} and then one pulls back the 
               target along the map $S^{\Solid} \to S^{\Solid\text{lp}}$. 
         
         \item[\textbf{B}]  The functor $(\_)^{\Solid}$ preserves étale morphisms and 
               sends maps represented by integral maps of affine formal schemes 
               to weakly proper morphisms. 

         \item[\textbf{C}] Given a map $X\to S=\Spf(A)$ of formal stacks for which we can find a 
               descendable cover 
                  $$\Spf(B) \to X$$
               from an affine formal scheme and
               which is locally on the source and 
               target of \textsuperscript{$+$}finite type, the induced map 
                    $$X^{\Solid} \to S^{\Solid}$$
               is $!$-able. 
         \item[\textbf{D}]Under the assumptions of \textbf{C} then the following are equivalent:
                   \begin{itemize}
                     \item[(a)] The map $X^{\Solid} \to S^{\Solid}$ is co-smooth. 
                     \item[(b)] The map $X\times_{S}\Spec(A/I) \to \Spec(A/I)$ is co-smooth, where 
                        $A/I$ is the quotient of $A$ by some finitely 
                        generated ideal of definition\footnote{Here, this quotient depends on a chosen collection of finitely many generators 
                        of $I$. The conclusion of the statement, though, is independent of this choice.}.
                   \end{itemize}
         \item[\textbf{E}] Under the assumptions of $\textbf{C}$ the map 
                               $$X^{\Solid} \to S^{\Solid}$$
                           is co-smooth, if and only if it is co-smooth over a finite 
                           Zariski stratification of $S$. 
      \end{itemize}
\end{thm}

Before proving the theorem let us record the following corollary.

\begin{cor}\label{Proper maps of formal schemes are proper}
   Consider a \textsuperscript{+}proper map $X\to S$ of $p$-adic formal schemes. 
   Then the induced map 
      $$X^{\Solid} \to S^{\Solid}$$
   is weakly proper.
\end{cor}

\begin{proof}
   Note that the diagonal of a \textsuperscript{+}proper morphism is a closed immersion and thus weakly proper by 
   \ref{weakly proper from decomposition} and \ref{integral affine maps are proper}.
   Thus by \ref{weakly proper from diagonal and cosmooth} we just have to see that the map 
   is co-smooth.
   To see this, we note that such a map is quasi-compact, and we can find a descendable cover using a Zariski cover. 
   In particular using \ref{From formal to analytic prop}[\textbf{D}]
   it suffices to check the claim after base change to $\Spec(\Fp)$. This base change
   now is a map of discrete adic spaces, so the claim follows from \cite{LucasMannPhD}[2.9.29].
\end{proof}

The proof of this Theorem will occupy the rest of the section. Let us 
start with the first paragraph. 

The left exactness claims follow from \ref{Fibre products of adic maps Remark}
together with the observation that a map between affine formal schemes 
in $S_{\text{éid}}$ is representable by affine schemes (i.e. adic).  
Thus, to prove the first claim, we have to argue that the functors $(\_)^{\Solid}$
and $(\_)^{\Solid\text{lp}}$ send $\text{éid}$-coverings to effective 
epimorphisms in analytic stacks.

\begin{non}\label{etale maps as pullbacks of etale maps}
    We will use the following observation in two places below. 
    Any étale morphism $\Spf(B) \to \Spf(A)$ of affine formal stacks sits in 
    a cartesian square 
\[\begin{tikzcd}
	{\Spf(B)} & {\Spf(A)} \\
	{\Spec(B')} & {\Spec(A)}
	\arrow[from=1-1, to=1-2]
	\arrow[from=1-1, to=2-1]
	\arrow[from=1-2, to=2-2]
	\arrow[from=2-1, to=2-2]
\end{tikzcd}\]
   where the lower horizontal map is an étale morphism of schemes. Furthermore if the 
   map of formal schemes was surjective, the map of schemes can be chosen to be surjective as 
   well. 

   This follows from \cite[ \href{https://stacks.math.columbia.edu/tag/0AN8}{Tag 0AN8}]{StacksProject} combined with the fact that, for an animated ring $A$, base change 
   along the map $A\to \pi_{0}A$ induces an equivalence on étale algebras.
\end{non}

\begin{proof}[Proof of \ref{From formal to analytic prop} \textbf{A} $(\_)^{\Solid\text{lp}}$]
    Consider an étale morphism $\Spf(B)\to \Spf(A)$ of formal schemes. Using \ref{etale maps as pullbacks of etale maps} together with the fact that faithfully flat étale maps of animated rings are 
    descendable \cite{MathewGaloisgroupsofstablehomotopy}[3.33], we see that 
       $$B\in \QCoh(\Spf(A))$$
    is a descendable algebra. Now, as $(\_)^{\Solid,\text{lp}}$ sends maps of affine 
    formal schemes to proper analytic ring maps \ref{integral affine maps are proper}, we have to see that an $\text{éid}$-covering 
    gets sent to a descendable map of algebras \ref{descendable cover of analytic stacks}. But there is 
    a lax monoidal functor 
       $$\QCoh(\Spf(A)) \to \QCoh_{I\text{-comp}}(\QCoh(\Spa(A,\widetilde{\Znumb}))) \subset \QCoh(\Spa(A,\widetilde{\Znumb}))$$
    and such functors preserve descendable algebras \cite{BhattScholzeWittvectorGrassmannian}[11.20].
\end{proof}

The following will also prove the second paragraph of the proposition. 

\begin{proof}[Proof of \ref{From formal to analytic prop} \textbf{A} \text{and} \textbf{B} $(\_)^{\Solid}$]
    We first do the case of an integral descendable cover. Using the same argument as 
    in the proof for $(\_)^{\Solid\text{lp}}$, it is enough to check that $(\_)^{\Solid}$
    sends integral maps to proper maps of analytic rings. As for any solid affinoid 
    $A$, the topological nilpotent elements are contained in $A^{+}$ \cite{AStacksLecture}[Lecture 8], 
    it suffices to check that 
       $$\widetilde{A+ IB}=B$$
    where $I$ is some ideal of definition for $A$. But given a point $b\in B$, 
    by assumption, we can find a monic polynomial $f$ over $A$ such that 
       $$f(b) = i \in IB.$$
    This means that 
      $$b \in \widetilde{A[i]+IB}=\widetilde{A+IB}$$
    and we win. 

    For the case of an étale morphism, consider a cartesian square as in \ref{etale maps as pullbacks of etale maps} and 
    write $\Spa(A)^{\text{disc}}$ for the discrete adic space associated to the underlying discrete ring 
    of $A$ (and similar for $B$). Then we have a commutative square
\[\begin{tikzcd}
	{\Spa(B)} & {\Spa(A)} \\
	{\Spa(B')^{\text{disc}}} & {\Spa(A)^{\text{disc}}}
	\arrow[from=1-1, to=1-2]
	\arrow[from=1-1, to=2-1]
	\arrow[from=1-2, to=2-2]
	\arrow["j"', from=2-1, to=2-2].
\end{tikzcd}\]
   Now it is proven in \cite{CondensedMath}[11.4, 11.6] that the lower horizontal map 
   is étale. In particular, the functor $j^{\ast}$ commutes with limits, which implies that 
   the square pictured above is cartesian. This implies that the upper horizontal map 
   is étale as well, and to see that it is an epimorphism in analytic stacks (in the case the map 
   of formal schemes was surjective), we have to check that the functor $\QCoh^{\ast}$ descends 
   along the map $j$. This is shown in \cite{LucasMannPhD}[2.10.6]. 
\end{proof}

\begin{proof}[Proof of \ref{From formal to analytic prop} \textbf{C}]
   We first claim that the map 
      $$(X/S)^{\Solid\text{lp}} \to S^{\Solid}$$
   is co-smooth. 
   
   Using the descendable cover, we obtain a proper descendable cover 
      $$(\Spf(B)/S)^{\Solid\text{lp}} \to (X/S)^{\Solid\text{lp}}$$
   of analytic stacks. By \ref{stability properties of covers in SIXF}[(c)] it now suffices 
   to see that the composition 
      $$(\Spf(B)/S)^{\Solid\text{lp}} \to S^{\Solid}$$
   is co-smooth. This is even weakly cohomologically proper \ref{integral affine maps are proper}.

    Under the \textsuperscript{$+$}finite type assumption, we claim that the map 
       $$X^{\Solid}\to (X/S)^{\Solid\text{lp}}$$
    is an open immersion and thus $!$-able. This implies the $!$-ability claim. 
    To see this, we can work locally on the target and thus assume 
       $$X\iso \Spf(B) \to \A^{n}_{S} \to S$$
    is affine and factors as an integral map followed by the projection from some affine space. 
    In this case, the map in question is obtained by solidifying the coordinates and thus an open 
    immersion. 
\end{proof}

To prove assertion \textbf{D}, we will need to associate a topological space 
to a formal stack. For this recall from \ref{topological spaces of a Huber pair} that, to an affine formal 
scheme $\Spf(B)$ over $\Spf(A)$, we can associate two topological spaces 
   $$|\Spa(B)| \text{ and } |\Spa(B,\widetilde{A})|.$$

We will need the following lemma. 

\begin{lem}\label{submersive lemma}
   For any \text{éid}-cover $\Spf(C)\to \Spf(B)$ of affine formal schemes over an affine formal scheme
   $\Spf(A)$, the induced 
   maps 
     $$|\Spa(C)| \to |\Spa(B)|  \text{ and } |\Spa(C,\widetilde{A})| \to |\Spa(B,\widetilde{A})|$$
   are submersive. 
\end{lem}

\begin{proof}
   We have already seen in \textbf{A} and \textbf{B} that the upper-$\ast$-functors 
   on these analytic stacks are conservative. This implies that the maps in question are surjective. 
   Now such an éid-cover is given by a finite composition of either étale covers or 
   integral descendable covers. Thus by \cite[\href{https://stacks.math.columbia.edu/tag/02YB}{Tag 02YB}]{StacksProject} 
   and \cite[ \href{https://stacks.math.columbia.edu/tag/0AAU}{Tag 0AAU}]{StacksProject} it suffices to show that each of those 
   is either open (in the case of an étale cover) or closed (in the case of an integral map). 
   The case of an étale map is \cite{HuberAdicSpaces}[1.7.9], but we will give a similar argument in the 
   case of an integral map. For this case, we will explain the map 
      $$|\Spa(C)| \to |\Spa(B)|$$
   as the argument for the other map is the same. 

   Note first that up to completion, we can write $C$ as a filtered colimit of 
   finite $B$-algebras. Thus, as the adic spectrum just depends on the completion, we can assume 
   that $B\to C$ is integral. Now as in \cite{OnValuationHub}, we write $\Spv(B)$
   for the spectrum of all valuations on $B$ and similar for $C$. Then by \cite{OnValuationHub}[2.1.7.(ii)]
   the map $\Spv(C) \to \Spv(B)$ is closed. Furthermore, let us write 
      $$\Spv(B)^{+} \text{ and } \Spv(B,B^{\circ\circ})^{+}$$
   for the subspace of those valuations which value any function of $B$ as $\le 1$ and the 
   subspace of those valuations, which, in addition, value any topological nilpotent 
   element as $< 1$. Then we have a Cartesian square 
\[\begin{tikzcd}
	{\Spv(C)^{+}} & {\Spv(C)} \\
	{\Spv(B)^{+}} & {\Spv(B)}
	\arrow[from=1-1, to=1-2]
	\arrow[from=1-1, to=2-1]
	\arrow[from=1-2, to=2-2]
	\arrow[from=2-1, to=2-2]
\end{tikzcd}\]
   which shows that the left vertical map is also closed. Furthermore the subspace 
      $$\Spv(B,B^{\circ\circ})^{+}\subset \Spv(B)^{+}$$
   is closed (and similar for $C$), which shows that also the map 
   $\Spv(C,C^{\circ\circ})^{+} \to \Spv(B,B^{\circ\circ})^{+}$ is closed. By \cite{HuberContVal}[2.6]
   there exist retractions making the following diagram commute 
\[\begin{tikzcd}
	{\Spa(C)} & {\Spv(C,C^{\circ\circ})^{+}} & {\Spa(C)} \\
	{\Spa(B)} & {\Spv(B,B^{\circ\circ})^{+}} & {\Spa(B)}
	\arrow["{i_{C}}", from=1-1, to=1-2]
	\arrow["f"', from=1-1, to=2-1]
	\arrow["{r_{C}}", from=1-2, to=1-3]
	\arrow["g", from=1-2, to=2-2]
	\arrow["f", from=1-3, to=2-3]
	\arrow["{i_{B}}"', from=2-1, to=2-2]
	\arrow["{r_{B}}"', from=2-2, to=2-3].
\end{tikzcd}\]
   From this we see that for any closed $Z\subset \Spa(C)$, we can write 
      $$f(Z)=g(r_{C}^{-1}(Z))\cap\Spa(B)$$
   which is closed. 
\end{proof}

\begin{constr}
   Using \ref{submersive lemma}, we see that associating to an affine formal scheme $\Spf(B)$
   over $\Spf(A)$ the topological spaces $|\Spa(B)|$ and $|\Spa(B,\widetilde{A})|$ induces colimit 
   preserving functors 
      $$|(\_)^{\Solid}| \text{ and } |(\_)^{\Solid\text{lp}}| \from A_{\text{éid}} \to \Top$$.
\end{constr}

\begin{non}\label{From smashing spectrum to topological space}
   Recall from \ref{topological spaces of a Huber pair} that for any affine formal scheme $\Spf(B) \to \Spf(A)$, we had maps 
   of locales 
     $$\SmSpec(\QCoh(\Spa(B))) \to |\Spa(B)| \text{ and } \SmSpec(\QCoh(\Spa(B,\widetilde{A}))) \to |\Spa(B,\widetilde{A})|.$$
   Using those maps, we obtain maps of locales 
      $$\SmSpec(\QCoh(X^{\Solid})) \to |X^{\Solid}| \text{ and } \SmSpec(\QCoh(X^{\Solid,\text{lp}}))\to |X^{\Solid\text{lp}}|$$
   for any formal stack $X$ over $\Spf(A)$. 
\end{non}

\begin{proof}[Proof of \ref{From formal to analytic prop} \textbf{D}]
   Arguing as in the proof of \textbf{C}, we can factor the map as 
      $$X^{\Solid} \to (X/S)^{\Solid} \to S^{\Solid}$$
   an open immersion followed by a proper map. So it suffices to see that the 
   open immersion is an isomorphism. This open immersion comes from an open immersion 
      $$|X^{\Solid}| \subset |(X/S)^{\Solid\text{lp}}|$$
   and we first claim that it suffices to show that the latter map is surjective. 
   Indeed, the open immersion $X^{\Solid}\to (X/S)^{\Solid\text{lp}}$ corresponds to 
   an idempotent algebra in $\QCoh((X/S)^{\Solid\text{lp}})$ and is an isomorphism if and 
   only if this idempotent algebra is $0$. But by the surjectivity assumption and 
   \ref{From smashing spectrum to topological space}, we can check this after pulling back 
   along the map itself, where the idempotent algebra vanishes by construction. 
   To see that this inclusion is surjective, we choose a descendable cover
      $$\Spf(B)\to X$$
   and pull the open subset back along the map 
     $$\Spa(B,\widetilde{A}) \to (X/S)^{\Solid\text{lp}}.$$
   This way we obtain an open subset $U\subset\Spa(B,\widetilde{A})$ which induces an 
   isomorphism on non-Tate points. The next lemma shows that any point in the target 
   specializes to such a point, which implies that it is an isomorphism itself. 
\end{proof}

We used the following lemma.

\begin{lem}
   Consider a ring $A$ which is $I$-adically complete for some finitely 
   generated ideal $I\subset A$. Furthermore, let $A^{+}\subset A$ be an 
   integrally closed subring that contains the generators of $I$. Then any point in 
   $\Spa(A,A^{+})$ specializes to a non-Tate point.
\end{lem}

\begin{proof}
   Let us be given a point $x\in \Spa(A,A^{+})$ corresponding to a valuation 
      $$x\from A \to \Gamma_{x}\cup\{0\}.$$
   Then we write $c\Gamma_{x} \subset \Gamma_{x}$ for the convex closure 
   of $\text{Im}(x)\cap \Gamma_{x,\ge 1}$ and consider the map 
      \begin{equation}
           x_{|c\Gamma_{x}} \from A \to c\Gamma_{x}\cup\{0\}, |f(x)| = 
           \begin{cases}
               |f(x)| & \text{ if } |f(x)| \in c\Gamma_{x} \\
               0 & \text{ if } |f(x)| \notin c\Gamma_{x}.
           \end{cases}
      \end{equation}
   As the subgroup is convex, this defines a valuation \cite{WedhornAdicSpacesLect}[4.15], and it clearly 
   values any function $f\in A^{+}$ with $\le 1$. As it sends any topological nilpotent 
   element to $0$, it is continuous and defines a non-Tate point. It is a 
   specialization of $x$ by \cite{WedhornAdicSpacesLect}[4.16]. 
\end{proof}

\begin{proof}[Proof of \ref{From formal to analytic prop} \textbf{E}]
   Using \textbf{D}, we can assume that the map 
      $$X^{\Solid} \to S^{\Solid}$$
   is a map of discrete adic spaces. For an affine discrete adic space $\Spa(A,A^{+})$, taking the 
   support 
      $$|\Spa(A,A^{+})| \to |\Spec(A)|$$
   is compatible with Zariski-closed and Zariski-open sets. So we can argue similarly to the proof of \textbf{D}. 
\end{proof}
\section{The solid syntomification}\label{solid Syntomic}


\subsection{The infinite roots topology}

We will introduce a topology on the category of $p$-adic formal schemes called the \emph{infinite $p$-root 
topology}. This topology will be weaker than the quasi-syntomic topology originally used in the theory of prismatic and 
syntomic cohomology \cite{Bhatt2018TopologicalHH} \cite{BhattLurieAPC}, but still strong enough to make most of the local to global arguments work. 
On the other hand, it will induce surjections of prismatic stacks on the level of analytic stacks.

By an affine $p$-adic formal scheme, we mean an affine (derived) formal scheme, which is 
adic over $\Spf(\Znumb_{p})$. The main example to have in mind is the following.

\begin{ex}\label{fundamental example of root cover}
  For any set $I$, we will call a map of the form 
     $$\Spf(\Znumb_{p}[x_{i}^{\frac{1}{p^{\infty}}}| i\in I] \to \Spf(\Znumb_{p}[ x_{i}| i\in I])$$
   a generic infinite $p$-root cover. 
\end{ex}

\begin{defn}\label{Definition of infinite root cover}
   A family of maps $\{\Spf(B_{i})\to \Spf(A)\}$ of $p$-adic formal schemes is called an \emph{infinite $p$-root cover}, if 
   it can be written as a finite composition of one of the following:
      \begin{itemize}
         \item[(1)] $\Spf(A)$ is the finite disjoint union of the $\Spf(B_{i})$.
         \item[(2)] The family is given by one map $\Spf(B) \to \Spf(A)$, which is pulled back from 
                    a generic infinite $p$-root cover \ref{fundamental example of root cover}.
      \end{itemize}
\end{defn}

\begin{lem}
   The collection of infinite $p$-root covers forms a Grothendieck pretopology on the 
   category of affine $p$-adic formal schemes. 
\end{lem}

\begin{proof}
   We have to see that the collection of those covers contains all identities and is stable under composition and 
   pullback. The first assertion is clear, and the others hold by definition.
\end{proof}

We will write $\Sheaves_{\sqrt{\infty}}(\fSch^{\text{aff}}_{\Znumb_{p}})$ for the induced topos. 

\begin{rem}\label{Remark on producing IFR sheaves}
   A presheaf on $p$-adic formal schemes is a sheaf for the infinite $p$-root topology, 
   if and only if it sends finite disjoint unions to products and satisfies the \v{C}ech condition 
   for maps pulled back from generic infinite $p$-root covers. 

   In the case we want to construct a map of topoi 
      $$\Sheaves_{\sqrt{\infty}}(\fSch^{\text{aff}}_{\Znumb_{p}})\to \mathcal{X}$$
   starting from a finite limit-preserving functor from (affine) $p$-adic formal schemes 
   to $\mathcal{X}$. The latter condition will simplify. Namely, we only have to check the 
   \v{C}ech condition on generic infinite $p$-root covers. 
\end{rem}

The following easy observation will be crucial for us.

\begin{lem}\label{fundamental properties of infinite root cover}
   Consider a map $\Spf(B) \to \Spf(A)$, which is a finite composition of maps 
   pulled back from generic infinite $p$-root covers. Then we have the following:
      \begin{itemize}
         \item[(1)] The map $A/p \to B/p$ is integral.
         \item[(2)] The cofibre of the map $A/p \to B/p$ is a projective $A/p$-module.
         \item[(3)] The shifted cotangent complex $\CoTan_{B/A}/p[-1]$ is projective.  
      \end{itemize}  
\end{lem}

\begin{proof}
   All assertions are stable under finite compositions and base change. So we can check them for 
   a generic infinite $p$-root cover. Then $(1)$ and $(2)$ are clear and $(3)$ follows, using the 
   fundamental sequence of the cotangent complex, from the observations that 
   $A/p \to B/p$ is the colimit perfection of the domain and that the cotangent complex 
   of perfect $\Fp$-algebras vanishes.
\end{proof}

We can also describe a basis for this topology. 

\begin{defn}
   A $p$-complete animated ring $S$ is called \emph{semiperfectoid}, if there exists a 
   $\pi_{0}$-surjection $R \to S$ from a perfectoid ring $R$. We will write 
      $$\SemPerfd_{\Znumb_{p}}^{\text{aff}}\subset \fSch_{\Znumb_{p}}^{\text{aff}}$$
   for the full subcategory of those affine $p$-adic formal schemes, whose functions form a 
   semiperfectoid ring. 
\end{defn}

\begin{rem}
   One natural way to define animated perfectoid rings is as quotients $A/I$ where 
   $I\to A$ is a perfect animated prism. Note that all such are static \cite{BhattLuriePrismatisationofpadicformalschemes}[2.16]. 
\end{rem}

\begin{lem}\label{Basis by semiperfectoids lemma}
   The category of affine semiperfectoid $p$-adic formal schemes forms a basis 
   of the infinite $p$-root topology on affine $p$-adic formal schemes. 
   In particular, restricting along the inclusion induces an equivalence 
   of topoi 
     $$\Sheaves_{\sqrt{\infty}}(\fSch^{\text{aff}}_{\Znumb_{p}})\iso \Sheaves_{\sqrt{\infty}}(\SemPerfd^{\text{aff}}_{\Znumb_{p}}).$$
\end{lem}

\begin{proof}
   This follows the same way as in \cite{Bhatt2018TopologicalHH}[4.28] using \ref{fundamental example of root cover}.
\end{proof}

\begin{non}
   We will write $\sqrt{\infty}^{\text{ét}}$-topology for the union of the infinite $p$-root topology 
   with the étale topology. We will consider this topology on the category of qcqs (derived) $p$-adic formal schemes. 
   That is, those formal schemes that are adic over $\Spf(\Znumb_{p})$. The corresponding topos will be 
   denoted by  
     $$\Spf(\Znumb_{p})_{\sqrt{\infty}^{\text{ét}}}.$$
   Clearly, this topos is also generated by semiperfectoid $p$-adic formal schemes. 
\end{non}


\subsection{$\Witt$-modules} 

We will collect some examples of $\Witt$-modules and recall some facts 
about them. 

\begin{non}
    Note that associating to an affine $p$-adic formal scheme its Witt-vectors 
    defines a sheaf of ($p$-complete) animated rings on the éid-topology. 
    To see this, we can write 
       $$\Witt \iso \limit_{n}\Witt_{n}$$
    as the limit of the truncated Witt-vectors and then recall that we have an isomorphism 
      $$\Witt_{n} \iso \A^{n}$$
    as presheaves of anima. So the claim follows as the éid-topology is subcanonical. 
\end{non} 

\begin{non}[\textbf{$\Gasharp$-bundles}]\label{Gasharp bundles}
   For an animated ring $R$ and a connective $R$-module $\E$, we will write 
      $$\Gamma_{R}(\E)$$
    for the \emph{animated divided power algebra} over $R$ associated to $\E$\footnote{This can be obtained 
    by animating the classical construction from pairs $(R,\E)$ where $R$ is a polynomial algebra in finitely many variables and 
    $\E$ a finite free $R$-module. The animation of the category of such pairs gives the total category of the fibration corresponding 
    to the functor which assigns to an animated ring its category of connective modules.}.
    We then will write 
       $$\V(\E)_{R}^{\#}:= \Spec(\Gamma_{R}(\E)).$$ 
\end{non}

\begin{ex}
   In the case the $\Gasharp$-bundle is associated to the trivial vector bundle of rank one, we obtain 
   the following.
   Let $\Gasharp$ be the PD-hull of the origin in $\Ga$ over $\Znumb$. Concretely we have 
      $$\Gasharp \iso \Spec(\Znumb[x,\frac{x^{2}}{2!},\frac{x^{3}}{3!},\dots]).$$
   Note that there is a map $\Gasharp \to \Ga$ and that the equalities 
     \begin{itemize}
      \item $\frac{(x+y)^{n}}{n!} = \sum_{i+j=n}\frac{x^{i}}{i!}\frac{y^{j}}{j!}$
      \item $\frac{(xy)^{n}}{n!}=\frac{x^{n}}{n!}\cdot y^{n}$
     \end{itemize}
   show that the additive group action on $\Ga$ induces an action on $\Gasharp$ and the the multiplicative 
   group action on $\Ga$ induces via the above map an action on $\Gasharp$ which makes it into a $\Ga$-module scheme. 
   In particular we can understand $\Gasharp$ as a $\Witt$-module scheme via the projection $\Witt \to \Ga$.
\end{ex}

The following in particular shows that $\Witt$ stays a sheaf if we let it 
take values in possibly non-connective spectra. 

\begin{lem}\label{Witt vecotors is an derived sheaf}
   Given an affine $p$-adic formal scheme $\Spf(S)$ and a finite projective $\Witt(S)$-module 
   $P$, we have 
      $$R\Gamma_{\text{éid}}(S,P\otimes_{\Witt(S)}\Witt)\in \QCoh(\Witt(S))_{\ge 0}$$
   and the same holds for the truncated Witt vectors. 
\end{lem}

\begin{proof}
   Writing $P$ as a summand of a finite free $\Witt(S)$-module this reduces to 
   showing $R\Gamma(S,\Witt)\in \QCoh(\Witt(S))_{\ge 0}$. Furthermore writing $\Witt$
   as a limit of the truncated Witt Vectors and observing that the transition maps 
   are surjective, the Milnor sequence tells us that this reduces to the claim 
       $$R\Gamma(S,\Witt_{n})\in \QCoh(\Witt(S))_{\ge 0}.$$
   This follows by induction on $n$ using the fiber sequences 
   $\Witt_{n-1}\to \Witt_{n} \to \Ga$ coming from the Verschiebung and the 
   corresponding claim for $\Ga$.  
\end{proof}

For an affine $p$-adic formal scheme $\Spf(S)$, we will write $$\QCoh_{\text{éid}}(\Spf(S),\Witt)$$ 
for the category of $\Witt$-modules in sheaves of spectra on the éid-topos of $\Spf(S)$. 

\begin{lem}\label{PD-hull via witt vectors}\cite{Drinfeld2020Prismatization}[3.4]\cite{FGauges}[2.6.1]\cite{BhattLurieAPC}[3.4.11]
   The Frobenius $F\from \Witt \to F^{\ast}\Witt$ is a $\pi_{0}$ surjection in the éid-topology. 
   Furthermore, if we write $\Witt[F]$ for the fibre of the Frobenius, the composition $\Witt[F] \subset \Witt \to \Ga$ 
   of the inclusion with the projection lifts uniquely to an isomorphism 
      $$\Witt[F]\iso \Gasharp.$$
   In particular, there is a fiber sequence
      $$\Gasharp \to \Witt \to F_{\ast}\Witt$$
   in $\QCoh_{\text{éid}}(\Spf(S),\Witt)$ where the second map is given by the Frobenius. 
\end{lem}

\begin{proof}
   The proof of \cite{FGauges}[2.6.1] shows that the Frobenius is faithfully flat; thus it is descendable by \cite{MathewGaloisgroupsofstablehomotopy}[3.31]
   as the domain is countable. It is also clearly integral. This shows $\pi_{0}$-surjectivity. 

   For the rest, note that the ordinary scheme represents $\Witt[F]$ also on derived schemes 
      $$\Spec(\Znumb_{(p)}[x_{0},x_{1},x_{2},\dots]/(x_{0}^{p}+px_{1},x_{1}^{p}+px_{2},\dots))$$
   and the proof of \cite{FGauges}[2.6.1] gives an isomorphism $\Witt[F]\iso \Gasharp$ of representing objects. 
\end{proof}

\begin{cor}\label{twistet witt vector sequences}\cite{BhattLuriePrismatisationofpadicformalschemes}[5.7]
   For any affine $p$-adic formal scheme $\Spf(S)$ and any finite projective $\Witt(S)$-module $P$ there are fibre sequences
      \begin{itemize}
         \item $R\Gamma_{\text{éid}}(\Spf(S),P\otimes_{\Witt}\Witt_{n}[F]) \to P\otimes_{\Witt}\Witt_{n}(S) \to P\otimes_{\Witt}F_{\ast}\Witt_{n-1}(S)$ 
         $(n\ge 2).$
         \item $R\Gamma_{\text{éid}}(\Spf(S),P\otimes_{\Witt}\Witt[F]) \to P\otimes_{\Witt}\Witt(S) \to P\otimes_{\Witt}F_{\ast}\Witt(S)$
      \end{itemize}
   in $\QCoh(\Witt(S))$. 
\end{cor}

\begin{proof}
   We claim that these fiber sequences arise as the global sections of the fiber sequence 
      $$\Witt[F] \to \Witt \to F_{\ast}\Witt$$
   tensored with $P$ (and the analogous fiber sequence with the truncated Witt Vectors). 
   That is the essential claim is that $R\Gamma(\Spf(S),P\otimes_{\Witt(S)}\Witt) \in \QCoh(\Witt(S))_{\ge 0}$.
   This was explained in \ref{Witt vecotors is an derived sheaf}. 
\end{proof}

\begin{cor}
    For any affine $p$-adic formal scheme, there is a commutative diagram 
\[\begin{tikzcd}
	& {F_{\ast}\Witt} & {F_{\ast}\Witt} \\
	{\Witt[F]} & \Witt & {F_{\ast}\Witt} \\
	\Gasharp & \Ga & {\GadR\iso F_{\ast}\Witt/p}
	\arrow[from=1-2, to=1-3, equal]
	\arrow["V"', from=1-2, to=2-2]
	\arrow["p", from=1-3, to=2-3]
	\arrow[from=2-1, to=2-2]
	\arrow[from=2-1, to=3-1, equal]
	\arrow["F", from=2-2, to=2-3]
	\arrow[from=2-2, to=3-2]
	\arrow[from=2-3, to=3-3]
	\arrow[from=3-1, to=3-2]
	\arrow[from=3-2, to=3-3]
\end{tikzcd}\]
   in $\QCoh_{\text{éid}}(\Spf(S),\Witt)$. Such that the square in the middle 
   is cartesian and the horizontal and vertical sequences are (co)fiber sequences. 
\end{cor}

\begin{proof}
    By \ref{PD-hull via witt vectors}, we only have to give the identification in the lower right 
    corner. This is explained in \cite{FGauges}[2.6.8]. 
\end{proof}

For the following we consider a static $p$-nilpotent ring $S$ and recall the topos $S_{\text{éfid}}$
from \ref{flat EID topos}.

\begin{lem}\label{Basis of witt vector surjectives}
   Let $S$ be a static $p$-nilpotent ring. Then the flat éid-topology on $S$ admits a basis of rings $T$, 
   such that the maps  
      $$F \from \Witt_{n}(T) \to F_{\ast}\Witt_{n-1}(T)$$
   are surjective for all $n\ge 2$.
\end{lem}

\begin{proof}
   In \cite{DavisKedlayaWittvectorFrobenious}[3.2] the authors explain that for this 
   surjectivity to hold it is enough that the Frobenius $F\from \Witt(T)\to F_{\ast}\Witt(T)$
   on the whole Witt Vectors hits all multiplicative lifts $[t]$ of elements 
   $t\in T$. As we have the formula $F([t])=[t^{p}]$ this is true for 
   any ring admitting all $p$-th roots of its elements. 

   Now consider a static $p$-nilpotent $S$-algebra $S_{0}$. Then we consider the 
   countable chain 
      $$S_{0} \to S_{1} \to S_{2} \to \cdots$$
    where the maps $S_{n} \to S_{n+1}$ are obtained by freely adjoining $p$-th roots of 
    all elements. Write $T=\colimit_{n}S_{n}$; then the map $S_{0} \to T$ is clearly 
    flat and integral, and any element in $T$ admits a $p$-th root. To see that it is 
    descendable, note that for each $n$ the fiber of the map $S_{n} \to S_{n+1}$ takes the 
    form $\oplus_{I}S_{n}[-1]$ for some set $I$. From this one easily sees that the 
    map $S_{n} \to S_{n+1}$ is descendable of index $\le 0$ \cite{BhattScholzeWittvectorGrassmannian}[11.18], such that 
    the map $S_{0} \to T$ becomes descendable by \cite{BhattScholzeWittvectorGrassmannian}[11.22].  
\end{proof}


\subsection{Recollections on prismatic stacks}

We will make use of the stacky approach to prismatic and syntomic 
cohomology invented by Bhatt-Lurie \cite{BhattLurieAPC}\cite{BhattLuriePrismatisationofpadicformalschemes} and Drinfeld \cite{Drinfeld2020Prismatization}. 
More references and details on this theory can be found in \cite{FGauges} \cite{Gardner2024AnAC} \cite{Hauck2025AnA}. 

Let us recall some definitions. We will interpret the following objects as éid-sheaves (with values in anima)
on affine $p$-adic formal schemes. Having said this, we will describe the functor of points on $p$-nilpotent rings 
and then extend the presheaves by left Kan-extension. As an example, we obtain the formula 
   $$\Ga(\Spf(A))=A_{\widehat{p}}.$$

\begin{non}[\textbf{The de Rham stack}]
    Recall the ring stack, which assigns to a $p$-nilpotent ring $R$ the animated ring 
       $$\GadR(R):=\Ga(R)/\Gasharp(R).$$
    We then define the \emph{de Rham stack} of a $p$-adic formal scheme $X$ via transmutation 
    along this ring stack. That is, the anima of $R$-valued points is given by 
       $$X^{dR}(R):= X(\Spec(\GadR(R))).$$
\end{non}

\begin{non}[\textbf{The Hodge stack}]
    Recall the ring stack over $\BGm$, which assigns to a $p$-nilpotent animated ring $R$ with a  
    line bundle $\Line$ on $\Spec(R)$
    the square zero extension 
       $$\GaHodge(\Spec(R)\to \BGm):=R\oplus \BV_{R}(\Line^{-1})^{\sharp}(R).$$
    For a $p$-adic formal scheme $X$, we then define the \emph{Hodge stack} via transmutation. That is 
    the anima of $R$-valued points for a $p$-nilpotent animated ring $R$ is given by 
       $$X^{Hodge}(\Spec(R)\to \BGm):= X(\Spec(\GaHodge(R))).$$
\end{non}

\begin{non}[\textbf{The Prismatisation}]
    For a $p$-nilpotent animated ring $R$, we will write 
       $$\Znumb_{p}^{\Prism}(R)$$
    for the anima of \emph{Cartier-Witt divisors} on $R$ as defined in \cite{BhattLuriePrismatisationofpadicformalschemes}[8.2].
    Then recall the ring stack over $\Znumb_{p}^{\Prism}$, which assigns
    to a Cartier-Witt divisor $\alpha\from I \to \Witt(R)$ on $R$, the animated 
    ring 
       $$\GaPrism(\Spec(R) \to \Znumb_{p}^{\Prism}):=\Witt(R)/I.$$ 
    To define the \emph{Prismatisation} of a general $p$-adic formal scheme $X$, we again 
    use transmutation. It assigns to a Cartier divisor on $R$ the anima 
       $$X^{\Prism}(\Spec(R) \to \Znumb_{p}^{\Prism}):=X(\Spec(\GaPrism(R))).$$
\end{non}

We already obtained the following.

\begin{cor}\label{Prismatiastion of the affine line is a derived sheaf}
   For a Cartier-Witt divisor $I \to \Witt(S)$, we have 
      $$R\Gamma_{\text{éid}}(S,\GaPrism) \in \QCoh(\Witt(S))_{\ge 0}$$
   and the latter admits the structure of an animated ring. 
\end{cor}

\begin{proof}
   The first claim follows from \ref{Witt vecotors is an derived sheaf} and the cofiber sequence
      $$I\otimes_{\Witt(S)}\Witt \to \Witt \to \overline{\Witt}$$
   of sheaves on $(S)_{\text{éid}}$ with values in $\QCoh(\Witt(S))$. For the second claim, note that 
   the Cartier-Witt divisor can be seen as a map $\Spec(\Witt(S)) \to \A^{1}/\Gm$
   and we can pull back this map to $\BGm$. 
\end{proof}

\begin{ex}\label{Prismatisation of a perfectoid}
   Given an animated prism $I\to A$, there is a map 
      $$\Spf(A) \to \overline{A}^{\Prism}$$
   which is constructed as follows. A map $A \to S$, by adjunction, uniquely refines to a map 
   $A\to \Witt(S)$ of $\delta$-rings. Then $I\otimes_{A}\Witt(S) \to \Witt(S)$ defines a Cartier-Witt divisor and we obtain a map $\overline{A} \to \Witt(S)/I$ by pushing out 
   along $A\to \overline{A}$.
   As explained in \cite{BhattLuriePrismatisationofpadicformalschemes}[3.12], in the case the prism is perfect with 
   corresponding perfectoid $\overline{A}\iso R$ this construction gives an isomorphism of functors
     $$\Spf(A) \iso R^{\Prism}.$$ 
\end{ex}

\begin{ex}\label{presentation of prismatisation}
   The moduli of Cartier-Witt divisors is essentially the moduli of animated prisms \cite{BhattLuriePrismatisationofpadicformalschemes}[8.3]. 
   An important example of a prism is the universal oriented prism. For an animated $p$-nilpotent 
   ring $R$ we write $\Witt_{0}(R) \subset \Witt(R)$ for the subspace of those Verschiebung 
   expansions $\sum_{n\ge 0}V^{n}[a_{n}]$ for which $a_{0}$ is nilpotent and $a_{1}$ is a unit. 
   Note that, as a functor, this assignment is represented by an affine formal scheme 
      $$\Spf(\Znumb_{p}\langle a_{0},a_{1}^{\pm 1},a_{2},\dots\rangle)$$
   where we also complete to $(a_{0})$, and that the Frobenius on the Witt Vectors restricts to this subspace, such that we obtain a 
   $\delta$-structure on the representing ring. Thus choosing $(a_{0})$ as an ideal we have produced a 
   prism and as in \cite{BhattLurieAPC}[3.2.4] we get a map 
      $$\Witt_{0} \to \Znumb_{p}^{\Prism}.$$
   By (the proof of) \cite{BhattLurieAPC}[3.2.3] (see also \cite{BhattLuriePrismatisationofpadicformalschemes}[8.5]), this map identifies the target as the quotient, 
   in the Zariski topology, of the domain by the canonical 
   action of the affine group scheme representing the functor $R\mapsto \Witt(R)^{\times}$. Where we write $\Witt(R)^{\times}$
   for the units in the Witt Vectors. This group scheme is represented 
   by the free delta ring on a unit $\Znumb_{p}\{u^{\pm}\}$.
   
   In particular, we see that we have 
   found a presentation of 
      $$\Znumb_{p}^{\Prism}\in (\Znumb_{p}^{\Prism})_{\text{éid}}.$$
\end{ex}

\begin{vari}\label{Relative prismatisation definition}
   As pointed out in \ref{presentation of prismatisation}, for any animated prism $I\to A$
   there is a map 
      $$\Spf(A) \to \Znumb_{p}^{\Prism}.$$
   Pulling back the ring stack $\GaPrism$ along this map gives a ring stack 
      $$\GaPrismA \to \Spf(A)$$
   which via transmutation defines the relative Prismatisation over $A$ for 
   $p$-adic formal schemes over $\Spf(A/I)$.
\end{vari}

\begin{non}[\textbf{The Hodge-Tate stack}]\label{HodgeTate stack}
    Base changing a Cartier-Witt divisor on a $p$-nilpotent ring $R$ along the map $\Witt(R) \to R$ produces a map 
       $$\Znumb_{p}^{\Prism} \to \AHadmodGm.$$
    The locus over $\BGm$ is called the \emph{Hodge-Tate} locus and will be denoted by 
        $$X^{HT}.$$
    One can also define this stack via transmutation using the ring stack 
       $$\GaHT := \GaPrism_{|\Znumb_{p}^{HT}}$$
    over $\Znumb_{p}^{HT}$ which we will investigate in more detail later. 
\end{non}

\begin{non}[\textbf{The Breuil-Kisin twist on tha Prismatisation}]\label{BK twist non analytic}
   Let us write $\Oh(-1) \to \Oh$ for the universal generalized Cartier divisor on $\A^{1}/\Gm$. Then pulling 
   back the line bundle $\Oh(-1)$ via the map 
      $$\Znumb_{p}^{\Prism} \to \widehat{\A^{1}}/\Gm$$
   defines a line bundle on $\Znumb_{p}^{\Prism}$ which we will denote by $\Oh\{1\}$ an call the 
   \emph{Breuil-Kisin twist}. We will denote several pullbacks of this line bundle in the same way. 
\end{non}

\begin{non}[\textbf{Filtered Cartier-Witt divisors}]
    The moduli of Nygaard filtered Cartier-Witt divisors can be described as 
    the following pullback (see \cite{Hauck2025AnA}[2.3.9])  
\[\begin{tikzcd}
	{\Znumb_{p}^{\Nyg}} & {\Znumb_{p}^{\Prism}} \\
	{\AmodGm\times(\AmodGm)^{dR}} & {(\AmodGm)^{dR}}
	\arrow["\pi", from=1-1, to=1-2]
	\arrow["{(t,u)}"', from=1-1, to=2-1]
	\arrow["{\tilde{\mu}}", from=1-2, to=2-2]
	\arrow["m"', from=2-1, to=2-2]
\end{tikzcd}\] 
    where the maps can be described as follows.
       \begin{itemize}
        \item The map $m$ takes a pair of generalized Cartier divisors $(t\from\Line_{1}\to R, u\from\GadR(R)\to \Line_{2})$ 
              to the generalized Cartier divisor 
                 $$\Line_{1}\otimes_{R}\Line_{2} \to \GadR(R)$$
            obtained by base changing $t$ along the map $R\to \GadR(R)$ and then precomposing 
            with $u^{-1}$. 
        \item To describe the map $\tilde{\mu}$, one starts with a Cartier-Witt divisor $I \to \Witt(R)$ and considers 
              its frobenius pullback $F_{\ast}I \to F_{\ast}\Witt(R)$. Base changing the latter 
              along the map 
                 $$F_{\ast}\Witt(R) \to F_{\ast}\Witt(R)/p\iso \GadR(R)$$
            (see \cite{FGauges}[2.6.8] for the used isomorphism) gives a generalized Cartier divisor 
                $$F_{\ast}I\otimes_{F_{\ast}\Witt(R)}\GadR(R) \to \GadR(R).$$
       \end{itemize}
    We will refer to $\pi$ as the \emph{structure map} and to $t$ as the \emph{Rees map}.    

    Thus for a $p$-nilpotent animated ring $R$, a Nygaard filtered Cartier-Witt divisor 
    $\Spec(R)\to \Znumb_{p}^{\Nyg}$ is given by a triple 
       $$(\Line_{1}\to R,\GadR(R)\to \Line_{2},I\to \Witt(R))$$
    consisting of two generalized Cartier divisors and a Cartier-Witt divisor, together with an isomorphism 
       $$(F_{\ast}I\otimes_{F_{\ast}\Witt(R)}\GadR(R)\to \GadR(R))\iso(\Line_{1}\otimes_{R}\Line_{2}\to \GadR(R))$$
    of generalized Cartier divisors. 
\end{non}

\begin{non}[\textbf{The Nygaard filtered Prismatisation}] 
    Out of the data of a filtered Cartier-Witt divisor, we can construct a Cartesian square of filtered animated 
    rings
\[\begin{tikzcd}
	{Fil^{\bullet}_{M}\Witt(R)} & {Fil^{\bullet}_{F_{\ast}I}F_{\ast}\Witt(R)} \\
	{Fil^{\bullet}_{\Line_{1}}R} & {Fil^{\bullet}_{\Line_{1}}\GadR(R)}
	\arrow[from=1-1, to=1-2]
	\arrow[from=1-1, to=2-1]
	\arrow[from=1-2, to=2-2]
	\arrow[from=2-1, to=2-2]
\end{tikzcd}\]
    where the lower horizontal map is the canonical one and the right vertical map uses 
    the identification in the data of a filtered Cartier-Witt divisor \cite{Hauck2025AnA}[2.3.12].
    Thus we obtain a (co)fiber sequence 
       $$M(R) \to \Witt(R) \to \GaNyg(R)$$
     with the cofibre carrying the structure of an animated ring. This defines a ring 
     stack 
         $$\GaNyg \to \Znumb_{p}^{\Nyg}$$ 
     over the moduli of filtered 
     Cartier-Witt divisors. 
     Finally for a $p$-adic formal scheme $X$, we define the \emph{Nygaard filtered Prismatisation} via 
     transmutation, by assigning a Nygaard filtered Cartier-Witt divisor to the anima 
       $$X^{\Nyg}(\Spec(R)\to \Znumb_{p}^{\Nyg}):= X(\Spec(\GaNyg(\Spec(R)\to \Znumb_{p}^{\Nyg}))).$$
\end{non}

\begin{non}\label{Prismatisation in Nyg filtered Prismatiastion}
     The Prismatisation sits openly inside the Nygaard filtered prismatisation in two disjoint ways 
\[\begin{tikzcd}
	{X^{\Prism}} & {X^{\Nyg}} & {X^{\Prism}}
	\arrow["{j_{dR}}", from=1-1, to=1-2]
	\arrow["{j_{HT}}"', from=1-3, to=1-2].
\end{tikzcd}\]
     Here the open immersion on the left comes from pulling back along $\Gm/\Gm \to \A^{1}/\Gm$ and the open 
     immersion on the right comes from pulling back along $(\Gm/\Gm)^{\text{dR}}\to (\A^{1}/\Gm)^{\text{dR}}$.
     Here composing $j_{\text{dR}}$ with the structure map recovers the identity, while composing 
     $j_{\text{HT}}$ with the structure map recovers the Frobenius.
\end{non}

\begin{ex}\label{Nygaard filtered prismatisation of a semiperfectoid example}
   Given a filtered animated ring $Fil^{\bullet}A$, we can associate to it the so called 
   \emph{Rees stack}
      $$\Rees(Fil^{\bullet}A):= \Spec(\oplus_{i\in \Znumb}Fil^{i}A \cdot t^{-1})/\Gm \to \A^{1}/\Gm$$
   where the $\Gm$-action comes from the grading. Then one obtains an equivalence 
      $$\QCoh(\Rees(Fil^{\bullet}A))\iso \Mod_{Fil^{\bullet}A}(\QCohF(A))$$
   where on the right we denote by $\QCohF(A)$ the category of filtered objects in $\QCoh(A)$ \cite{MoulinosGeomofFiltr}.
   
   Now given a semiperfectoid $S$, by \cite{FGauges}[5.5.10] there is a canonical identification 
      $$\Rees(Fil^{\bullet}_{\Nyg}\Prism_{S}) \iso R^{\Nyg}$$
   where we write $Fil^{\bullet}_{\Nyg}\Prism_{S}$ for the absolute Nygaard filtration \cite{BhattLurieAPC}[5.5]. 
\end{ex}

\begin{ex}\label{Nygaard stratification on perferctoids example}
   Continuing \ref{Nygaard filtered prismatisation of a semiperfectoid example}, let us assume 
   $R$ is integral perfectoid and $A\iso \Ainf(R)$. Then we have an isomorphism of graded rings 
      $$\oplus_{i \in \Znumb}Fil^{i}_{\Nyg}A\cdot t^{-1}\iso A[t,u]/(tu-\varphi^{-1}(d))$$
   where $d$ is a generator of $I$ and $u=\varphi^{-1}(d)t^{-1}$ sits in degree $-1$.
   
   From this description, we obtain the following Zariski stratification:
      \begin{itemize}
         \item[(a)] The locus $(R^{\Nyg})_{t\neq 0}$ corresponds to taking the colimit of the 
                    filtration and thus recovers $\Spf(A)\iso R^{\Prism}$.
         \item[(b)] The locus $(R^{\Nyg})_{t=0,u\neq 0}$ identifies with 
                    $\Spf(A/\varphi^{-1}(d))\iso \Spf(\varphi_{\ast}R)$ via the filtered Frobenius \cite{FGauges}[5.5.6].
         \item[(c)] The locus $(R^{\Nyg})_{u=t=0}$ identifies with $\BGm$ over $\Spf(\varphi_{\ast}R)$.
      \end{itemize}
\end{ex}

\begin{ex}\label{stratification of Nygaard filtered Prismatisation}
   We further continue on \ref{Nygaard stratification on perferctoids example} and describe the ring stack 
   $\GaNyg$ over a perfectoid $R$. That is, in order to describe the pullback of the latter ring stack 
   to $R^{\Nyg}$, we can consider the following diagram of $A$-module schemes 
\[\begin{tikzcd}
	{\varphi^{-1}(I)\otimes_{A}\Gasharp} & {\varphi^{-1}(I)\otimes_{A}\Witt} & {F_{\ast}(I\otimes_{A}\Witt)} \\
	{\V(\Line_{2})^{\sharp}} & M & {F_{\ast}(I\otimes_{A}\Witt)} \\
	\Gasharp & \Witt & {F_{\ast}\Witt}
	\arrow[from=1-1, to=1-2]
	\arrow[from=1-1, to=2-1, "u^{\sharp}"']
	\arrow[from=1-2, to=1-3]
	\arrow[from=1-2, to=2-2]
	\arrow[from=1-3, to=2-3, equal]
	\arrow[from=2-1, to=2-2]
	\arrow[from=2-1, to=3-1, "t^{\sharp}"']
	\arrow[from=2-2, to=2-3]
	\arrow[from=2-2, to=3-2]
	\arrow[from=3-1, to=3-2]
	\arrow[from=3-2, to=3-3]
   \arrow[from=2-3, to=3-3]
\end{tikzcd}\] 
   where the horizontal compositions are fiber sequences and the vertical compositions are the 
   canonical maps. Then we obtain 
       $$\Witt/M \iso \GaNyg.$$
   
   Thus, further restricting to the stratification described in \ref{Nygaard stratification on perferctoids example}
   we obtain: 
      \begin{itemize}
         \item[(a)] Over $(R^{\Nyg})_{t\neq 0}$ we obtain 
                         $$\GaNyg \iso F_{\ast}(\Witt/I)\iso \varphi^{\ast}\GaPrism$$
                     as one sees by taking cofibres of the lower vertical maps. 
         \item[(b)] Over $(R^{\Nyg})_{t=0,u\neq 0}$, we obtain 
                         $$\GaNyg \iso \Witt/\varphi^{-1}(I)\iso \varphi^{\ast}\GaHT.$$
         \item[(c)] Over $(R^{\Nyg})_{u=t=0}$ we first observe that 
                         $$M\iso \V(\Line_{2})^{\sharp}\oplus F_{\ast}(I\otimes_{A}\Witt).$$
                    Furthermore, as we live over the Hodge-Tate locus, the Cartier-Witt divisor is 
                    locally given by $V(1)\from \Witt \to \Witt$. Such that via this 
                    identification the map $M\to \Witt$ sits in a pushout square 
\[\begin{tikzcd}
	{\V(\Line_{2})^{\sharp}\oplus F_{\ast}\Witt} & \Witt \\
	{\V(\Line_{2})^{\sharp}} & \Ga
	\arrow["{(0,V)}", from=1-1, to=1-2]
	\arrow["pr"', from=1-1, to=2-1]
	\arrow["pr", from=1-2, to=2-2]
	\arrow["0"', from=2-1, to=2-2].
\end{tikzcd}\]
                    From this we learn that over this locus we have an identification 
                       $$\GaNyg\iso \Ga\oplus\BV(\Line_{2})^{\sharp}\iso \GaHodge.$$
                      
      \end{itemize}

      Finally, unwinding what we did for a map $X\to \Spf(R)$ of $p$-adic formal schemes, we learn 
      that we obtain a stratification:
         \begin{itemize}
            \item[(a)] $(X^{\Nyg})_{t\neq0} \iso X^{\Prism}$ over $\Spf(A)$.
            \item[(b)] $(X^{\Nyg})_{t=0,u\neq 0}\iso \varphi_{\ast}X^{\text{HT}}$ over $\Spf(\varphi_{\ast}R)$.
            \item[(c)] $(X^{\Nyg})_{u=t=0}\iso \varphi_{\ast}X^{\text{Hodge}}$ over $\BGm_{\Spf(\varphi_{\ast}R)}$.
         \end{itemize}
\end{ex}


\subsection{Recollections on the Hodge-Tate locus}

One can understand the geometry of the Hodge-Tate stack \ref{HodgeTate stack} quite well. In order to 
make use of this later, we need to recall some aspects of these stacks. For this, we will essentially copy 
\cite{BhattLuriePrismatisationofpadicformalschemes}[5], but taking care that we work in a slightly different topology. 

\begin{non}\label{Classifying spcae via witt wectors}
   Let us write 
      $$\Gasharp\{1\} \in (\Znumb_{p}^{HT})_{\text{éid}}$$
   for the sheaf which takes a Hodge-Tate divisor $I \to \Witt(S)$ on 
   $\Spf(S)$ to $\Gasharp(S)\{1\}$ where the Breuil-Kisin twist was defined in \ref{BK twist non analytic}. 
   This defines a group object 
   and we will write $\BGasharp\{1\}$ for the sheaf of torsors on it. Not that using 
   \ref{twistet witt vector sequences} and \ref{PD-hull via witt vectors} we obtain an identification 
      $$\BGasharp\{1\} \iso R\Gamma_{\text{éid}}(\_,\Witt[F]\{1\}[1])$$
   where the twist $\{1\}$ on the right is defined analogously. 
\end{non}

Let us write $\GaHT$ for the ring stack over $\Znumb_{p}^{HT}$ obtained by 
pulling back $\GaPrism$. The following, then, is an absolute version of \cite{BhattLuriePrismatisationofpadicformalschemes}[5.10]
with the same proof. 

\begin{prop}\label{HodgeTate as square zero extension}
   The ring stacks $\GaHT$ is a square-zero extension of $\Ga$ by $\BGasharp\{1\}[-1]$. That is 
   there exists a natural $\Witt$-linear derivation $\partial\from  \Ga \to\Ga \oplus \BGasharp\{1\}$ fitting into a 
   cartesian square 
\[\begin{tikzcd}
	{\GaHT} & \Ga \\
	\Ga & {\Ga\oplus\BGasharp\{1\}}
	\arrow[from=1-1, to=1-2]
	\arrow["\pi^{HT}"', from=1-1, to=2-1]
	\arrow["\partial", from=1-2, to=2-2]
	\arrow["{\partial_{triv}}"', from=2-1, to=2-2]
\end{tikzcd}\]
   of ring stacks over $\Znumb_{p}^{HT}$.  
\end{prop}

\begin{proof}
   First, all ring stacks in the square define sheaves for the 
   éid-topology. For the lower right corner, this follows from the 
   identification \ref{Classifying spcae via witt wectors}, for the left upper corner 
   from \ref{Prismatiastion of the affine line is a derived sheaf}, and 
   for $\Ga$ as the topology is subcanonical. That means to show the claim 
   for the values on an animated Hodge-Tate divisor $\alpha \from I \to \Witt(S)$, we can 
   resolve $\Znumb_{p}^{HT}$ by (ordinary) prisms and show the claim on the induced 
   cover of $\Spf(S)$. Concretely, we can, for example, take the covering from \ref{presentation of prismatisation}
   then all objects appearing in the \v{C}ech nerve are represented by (ordinary) prisms 
   \cite{BhattLurieAPC}[3.2.8+3.2.10]. This reduces the claim to the situation relative to a prism, 
   in which case the argument is given in \cite{BhattLuriePrismatisationofpadicformalschemes}[5.10]. For the convenience of the reader, 
   we recall this argument now. 

   We fix a prism $(A, I)$ and let us assume that $p$ is a non-zero divisor in $\overline{A}$. 
   Using \ref{Classifying spcae via witt wectors} we can replace $\BGasharp\{1\}$
   by $R\Gamma(\_,\Witt[F]\{1\}[1])$ and by taking limits it suffices to prove the 
   corresponding claim for $R\Gamma(\_,\Witt_{n}[F]\{1\}[1])$ functorial in $n$. 
   
   We first prove this claim evaluated at a static $p$-nilpotent $\overline{A}$-algebra $S$ for which the Frobenius maps 
      $$F \from \Witt_{n}(S) \to F_{\ast}\Witt_{n-1}(S)$$
   are surjective for all $n\ge 2$. To see this, we first claim:
      \begin{itemize}
         \item[($\ast$)] The map $\alpha \from I\otimes_{\Witt(S)}\Witt_{n}(S) \to \Witt_{n}(S)$ maps surjectivity onto 
               $V\Witt_{n}(S)$. 
      \end{itemize}
   Using this claim, we can compute 
       $$\pi_{0}\overline{\Witt_{n}}(S)\iso \Witt_{n}(S)/V\Witt_{n}(S) \iso S$$
   and 
     $$\pi_{1}\overline{\Witt_{n}}(S)\iso ker(I\otimes_{\Witt(S)}\Witt_{n}(S)\to \Witt_{n}(S))\iso 
     I\otimes_{\Witt(S)}\Witt_{n}[F](S).$$
   So the claim of the proposition, in this case, follows as any $1$-truncated animated ring 
   is naturally a square zero extension of its $\pi_{0}$ by its $\pi_{1}[1]$. To see $(\ast)$
   we can work zariski locally on $\Witt(S)$ and thus assume $\alpha$ corresponds to a distinguished 
   element $d= (x_{0},x_{1},\dots) \in \Witt(S)$. The assumption that the Cartier-Witt divisor lives 
   in the Hodge-Tate stack tells us that $x_{0}=0$, such that $d=V(u)$ where $u$ is a unit as 
   $d$ is distinguished. Thus the formula $V(u)\cdot \underline{x}=V(u\cdot F(\underline{x}))$
   shows that $\alpha$ maps into the image of the Verschiebung map. On the other hand, as the Frobenius
   is surjective, we can write any $\underline{y}\in \Witt_{n+1}(S)$ as 
   $\underline{y}=u\cdot F(\underline{x})$ for some $\underline{x}$ and the surjectivity follows from 
   the equality
      $$V(\underline{y})= V(u\cdot F(\underline{x})) = V(u)\cdot \underline{x}.$$
    
    To deduce the claim of the proposition, we proceed as follows. As all functors are sheaves for the 
    éid-topology, by \ref{Basis of witt vector surjectives}, we can descend the statement to all static 
    $p$-nilpotent $\overline{A}$ algebras. This in particular implies the statement for all polynomial 
    algebras over $\overline{A}/p^{n}$. Now all functors are left Kan extended from 
    polynomial algebras, such that we obtain the claim for all $p$-nilpotent animated $\overline{A}$-algebras. 
    The general claim now follows by left Kan extending from affine $p$-nilpotent schemes. 
\end{proof}

\begin{non}
   For any morphism $f\from X\to S$ of $p$-adic formal schemes precomposing with the map $\pi_{HT} \from \GaHT \to \Ga$
   induces a map 
      $$\pi^{HT}_{f} \from X^{HT} \to X\times_{S}S^{HT}$$
   which we will call the \emph{Hodge-Tate structure map} of $f$. Using \ref{HodgeTate as square zero extension}, we can 
   understand the geometry of this map in terms of the cotangent complex of $f$. 
\end{non}

\begin{non}\label{torsors along square zero extentions}
   Given a square zero extension $\tilde{B} \to B$ along a connective $B$-module $N$. Then, for any morphism $X\to S$ of (derived formal) schemes and any point 
   $\eta \in X(B)\times_{S(B)}S(\tilde{B})$ the fibre at $\eta$ of the map 
       $$X(\tilde{B})\to X(B)\times_{S(B)}S(\tilde{B})$$
   gives a torsor over $\Hom_{X}(\CoTan_{f},N)$. Let us informally\footnote{To make this coherent, we have to specify
   a point in $\BHom_{X}(\CoTan_{f},N)\iso \Hom_{X}(\CoTan_{f},N[-1])$. But the trivial map does the job.} describe the 
   action. The point $\eta$ corresponds to a commutative square like the outer square in the following diagram 
\[\begin{tikzcd}
	{\Oh_{S}} & {\Oh_{X}\times_{B}\tilde{B}} & {\tilde{B}} \\
	{\Oh_{X}} & {\Oh_{X}} & B
	\arrow[from=1-1, to=1-2]
	\arrow[from=1-1, to=2-1]
	\arrow[from=1-2, to=1-3]
	\arrow[from=1-2, to=2-2]
	\arrow[from=1-3, to=2-3]
	\arrow[dashed, from=2-1, to=1-2]
	\arrow[dashed, from=2-1, to=1-3]
	\arrow[from=2-1, to=2-2]
	\arrow[from=2-2, to=2-3].
\end{tikzcd}\]
   Thus, the fiber is given by the anima of dashed lifts in the square. This anima is equivalent to the dashed lifts in the 
   left square, but now choosing such a lift gives an identification of $\Oh_{X}\times_{B}\tilde{B}$ with the trivial 
   square zero extension. Via this identification, the anima of lifts, by definition, becomes the anima of $\Oh_{S}$-linear 
   derivations in $N$, which is isomorphic to $\Hom_{X}(\CoTan_{f},N)$. 
\end{non}

\begin{nota}
   Consider a morphism $f\from X\to S$ of $p$-adic formal schemes. Then, we will write 
      $$\V(\CoTan_{f}\{1\})^{\sharp} \to X\times_{S}S^{HT}$$
   for the bundle which takes a point $\Spec(R) \to X\times_{S}S^{HT}$ to 
      $$\Hom_{R}(\eta_{X}^{\ast}\CoTan_{f},\Gasharp\{1\}(R))$$
   where $\eta$ corresponds to the point in $X$ and the Breuil-Kisin twist is defined via the point in $S^{HT}$. Note that this 
   defines a group object. The following proposition is just a reformulation of \cite{BhattLuriePrismatisationofpadicformalschemes}[5.12]
   in its natural generality. 
\end{nota}

\begin{prop}\label{Hodge Tate map is a gerbe}
   Given a map $f\from X\to S$ of derived $p$-adic formal schemes, the associated Hodge-Tate structure map 
      $$\pi^{HT}_{f} \from X^{HT} \to X\times_{S}S^{HT}$$
   defined a gerbe baned by $\V(\CoTan_{f}\{1\})^{\sharp}$.  
\end{prop}

\begin{proof}
   We claim that for any point $\eta\from \Spec(R) \to X\times_{S}S^{HT}$ the fibre of the map 
      $$X^{HT}(R) \to X(R)\times_{S(R)}S^{HT}(R)$$
   at $\eta$ defines a torsor over 
      $$\BV(\CoTan_{f})^{\sharp}(R) \iso \Hom_{R}(\eta_{X}^{\ast}\CoTan_{f},\BGasharp(R)).$$
   This follows by combining \ref{torsors along square zero extentions} and \ref{HodgeTate as square zero extension}. 
\end{proof}

\begin{cor}\label{Hodge-Tate square for trivial cotangent complex}
   For any map $f\from X\to S$ of $p$-adic formal schemes, which has vanishing cotangent
   complex the square 
\[\begin{tikzcd}
	{X^{HT}} & {S^{HT}} \\
	X & S
	\arrow[from=1-1, to=1-2]
	\arrow["{\pi^{HT}}"', from=1-1, to=2-1]
	\arrow["{\pi^{HT}}", from=1-2, to=2-2]
	\arrow["f"', from=2-1, to=2-2]
\end{tikzcd}\]
   is Cartesian. 
\end{cor}

In the following statement, we will make use 
of the notion of derived algebras in the sense of \cite{DerivedAlgMathew}.

\begin{prop}\label{affines of Hodge structure map}
   Consider an affine $p$-adic formal scheme $f\from \Spf(B)\to \Spf(\Znumb_{p})$, then the map 
      $$B^{\text{Hodge}} \to \BGm$$
   is relatively affine and represented by the graded (possibly non-connective) derived algebra 
       $$\oplus_{i}\Gamma(\Spf(B),\wedge^{i}\CoTan_{f}[-i])(i)$$
   computing absolute Hodge cohomology. 
\end{prop}

\begin{proof}
   Note that both sides are compatible with colimits in $B$. For $B^{\text{Hodge}}$, this follows as it is defined by 
   transmutation, and for the absolute Hodge cohomology this follows as it is left Kan-extended from 
   polynomial algebras. Thus it suffices to produce the identification for polynomial algebras. 
   But for polynomial 
   algebras, we can combine \ref{Hodge Tate map is a gerbe} and \cite{DerivedAlgMathew}[4.19] to obtain the affineness. 
   The computation of the global sections in this case follows from Cartier duality (see \cite{FGauges}[2.5.6]). 
\end{proof}

\subsection{Solid Hodge(-Tate) stacks}

In the following section we will work with the topos $\Spf(\Znumb_{p})_{\text{éid}}$ 
from \ref{EIDtopos}. Recall from \ref{From formal to analytic prop} that there is a 
functor 
   $$(\_)^{\Solid} \from \Spf(\Znumb_{p})_{\text{éid}} \to \AnStack_{/\Spa(\Znumb_{p})}$$
which preserves finite limits and colimits. 

\begin{non}
    Given a $p$-adic formal stack $X\in \Spf(\Znumb_{p})_{\text{éid}}$ and a locally connective quasi coherent sheaf 
    $\E \in \QCoh(X)$, we consider the bundle 
       $$\V(\E)^{\sharp} \to X.$$
    Concretely this bundle sends a point $\eta_{A} \from \Spf(A) \to X$ to the anima 
       $$\Hom_{\QCoh(\Spf(A))}(\eta_{A}^{\ast}\E,\Gasharp(A))= \Hom_{\QCoh(\Spf(A))}(\Gamma(\eta_{A}^{\ast}\E),A)$$
    and thus defines a group object in $X_{\text{éid}}$.
\end{non}

The goal of this section is now to prove the following. 

\begin{thm}\label{Properties of Sharpgerbes}
    Given a $p$-adic formal stack $X\in \Spf(\Znumb_{p})_{\text{éid}}$ together with a 
    locally connective perfect complex $\E \in \Perf(X)$, we have the following. 
      \begin{itemize}
        \item[(1)] For any $\V(\E)^{\sharp}$-gerbe $G\to X$, the induced map 
                       $$G^{\Solid} \to X^{\Solid}$$
                   is weakly proper. 
      \end{itemize}
\end{thm}

Before proving this theorem, let us record some corollaries. 

\begin{cor}\label{Properness on HodgeTate}
    Consider a proper map of $p$-adic formal schemes $X\to S$. 
    Then the induced map 
       $$X^{\text{HT},\Solid} \to S^{\text{HT},\Solid}$$
    is weakly proper.
\end{cor}

\begin{proof}
    Consider the commutative diagram 
\[\begin{tikzcd}
	{X^{\text{HT},\Solid}} && \\
	& {X^{\Solid}\times_{S^{\Solid}}S^{\text{HT},\Solid}} & {S^{\text{HT},\Solid}} \\
	& {X^{\Solid}} & {S^{\Solid}}
	\arrow[from=1-1, to=2-2]
	\arrow[from=1-1, to=2-3, bend left=15]
	\arrow[from=1-1, to=3-2, bend right=15]
	\arrow[from=2-2, to=2-3]
	\arrow[from=2-2, to=3-2]
	\arrow[from=2-3, to=3-3]
	\arrow[from=3-2, to=3-3].
\end{tikzcd}\]
   The lower horizontal map is weakly proper by \ref{Proper maps of formal schemes are proper}, so the upper horizontal 
   map is also weakly proper. By \ref{Hodge Tate map is a gerbe}, the map into the fiber product is 
   a $\V(\CoTan_{X/S}\{1\})^{\sharp}$-gerbe and thus weakly proper by \ref{Properties of Sharpgerbes} (note that the cotangent complex
   is a perfect complex because the map is of finite presentation).
   This shows what we want, as weakly proper maps are stable under composition.
\end{proof}

\begin{cor}\label{Properness of Hodge stacks}
    Consider a proper map $X\to S$ of $p$-adic formal schemes. Then the induced map 
       $$X^{\text{Hodge},\Solid} \to S^{\text{Hodge},\Solid}$$
    is weakly proper. 
\end{cor}

\begin{proof}
    Note that $\GaHodge \to \Ga$ is a square zero extension by $\BV(\Oh(-1))^{\sharp}$
    of ring stacks over $\BGm$. Using the same argument as in \ref{Hodge Tate map is a gerbe}
    we see that the map 
       $$X^{\text{Hodge}} \to X\times_{S}S^{\text{Hodge}}$$
    is a $\V(\CoTan_{X/S}(-1))^{\sharp}$-gerbe. So the claim follows the same way as in \ref{Properness on HodgeTate}.
\end{proof}

Let us now come to the proof of \ref{Properties of Sharpgerbes}.

\begin{constr}\label{PD Hodge filtration}
   Let (for simplicity) $R$ be a static ring. Then for a free $R$-module $\E$, we 
   consider the divided power deRham complex 
\[\begin{tikzcd}
	R & {\Gamma_{R}(\E)} & {\Gamma_{R}(\E)\otimes M} & {\Gamma_{R}(\E)\otimes \wedge^{2}M} & \cdots
	\arrow[from=1-1, to=1-2]
	\arrow[from=1-2, to=1-3]
	\arrow["d", from=1-3, to=1-4]
	\arrow["d", from=1-4, to=1-5]. 
\end{tikzcd}\]
   By the divided power Poincaré lemma \cite[\href{https://stacks.math.columbia.edu/tag/07LC}{Tag 07LC}]{StacksProject} this chain 
   complex is acyclic and thus gives an isomorphism 
      $$R\iso \DR^{\text{dp}}_{B/R}$$
   where the right-hand side denotes the complex starting at the first stage and $B$ 
   the free algebra on $\E$. As a chain complex the object $\DR^{\text{pd}}_{B/R}$ corresponds to a 
   filtration 
      $$Fil^{\bullet}_{\text{Hodge}}\DR^{\text{dp}}_{B/R} \to R$$
   which we will refer to as the \emph{Hodge filtration}. Animating this construction, we thus 
   obtain a Hodge filtration resolving $R$ for any connective object $\E\in \QCoh(R)$. The graded pieces 
   of this filtration are given by 
     $$gr^{i}_{\text{Hodge}}\DR^{\text{dp}}_{B/R}\iso \wedge^{i}\Gamma_{R}(\E)[-i].$$
   Note also that this filtration lives in $\Gamma_{R}(\E)$-comodules. 
\end{constr}

We need the following lemma. 

\begin{lem}\label{finiteness of Hodge filtration}
   Consider an animated ring $R$ and a connective perfect complex $\E\in \QCoh(R)$. 
   Furthermore, let $B$ denote the free $R$-algebra on $\E$. Then the Hodge filtration 
      $$Fil^{\bullet}_{\text{Hodge}}\DR^{\text{dp}}_{B/R}$$
   is finite. 
\end{lem}

\begin{proof}
   By construction, the functor 
      $$Fil^{\bullet}_{\text{Hodge}}\DR^{\text{dp}}_{B(\_)/R}\from \QCoh(R)_{\ge 0} \to \CAlg(\QCohF(R))$$
   preserves colimits. As on the right, finite filtrations are stable under finite colimits and retractions, it is enough 
   to observe the claim for $\E\iso R$. This is clear. 
\end{proof}

\begin{proof}[Proof of \ref{Properties of Sharpgerbes}$(1)$]
   By \ref{stability properties of covers in SIXF}[c] we can assume that $X=\Spf(A)$ is affine and $G\iso \BV(\E)^{\sharp}_{X}$ for some connective perfect complex 
   $\E$ on $X$. 
   We now claim that the map $X\to G$ is a descendable cover. For this first recall that 
   we can identify $\QCoh(G)$ with $\Gamma_{R}(\E)$-comodules in $\QCoh(X)$. So to see the descendability, 
   we can use the filtration from \ref{PD Hodge filtration}, which is finite by \ref{finiteness of Hodge filtration}
   and has graded pieces given by perfect complexes. 
   
   In the following, we will in particular deduce tha the diagonal is weakly cohomologically proper. So 
   by \ref{weakly proper from diagonal and cosmooth} and
   \ref{stability properties of covers in SIXF}[c], it suffices to see that the 
   map $X^{\Solid}\to G^{\Solid}$ is co-smooth. This we can see locally on $G$, so it reduces to 
   checking that the map 
      $$(\V(\E)^{\sharp}_{X})^{\Solid} \to X^{\Solid}$$
   is co-smooth. By \ref{From formal to analytic prop}[D], we can check this after 
   base changing to $\Spec(\Fp)$. Then the map in question becomes 
      $$\Spa(\Gamma_{A/p}(\E/p))\to \Spa(A/p)$$
   which is a map of affine discrete adic spaces. As $\E/p$ is built out of colimits 
   from $A/p$, the functor $\Gamma_{A/p}(\_)$ commutes with colimits and proper maps 
   of analytic rings are stable under colimits, we see that it suffices to check that 
      $$(\Gasharp)_{A/p} \to \Spec(A/p)$$
   becomes proper. That this map is integral can either be seen by hand or follows from \cite{FGauges}[2.6.5].
\end{proof}

From the proof we also obtain the following corollary. 

\begin{prop}
   Consider a map $f\from \Spf(S) \to \Spf(R)$ and assume one of the following: 
      \begin{itemize}
         \item[(a)] $R$ is semiperfectoid and $f$ is a naive syntomic cover. 
         \item[(b)] $R$ is semiperfectoid and $f$ is a zariski closed immersion. 
      \end{itemize}
   then the induced map  
      $$S^{\text{Hodge},\Solid} \to R^{\text{Hodge},\Solid}$$
   is weakly proper. 
\end{prop}

\begin{proof}
   We start with $(a)$. As weakly proper maps are stable under base change and composition, 
   we can assume that $f$ is a generic naive syntomic cover. Furthermore, by choosing a perfectoid surjecting 
   onto $R$ and lifting the cover, we can assume $R$ to be perfectoid. Now we consider the cofiber sequence 
      $$f^{\ast}\CoTan_{\Spf(R)}\to \CoTan_{\Spf(S)} \to \CoTan_{f}$$
   where the left-hand side is a free module sitting in homological degree $1$ as $R$ is perfectoid and 
   the right-hand side is free sitting in homological degree $1$ by \ref{fundamental properties of infinite root cover}.
   So the same also holds for the middle term. Using \ref{affines of Hodge structure map} we thus learn that 
      $$S^{\text{Hodge},\Solid} \to R^{\text{Hodge},\Solid}$$
   is an affine map of analytic stacks. As proper maps of analytic rings are stable under colimits (and $p$-adic completion), we now can 
   assume that the set of generators for the generic naive syntomic cover $\Spf(S) \to \Spf(R)$
   is finite, such that $\CoTan_{f}$ becomes finite free sitting in homological degree $1$. 
   Finally using \ref{From formal to analytic prop}[D], we can check the weakly properness modulo $p$ and the 
   last part of the proof of \ref{Properties of Sharpgerbes} shows that $S^{\text{Hodge},\Solid}\times_{\Spa(\Znumb_{p})} \Spa(\Fp)$
   lives in the subtopos of 
      $$\AnStack_{/R^{\text{Hodge},\Solid}\times_{\Spa(\Znumb_{p})}\Spa(\Fp)}$$ 
   generated by proper analytic rings over the base. Any affine map in there is weakly proper. 
   To prove $(b)$, by a similar argument as above, we first observe that the cotangent complex of a semiperfectoid 
   is $1$-connective. As both $R$ and $S$ are semiperfectoid, using \ref{affines of Hodge structure map}, we see that the 
   map 
     $$S^{\text{Hodge},\Solid} \to R^{\text{Hodge},\Solid}$$
   is an affine map of analytic stacks. Now writing $S$ as a colimit of finite presentation Zariski quotients of $R$ and 
   using that proper maps of analytic rings are stable under colimits, we deduce the claim from \ref{Properness of Hodge stacks}. 
\end{proof}


\subsection{Inducing covers on Prismatic stacks}

In this section we will explain how to induce covers on the Nygaard 
filtered prismatisation. 

\begin{non}
   In the following we will make use of the functor 
      $$(\_)^{\Solid}\from (\Znumb_{p}^{\Prism})_{\text{éid}} \to \AnStack_{/\Znumb_{p}^{\Prism,\Solid}}$$
   introduced in \ref{From formal to analytic prop}[A]. Note that using \ref{Nygaard filtered prismatisation of a semiperfectoid example}, 
   we see that for a semiperfectoid $R$, the stack $R^{\Nyg}$ lives in the domain of this functor. 
\end{non}

The main theorem of this section is the following.

\begin{thm}\label{Inducing descendable covers}
   Consider a map $g\from \Spf(S)\to \Spf(R)$ of affine $p$-adic formal 
   schemes and assume $\Spf(R)$ is semiperfectoid. Then we have the following:
      \begin{itemize}
         \item[(a)] If $f$ is a Zariski open immersion (resp. a Zariski cover), then the induced map 
                      $$S^{\Nyg,\Solid} \to R^{\Nyg,\Solid}$$
                    is an open immersion (resp. an open covering) of analytic stacks. 
         \item[(b)] If $f$ is étale (resp. an étale covering), then the induced map 
                       $$S^{\Nyg,\Solid} \to R^{\Nyg,\Solid}$$
                     is an étale map (resp. étale cover) of analytic stacks. 
         \item[(c)] If $f$ is an infinite $p$-root cover, the induced map 
                       $$S^{\Nyg,\Solid} \to R^{\Nyg,\Solid}$$
                     is a proper descendable cover of analytic stacks.
      \end{itemize}
\end{thm}

Before proving the theorem let us also state the following corollary. 

\begin{cor}\label{inducing descendable covers prismatisation}
   Consider a map $g\from \Spf(S)\to \Spf(R)$ of affine $p$-adic formal 
   schemes and assume $\Spf(R)$ is semiperfectoid. Then we have the following:
      \begin{itemize}
         \item[(a)] If $f$ is a Zariski open immersion (resp. a Zariski cover), then the induced map 
                      $$S^{\Prism,\Solid} \to R^{\Prism,\Solid}$$
                    is an open immersion (resp. an open covering) of analytic stacks. 
         \item[(b)] If $f$ is étale (resp. an étale covering), then the induced map 
                       $$S^{\Prism,\Solid} \to R^{\Prism,\Solid}$$
                     is an étale map (resp. étale cover) of analytic stacks. 
         \item[(c)] If $f$ is an infinite $p$-root cover, the induced map 
                       $$S^{\Prism,\Solid} \to R^{\Prism,\Solid}$$
                     is a proper descendable cover of analytic stacks.
      \end{itemize}
\end{cor}

\begin{proof}
   This follows from \ref{Inducing descendable covers} and the cartesian square 
\[\begin{tikzcd}
	{S^{\Prism,\Solid}} & {S^{\Nyg,\Solid}} \\
	{R^{\Prism,\Solid}} & {R^{\Nyg,\Solid}}
	\arrow[from=1-1, to=1-2]
	\arrow[from=1-1, to=2-1]
	\arrow[from=1-2, to=2-2]
	\arrow[from=2-1, to=2-2]
\end{tikzcd}\]
   using the open immersion $j_{\text{dR}}$ from \ref{Prismatisation in Nyg filtered Prismatiastion}.
\end{proof}

\begin{proof}[Proof of \ref{Inducing descendable covers}$(a)$ and $(b)$]
   We start with the case of an elementary Zariski open immersion (resp. a Zariski cover consisting of elementary open immersions). 
   Choosing a perfectoid mapping surjective to $R$, we can lift the open immersion (resp. the cover) to this perfectoid, and as the conclusion 
   is stable under base change, we can assume $R$ is perfectoid. In this case $S$ is perfectoid as well, and we have to check that 
   the map 
      $$\Spa(\Ainf(S)[t,u]/(tu-\varphi^{-1}(d))) \to \Spa(\Ainf(R)[t,u]/(tu-\varphi^{-1}(d)))$$
   is an open immersion (resp. an open cover). Using \ref{From formal to analytic prop}[B], we see that it is enough to show that the map 
      $$\Spf(\Ainf(S)) \to \Spf(\Ainf(R))$$
   is an open immersion (resp. a Zariski cover). This follows from deformation theory, as the conclusion holds for the map 
   $\Spf(S) \to \Spf(R)$. 

   To see the assertion in $(b)$, we can now work Zariski locally on $R$. Locally on $R$ the map $\Spf(S)\to \Spf(R)$
   factors as an open immersion followed by a standard étale map (or a collection of such maps in the case of an étale cover). 
   As above, such a map can be lifted to a perfectoid, so we may assume $R$ is perfectoid. Again 
   $S$ is perfectoid in this case as well, and we argue as in $(a)$ to deduce the claim. 
\end{proof}

In order to prove part $(c)$, we will need some preparations. First, let us recall 
some facts on filtered complexes. 

\begin{defn}
   Given an animated ring $R$, we will write 
      $$\QCohF(R):=\Fun(\Znumb^{\text{op}},\QCoh(R))$$
   for the category of \emph{filtered objects} in $\QCoh(R)$. Here we understand 
   $\Znumb^{\text{op}}$ as a poset. 
\end{defn}

\begin{rem}
   We often work with the full subcategory of derived $I$-complete objects 
   in $\QCohF(R)$ for some ideal in $R$. But it will be obvious how to adopt everything 
   we say about this category. 
\end{rem}

\begin{non}\label{colimit and associated graded of a filtration}
   Taking the colimit and taking the cofiber of each map in the filtration induces 
   functors 
\[\begin{tikzcd}
	{\QCoh(R)} & {\QCohF(R)} & {\Fun(\Znumb,\QCoh(R))}
	\arrow["{\text{colim}}"', from=1-2, to=1-1]
	\arrow["{\text{gr}^{\bullet}}", from=1-2, to=1-3]
\end{tikzcd}\]
   where we understand $\Znumb$ as a discrete category. We will refer 
   to the category $\Fun(\Znumb,\QCoh(R))$ as \emph{graded objects} in $\QCoh(R)$.
\end{non}

\begin{non}\label{computation of mapping spectra in filtered objects}
   Given two filtered objects $\F^{\bullet},\E^{\bullet} \in \QCohF(R)$, the mapping spectrum 
      $$\Hom_{\QCohF(R)}(\F^{\bullet},\E^{\bullet})$$
   can be computed as the equalizer of the two maps 
\[\begin{tikzcd}
	{\prod_{i\in \Znumb^{\text{op}}}\Hom_{\QCoh(R)}(\F^{i},\E^{i})} & {\prod_{j\in\Znumb^{\text{op}}}\Hom_{\QCoh(R)}(\F^{j},\E^{j-1})}
	\arrow[shift left, from=1-1, to=1-2]
	\arrow[shift right, from=1-1, to=1-2]
\end{tikzcd}\]
   one coming from postcomposing along the target filtration and one from precomposing 
   along the source filtration. This follows as one can write the category $\Znumb^{\text{op}}$
   as the Segal completion of the simplicial anima
      $$\dots \Delta^{1}\coprod_{\Delta^{0}}\Delta^{1}\coprod_{\Delta^{0}}\Delta^{1}\dots .$$  
\end{non}

\begin{non}\label{monoidal structure on filtered objects}
   Day-convolution equips the category $\QCohF(R)$ with a symmetric monoidal structure. That is, we obtain the 
   formula 
      $$(\F^{\bullet}\otimes \E^{\bullet})^{n} \iso \colimit_{i+j\ge n}\F^{i}\otimes\E^{j}.$$
   Furthermore, this monoidal structure is closed, so we have an internal hom. Note also that the 
   colimit functor from \ref{colimit and associated graded of a filtration} is symmetric monoidal. 

   Day-convolution also equips the category $\QCoh(R)_{\text{graded}}$ of graded objects with 
   a symmetric monoidal structure. For two graded objects $\F^{\bullet},\E^{\bullet}$ this produces the formula 
      $$(\F^{\bullet}\otimes\E^{\bullet})^{n}\iso \colimit_{i+j=n}\F^{i}\otimes\E^{j}\iso \bigoplus_{j+i=n}\F^{i}\otimes\E^{j}.$$
   Furthermore, the associated graded functor from \ref{colimit and associated graded of a filtration}
   is symmetric monoidal for these structures. Also, the symmetric monoidal structure 
   on graded objects is closed. Unwinding the formula for the internal hom, one gets 
      $$\IntHom_{\QCoh(R)_{\text{graded}}}(\F^{\bullet},\E^{\bullet})^{n}\iso \prod_{m\in \Znumb}\IntHom_{\QCoh(R)}(\F^{m},\E^{m-n}).$$
   The associated graded functor is also compatible with the internal homs \cite{Gwilliam2016EnhancingTF}[2.28], that is, we have the formula 
      $$\text{gr}^{\bullet}\IntHom_{\QCohF(R)}(\F^{\bullet},\E^{\bullet})\iso \IntHom_{\QCoh(R)_{\text{graded}}}(\text{gr}(\F^{\bullet}),\text{gr}(\E^{\bullet})).$$
\end{non}

\begin{non}\label{associated graded is conservative on complete objects}
   A filtered object $\F^{\bullet}$ in $\QCoh(R)$ is called complete, if 
      $$\limit_{n\in \Znumb^{\text{op}}}\F^{n} \iso 0.$$
   The full subcategory $\widehat{\QCohF}(R)\subset \QCohF(R)$ of complete filtered objects admits a symmetric monoidal 
   left adjoint \cite{Gwilliam2016EnhancingTF}[2.25]. Furthermore, the associated grade functor is conservative when restricted to complete objects.  
\end{non}

\begin{non}
   There is a $t$-structure on $\QCohF(R)$, called the \emph{standard $t$-structure}. 
   The connective objects are given by those filtered objects $\F^{\bullet}$, for which 
   each $\F^{i}$ is connective in the standard $t$-structure on $\QCoh(R)$. 
\end{non}

The crucial input will be the following. Recall that for an algebra $R\to S$
over an integral perfectoid with corresponding perfect prism $A$, we write $Fil^{\bullet}_{\Nyg}F^{\ast}\Prism_{S/A}$
for the relative Nygaard filtration \cite{BhattLurieAPC}[5.1].

\begin{prop}\label{Nygaard filtered Prismatisation descendable on infinite root cover}
   Given an infinite $p$-root cover $R\to S$ with $R$ an integral perfectoid. Then the 
   algebra 
      $$Fil^{\bullet}_{\Nyg}F^{\ast}\Prism_{S/A}\in \QCohF_{(I,p)\text{-comp}}(Fil_{\Nyg}^{\bullet}A)$$
   is descendable. 
\end{prop}

\begin{proof}
   Let us write $F^{\bullet}$ for the fibre of the map 
      $$Fil_{\Nyg}^{\bullet}A\to Fil_{\Nyg}^{\bullet}F^{\ast}\Prism_{S/A}$$
   and we claim that 
      $$\pi_{0}\Hom_{\QCohF}((F^{\bullet})^{\otimes 3},Fil_{\Nyg}^{\bullet}A)\iso 0.$$
   This will finish the proof by \cite{BhattScholzeWittvectorGrassmannian}[11.20]. 
   
   To check this, let us 
   first make some reduction steps. First, we can take the tensor product over $A$. Also as 
   the Nygaard filtration on $A$ is complete, we can take the 
   completed tensor product and thus assume that both sides are complete. 
   Now to see the above, by \ref{computation of mapping spectra in filtered objects}, 
   it suffices to check that 
      $$\IntHom_{\QCohF}((F^{\bullet})^{\otimes d+3},Fil_{\Nyg}^{\bullet}F^{\ast}A)\in \QCohF_{>1}$$
   where we use the standard $t$-structure. As both sides are (assumed to be) complete, we 
   can use \ref{associated graded is conservative on complete objects} and the compatibility of taking associated 
   graded with the internal Hom \ref{monoidal structure on filtered objects}, to reduce this to showing that 
      $$\IntHom_{\QCoh_{\text{graded}}}(gr^{\bullet}(F^{\bullet})^{\otimes d+3},gr^{\bullet}_{\Nyg}F^{\ast}A)\in (\QCoh_{\text{graded}})_{>2}.$$
   Unwinding the internal Hom \ref{monoidal structure on filtered objects} and the definition of the 
   $t$-structure, this boils down to showing that for each pair of integers $i,j$, we have 
      $$\Hom_{\QCoh(R)}(gr^{i}(F^{\bullet})^{\otimes 3},gr^{j}_{\Nyg}A)\in \QCoh_{>2}.$$
   
   For this, we recall from \ref{fundamental properties of infinite root cover} that we have the following:
     \begin{itemize}
         \item[(a)] The cofibre of the map $R\to S$ is a free $R$-module.
         \item[(b)] $\CoTan_{S/R}[-1]$ is a free $S$-module.  
      \end{itemize}
   Now, let us identify the graded pieces with the filtered pieces in the 
   conjugate filtration. Then using $(b)$, we observe that the conjugate filtration 
   on $\overline{\Prism}_{S}$ takes the form 
      $$S\to \bigoplus_{I_{1}}S \to \bigoplus_{I_{2}}S \to \bigoplus_{I_{3}}S \to \dots$$
   where each map is an inclusion of a direct summand. In particular using $(b)$, we learn that 
     $$gr^{i}(F^{\bullet})\iso \bigoplus_{K}R[-1]$$
   for some set $K$ for all $i$. From this, one can easily observe what we want. 
\end{proof}

We now come to the proof of \ref{Inducing descendable covers}$(b)$.

\begin{proof}[Proof of \ref{Inducing descendable covers}$(b)$]
   The assertion is stable under pullbacks and finite compositions. Thus, as explained in 
   \ref{Remark on producing IFR sheaves}, we can assume that there is a cocartesian square 
\[\begin{tikzcd}
	{\Znumb_{p}[ x_{i}|i\in I]} & {\Znumb_{p}[ x_{i}^{\frac{1}{p^{\infty}}}|i\in I]} \\
	R & S
	\arrow[from=1-1, to=1-2]
	\arrow[from=1-1, to=2-1]
	\arrow[from=1-2, to=2-2]
	\arrow[from=2-1, to=2-2].
\end{tikzcd}\]
   Furthermore, by lifting functions to a perfectoid mapping to $R$ and pushing the 
   upper horizontal map to $R$, we can assume $R$ is perfectoid. By possibly refining the map, we also can 
   assume $B$ is semiperfectoid \ref{Basis by semiperfectoids lemma}. 

   Using \ref{Nygaard filtered prismatisation of a semiperfectoid example}, we can identify the map in question 
   with the map 
      $$\Rees(Fil^{\bullet}_{\Nyg}\Prism_{S}) \to \Rees(Fil^{\bullet}_{\Nyg}\Prism_{R})$$
   over $\A^{1}/\Gm$. After pulling back along $\A^{1} \to \A^{1}/\Gm$, this map becomes a descendable 
   map of algebras by \ref{Nygaard filtered Prismatisation descendable on infinite root cover}.

   To finish the proof, it now suffices to see that the map 
      $$\Spa(\oplus_{i}Fil_{\Nyg}^{i}\Prism_{S}) \to \Spa(\oplus_{i}Fil_{\Nyg}^{i}\Prism_{R})$$
   is proper. For this note first that for any $f\in Fil_{\Nyg}^{i}\Prism_{S}$ 
   for $i\ge 1$, we have 
      $$\varphi(f) = f^{p} \text{ mod } p \text{ and } \varphi(f) \in I.$$
   This shows that such $f$ are topologically nilpotent, and it suffices to see that the map 
      $$\Prism_{R} \to \Prism_{S}$$
   is integral modulo $(I,p)$. As this statement is stable under colimits in $S$, we can assume $S$ arises as the 
   pullback of a generic infinite $p$-root cover in one generator. In this case, using the conjugate filtration, it is 
   easy to see that the map 
     $$S\to \bar{\Prism}_{S}$$
   is finite. 
\end{proof}

\subsection{Solid Prismatic stacks}

We are now able to define the \emph{Solid syntomification}. 

\begin{constr}
    Given a semiperfectoid $S$, as in the last section, we can associate to $S$ 
    the analytic stacks 
       $$S^{\Nyg,\Solid} \text{ and } S^{\Prism,\Solid}\iso \Spa(\Prism_{S}) \in \AnStack_{/\Znumb_{p}^{\Prism}}.$$
    Using \ref{Inducing descendable covers} and \ref{inducing descendable covers prismatisation}, we see that this 
    assignment sends infinite $p$-root covers and étale covers to effective epimorphisms 
    of analytic stacks. 
    
    Now let us consider the topos 
       $$(\Znumb_{p})_{_{\sqrt{\infty},\text{ét}}}$$
    generated by affine $p$-adic formal schemes via the union of the infinite $p$-root and the 
    étale topology. Then this topos is generated by semiperfectoids \ref{Basis by semiperfectoids lemma}. 
    Using this observation, we obtain colimit-preserving functors 
       $$(\_)^{\Nyg,\Solid} \text{ and } (\_)^{\Prism,\Solid} \from (\Znumb_{p})_{_{\sqrt{\infty},\text{ét}}} \to \AnStack_{/\Znumb_{p}^{\Prism}}.$$
    Note that $(\_)^{\Prism,\Solid}$ also preserves finite limits, and that this also 
    becomes true for $(\_)^{\Nyg,\Solid}$ if we change the base in the target to 
       $$(\Znumb_{p})^{\Nyg,\Solid}.$$ 
\end{constr}

\begin{defn}
    For a $p$-adic formal scheme $X$, we will call the analytic stack 
       $$X^{\Prism,\Solid}$$
    the \emph{solid Prismatisation} of $X$. 

    Analogously, we will refer to the analytic stack 
      $$X^{\Nyg,\Solid}$$
    as the \emph{solid Nygaard filtered Prismatisation} of $X$. 
\end{defn}

\begin{non}
   The open immersions from \ref{Prismatisation in Nyg filtered Prismatiastion} 
   induce disjoint open immersions 
      \[\begin{tikzcd}
	{X^{\Prism,\Solid}} & {X^{\Nyg,\Solid}} & {X^{\Prism,\Solid}}
	\arrow["{j_{dR}}", from=1-1, to=1-2]
	\arrow["{j_{HT}}"', from=1-3, to=1-2].
\end{tikzcd}\]
   of analytic stacks. Again composing $j_{\text{dR}}$ with the structure map recovers the identity, while composing 
     $j_{\text{HT}}$ with the structure map recovers the Frobenius.
\end{non}

\begin{defn}
   Given a derived $p$-adic formal scheme $X$, we write $X^{\text{syn},\Solid}$ 
   for the pushout 
\[\begin{tikzcd}
	{X^{\Prism,\Solid}\coprod X^{\Prism,\Solid}} & {X^{\Nyg,\Solid}} \\
	{X^{\Prism,\Solid}} & {X^{\text{syn},\Solid}}
	\arrow[from=1-1, to=1-2]
	\arrow["can"', from=1-1, to=2-1]
	\arrow[from=1-2, to=2-2]
	\arrow[from=2-1, to=2-2]
\end{tikzcd}\]
   where the upper horizontal map is induced by the two open immersions $j_{dR}$ and $j_{HT}$.
   We will refer to this analytic stack as the \emph{solid Syntomification}.
\end{defn}

\begin{non}\label{from NFP to syn remark}
   Using that the open immersions $j_{dR}$ and $j_{HT}$ are stable under 
   base change and descent for 
   $\infty$-topoi, we see that for any morphism $f\from X\to S$ of derived $p$-adic 
   formal schemes, the induced square 
\[\begin{tikzcd}
	{X^{\Nyg,\Solid}} & {S^{\Nyg,\Solid}} \\
	{X^{\text{syn},\Solid}} & {S^{\text{syn},\Solid}}
	\arrow[from=1-1, to=1-2]
	\arrow[from=1-1, to=2-1]
	\arrow[from=1-2, to=2-2]
	\arrow[from=2-1, to=2-2]
\end{tikzcd}\]
   is Cartesian. Furthermore for any derived $p$-adic formal scheme $X$ the 
   map 
      $$X^{\Nyg,\Solid} \to X^{\text{syn},\Solid}$$
   is an \'etale surjection. Combining these two observations, one can deduce 
   most properties for the solid syntomification from the solid Nygaard filtered 
   Prismatisation. 
\end{non}

\begin{non}\label{Qcoh of syntomic}
   For a $p$-adic formal scheme $X$, we can compute the category of 
   solid prismatic F-gauges as the equalizer 
\[\begin{tikzcd}
	{\FGauge_{\Prism}^{\Solid}(X)} & {\QCoh(X^{\Nyg,\Solid})} & {\QCoh(X^{\Prism,\Solid})}
	\arrow[from=1-1, to=1-2]
	\arrow["{j_{dR}^{\ast}}", shift left, from=1-2, to=1-3]
	\arrow["{j_{HT}^{\ast}}"', shift right, from=1-2, to=1-3].
\end{tikzcd}\]
\end{non}

Let us record the following.

\begin{prop}\label{dualizable objects in Fgauges}
   For a $p$-adic formal scheme $X$, the subcategories of dualizable objects canonically identify as 
      \begin{itemize}
         \item $\QCoh(X^{\Prism,\Solid})^{\text{dual}} \iso \Perf(X^{\Prism})$ 
         \item $\QCoh(X^{\Nyg,\Solid})^{\text{dual}} \iso \Perf(X^{\Nyg})$
         \item $\FGauge_{\Prism}^{\Solid}(X)^{\text{dual}} \iso \Perf(X^{\text{Syn}})$
      \end{itemize}
   where the categories on the right are the perfect complexes on the classical incarnations of these 
   stacks as defined in \cite{FGauges}[6.1]. 
\end{prop}

\begin{proof}
   As taking the dualizable object of symmetric monoidal categories commutes with limits \cite{HA}[4.6.1.11], the claims 
   easily reduce to the case $X \iso \Spf(R)$ for $R$ a semiperfectoid. In this case the claims follow from the next lemma.
\end{proof}

\begin{lem}
   Consider an animated ring $A$ which is (derived) $I$-adically complete for some finitely generated ideal $I$ together with some subring 
   of integral elements $A^{+} \subset A$. Then we have a canonical identification 
      $$\QCoh(\Spa(A,A^{+}))^{\text{dual}}\iso \Perf(A).$$
\end{lem}

\begin{proof}
   Let us first assume that $A$ is discrete. Then any dualizable object is in particular nuclear and thus discrete 
   by \cite{Andreychev2023KTheorieAR}[3.19]. But the dualizable objects in the classical derived category of a ring are 
   exactly the perfect complexes. 
   Now, to see the claim for general $A$, we first observe that for a dualizable object $M$ the functor $M\otimes\_$
   preserves limits. Thus we have 
      $$M\iso M\otimes A \iso M\otimes \limit_{n}A/I^{n} \iso \limit_{n}M/I^{n}$$
   which shows that $M$ is $I$-adically complete. From this we see that 
      $$\QCoh(\Spa(A,A^{+}))^{\text{dual}}\iso \QCoh(\Spa(A,A^{+}))_{\widehat{I}}^{\text{dual}}\iso
      \limit_{n}\QCoh(\Spa(A/I^{n},A^{+}))^{\text{dual}}\iso \Perf(A).$$
   Where the second equivalence uses that taking duals commutes with limits \cite{HA}[4.6.1.11] and the third 
   equivalence the case of a discrete ring. 
\end{proof}

\begin{non}
   Recall from \cite{BhattLurieAPC}[2.2.11+3.3.8] that there is a line bundle 
      $$\Oh_{\Znumb_{p}^{\Prism}}\{1\}\in \Perf(\Znumb_{p}^{\Prism})$$
   called the \emph{Breuil-Kisin twist}, which comes together with a frobenius 
   isomorphism \cite{BhattLurieAPC}[2.2.14]
     $$\varphi \from F^{\ast}\Oh_{\Znumb_{p}^{\Prism}}\{1\}\iso \mathcal{I}^{-1}\otimes \Oh_{\Znumb_{p}^{\Prism}}\{1\}$$
   where $\mathcal{I}$ denotes the Hode-Tate ideal.
   Heuristically, this line bundle should be thought of as 
      $$\otimes_{k\ge 0}(F^{k})^{\ast}\mathcal{I}$$
   where the Frobenius isomorphism is the evident one.
   Using \ref{dualizable objects in Fgauges}, we can understand this object as a line bundle on $\Znumb_{p}^{\Prism,\Solid}$.

   Recall that we had two maps 
\[\begin{tikzcd}
	{\Znumb_{p}^{\Prism,\Solid}} & {\Znumb_{p}^{\Nyg,\Solid}} & {\Disc^{1}_{\Solid}/(\Gm)_{\Solid}^{\text{an}}}
	\arrow["\pi"', from=1-2, to=1-1]
	\arrow["t", from=1-2, to=1-3]
\end{tikzcd}\]
   where $\pi$ is the structure map and $t$ is the Rees map. 
   Using these maps, the \emph{Nygaard filtered Breuil-Kisin twist} can be defined via 
      $$\Oh_{\Znumb_{p}^{\Nyg}}\{1\}:= \pi^{\ast}\Oh_{\Znumb_{p}^{\Prism}}\{1\}\otimes t^{\ast}\Oh(-1).$$
   We will understand this line bundle on $\Znumb_{p}^{\Nyg,\Solid}$. 

   Note that on the one hand we have 
      $$j_{dR}^{\ast}\Oh_{\Znumb_{p}^{\Nyg}}\{1\}\iso \Oh_{\Znumb_{p}^{\Prism}}\{1\}$$
   as composing $j_{dR}$ with the structure map gives the identity, and the composition 
   $t\circ j_{dR}$ factors over $(\Gm)_{\Solid}^{\text{an}}/(\Gm)_{\Solid}^{\text{an}}$.
   On the other hand, we have 
      $$j_{HT}^{\ast}\Oh_{\Znumb_{p}^{\Nyg}}\{1\}\iso F^{\ast}\Oh_{\Znumb_{p}^{\Prism}}\{1\}\otimes \mathcal{I}\iso \Oh_{\Znumb_{p}^{\Prism}}\{1\}.$$
   Where, for the first isomorphism, we use that composing $j_{HT}$ with the structure map 
   recovers the frobenius and that the composition $t\circ j_{HT}$ classifies the 
   Hodge-Tate locus $\Znumb_{p}^{HT}\subset \Znumb_{p}^{\Prism}$. The second 
   Isomorphism comes from the Frobenius automorphism on the Breuil-Kisin twist recalled above.
   In particular, using this identification, we obtain a line bundle 
     $$\Oh_{\Znumb_{p}^{\text{syn}}}\{1\}$$
   on $\Znumb_{p}^{\text{syn},\Solid}$ which we will refer to as the 
   \emph{syntomic Breuil-Kisin twist}. 

   By base change, we also obtain line bundles 
      $$\Oh_{X^{\Prism}}\{1\}, \Oh_{X^{\Nyg}}\{1\}, \Oh_{X^{\text{syn}}}\{1\}$$
   for an arbitrary derived $p$-adic formal scheme $X$. 
\end{non}

\begin{non}\label{Mapping spectrum computes syn cohomology remark}
   Given a $p$-adic formal scheme $X$, then the mapping spectrum 
      $$\Hom_{X^{\text{syn},\Solid}}(\Oh_{X^{\text{syn},\Solid}},\Oh_{X^{\text{syn},\Solid}}\{n\})$$
   via the presentation given in \ref{Qcoh of syntomic} become the 
   fiber of the map 
      $$\varphi\{n\}-\text{can}\from Fil_{\Nyg}^{n}\Prism_{X}\{n\} \to \Prism_{X}\{n\}$$
   where $\varphi\{n\}$ is the twisted filtered frobenius and $\text{"can"}$ comes from 
   the inclusion of the filtration. In particular, this mapping spectrum identifies with 
   the syntomic cohomology
      $$R\Gamma_{\text{syn}}(X,\Znumb_{p}(n))$$
   for $p$-adic formal schemes as defined in \cite{BhattLurieAPC}[7.4]. 
\end{non}

\section{A six-functor formalism for syntomic cohomology}\label{theSIXFF}

\subsection{Solid F-Gauges}

At this point, we have obtained a six-functor formalism for syntomic cohomology, of which we will now 
start to discuss properties. 

\begin{defn}
    Given a $p$-adic formal scheme $X$, we will write 
       $$\FGauge_{\Prism}^{\Solid}(X):= \QCoh_{\text{qc}}(X^{\text{syn},\Solid})$$
    for the category of quasi-coherent sheaves on the solid Syntomification 
    of $X$ and refer to it as \emph{solid prismatic F-gauges} on $X$. 
\end{defn}

\begin{constr}
    Recall from \ref{sixfunctors on analytic stacks}, that we have a six-functor formalism 
    on analytic stacks, which considers the category of quasi-coherent sheaves on an 
    analytic stack. Precomposing with the assignment 
       $$X \mapsto X^{\text{syn},\Solid}$$
    thus tells us that the assignment 
      $$X\mapsto \FGauge_{\Prism}^{\Solid}(X)$$
    extends to a six-functor formalism. 
\end{constr}

The goal of the next two sections is to prove the following theorem. 

\begin{thm}\label{Main theorem}
    The functor 
       $$\FGauge_{\Prism}^{\Solid}(\_)\from \fSch_{\Spf(\Znumb_{p})}^{\text{op}} \to \PrL_{\text{st}}$$
    can be extended to a six-functor formalism on the category of (derived)
    $p$-adic formal schemes satisfying the following:
       \begin{itemize}
        \item[(A)] Morphisms locally of finite type are 
                   $!$-able. 
        \item[(B)] \'Etale morphisms are cohomologically \'etale. 
        \item[(C)] Any finite composition of closed immersions and proper morphisms is weakly cohomologically proper. 
        \item[(D)] The functor $(\_)^{\text{syn},\Solid}$ preserves étale covers and sends Zariski covers to 
                   open covers of analytic stacks. In particular the functors $\FGauge_{\Prism}^{\Solid}(\_)^{\ast}$ and 
                   $\FGauge_{\Prism}^{\Solid}(\_)^{!}$ are \'etale sheafs. 
        \item[(E)] It admits Tate twists. That is for any $p$-adic formal scheme $X$ 
                   the object 
                      $$\Oh_{X^{\text{syn}}}(-1):= cof(\Oh_{X^{\text{syn}}} \to f_{\ast}\Oh_{(\ProS^{1})^{\text{syn}}})$$
                   is $\otimes$-invertible and its inverse identifies with 
                      $$\Oh_{X^{\text{syn}}}(1)\iso \Oh_{X^{\text{syn}}}\{1\}[-2]$$
                   a shift of the Breuil-Kisin twist. 
        \item[(F)] Any smooth morphism is cohomologically smooth. Furthermore, for such a smooth 
                   morphism $f\from X\to S$, we have an identification 
                      $$f^{!}\Oh_{S^{\text{syn}}}:= \omega_{f} \iso \Oh_{X^{\text{syn}}}(d)\iso \Oh_{X^{\text{syn}}}\{d\}[-2d]$$
                    of the dualizing sheaf, where $d$ is the relative dimension 
                    of $f$. 
        \item[(G)] The dualizable objects identify 
                      $$(\FGauge_{\Prism}^{\Solid}(\_))^{\text{dual}}\iso \Perf((\_)^{\text{syn}})$$
                    with perfect F-Gauges as defined in \cite{FGauges}[6.1].
                   In particular, there is a functorial identification 
                      $$R\Gamma_{\text{syn}}^{\text{BMS}}(\_,\Znumb_{p}(n))\iso \Hom_{(\_)^{\text{syn}}}(\Oh_{(\_)^{\text{syn}}},\Oh_{(\_)^{\text{syn}}}\{n\})$$
                   of the mapping spectrum with the syntomic cohomology of $p$-adic formal schemes as defined in \cite{BhattLurieAPC}.
       \end{itemize}
\end{thm}

\begin{proof}[Proof of \ref{Main theorem}(G)]
    This was explained in \ref{dualizable objects in Fgauges} and \ref{Mapping spectrum computes syn cohomology remark}.
\end{proof}

\begin{non}
    Note that assertions $(A),(B),(C)$ and $(D)$ are étale local on the target. So using 
    \ref{from NFP to syn remark}, we see that it suffices to check the analogous claims for the solid Nygaard
    filtered prismatisation. We will make use of this without further mentioning.
\end{non}

\begin{proof}[Proof of \ref{Main theorem} $(B)$ and $(D)$]
    We claim that for an étale morphism (resp. an étale cover) $X\to S$ of $p$-adic formal schemes
    the induced morphism 
       $$X^{\Nyg,\Solid} \to S^{\Nyg,\Solid}$$
    is an étale morphism (resp. an étale cover).  
    As this assertion is local on the target, we can assume $S\iso \Spf(R)$
    for $R$ a semiperfectoid. But then the claim is \ref{Inducing descendable covers}.
\end{proof}

Let us outsource parts of the proof in the following propositions. 

\begin{prop}\label{closed immersions are proper}
    Given a Zariski closed immersion $i\from Z\to X$ of $p$-adic formal schemes. 
    The induced morphism 
       $$Z^{\text{syn},\Solid} \to X^{\text{syn},\Solid}$$
    is affine and weakly proper. 
\end{prop}

\begin{proof}
    It is enough to prove the analogous assertion for the solid Nygaard 
    filtered prismatisation. Furthermore, using \ref{Main theorem}[(D)] we can assume 
    that $X$ (and thus also $Z$) is affine and then semiperfectoid. 
    Let us write 
      $$\Spf(S)\iso Z\to X\iso \Spf(R).$$
    As the assertion is stable under colimits in $S$, we can assume 
      $$S \iso R/(f_{1},\dots,f_{n})$$
    and by lifting the functions to a perfectoid mapping surjective to $R$, we can 
    assume $R$ to be perfectoid. Now one argues similarly to the last paragraph 
    of the proof of \ref{Inducing descendable covers}[(c)]. 
\end{proof}

\begin{prop}\label{affine space is shreakable}
    Given a $p$-adic formal scheme $X$. The induced map 
       $$(\A^{n}_{X})^{\text{syn},\Solid} \to X^{\text{syn},\Solid}$$
    is locally of \textsuperscript{+}finite type and thus in particular $!$-able. 
\end{prop}

\begin{proof}
    As the assertion is stable under base change, we can assume $X\iso \Spf(\Znumb_{p})$.
    Now $\Spf(\Znumb_{p})$ can be resolved by an integral perfectoid $R$ via an infinite $p$-root 
    cover. Thus we can assume $X\iso \Spf(R)$ with $R$ perfectoid. 

    Now write 
       $$\Spf(R^{\infty}) := \Spf(R[x_{1}^{\frac{1}{p^{\infty}}},\dots, x_{n}^{\frac{1}{p^{\infty}}}]).$$
    Then we have the infinite $p$-root cover 
       $$\Spf(R^{\infty}) \to \A^{n}_{R}$$
    and it suffices to see that the map 
       $$(R^{\infty})^{\text{syn},\Solid} \to R^{\text{syn},\Solid}$$
    is of \textsuperscript{+}finite type. As in the proof of \ref{Inducing descendable covers}[(c)]
    for this it is enough to argue that the map 
       $$\Spf(\Ainf(R^{\infty})) \to \Spf(\Ainf(R))$$
    is of \textsuperscript{+}finite type. But as 
      $$\Ainf(R^{\infty})\iso \Ainf(R)[x_{1}^{\frac{1}{p^{\infty}}},\dots,x_{n}^{\frac{1}{p^{\infty}}}]$$
    this is clear. 
\end{proof}

\begin{proof}[Proof of \ref{Main theorem}(A),(C)]
    Using \ref{Main theorem}$(D)$, $(A)$ follows from \ref{affine space is shreakable} and 
    \ref{closed immersions are proper}. 

    We now prove $(C)$. Consider a proper map $X\to S$, then the diagonal is a closed immersion, such that the induced map becomes weakly 
    cohomologically proper by \ref{closed immersions are proper}. Thus by \ref{weakly proper from diagonal and cosmooth}
    is suffices to see that the induced map is co-smooth.
    
    By \ref{Main theorem}$(D)$, we can assume 
    $S$ to be affine and resolving $S$ by a semiperfectoid, we can even assume $S$ to be semiperfectoid. 
    As a proper map is of finite type, we can use \ref{Main theorem}$(D)$, \ref{affine space is shreakable}
    and \ref{closed immersions are proper}, to see that the induced map 
       $$X^{\Nyg,\Solid} \to S^{\Nyg,\Solid}$$
    is locally of \textsuperscript{+}finite type. Thus by \ref{From formal to analytic prop}$(E)$, we can check the 
    assertion on a Zariski stratification of the target. Using the stratification \ref{stratification of Nygaard filtered Prismatisation}
    the claim now follows from \ref{Properness of Hodge stacks} and \ref{Properness on HodgeTate}. 
\end{proof}

\subsection{The additive orientation}

To prove the rest, we will construct an additive orientation. 

\begin{non}\label{additive orientation for syn}
   Using \ref{Mapping spectrum computes syn cohomology remark}, we recall that there is 
   a map 
      $$\text{c}_{1}^{\text{syn}}\from R\Gamma_{\text{\'et}}(X,\Gm)[-1] \to \Hom_{X^{\text{syn}}}(\Oh_{X^{\text{syn}}},\Oh_{X^{\text{syn}}}\{1\}[-2])$$ 
   coming from the prismatic logarithm \cite{BhattLurieAPC}[7.5.2]. This map 
   is natural in $X$ and identifies the target with the derived $p$-completion 
   of the source \cite{BhattLurieAPC}[7.5.6]. In particular, this gives us a theory 
   of first Chern classes.
\end{non}

\begin{constr}\label{Projective bundle morphism construction for syn}
   Using the first Chern classes from \ref{additive orientation for syn}, as in \ref{Projective bundle formula morphism Construction}
   we can construct a morphism 
      $$\sum_{i=0}^{d} c^{\text{syn}}_{1}(\Oh(1))^{i}\{d-i\}[2(i-d)]\from \bigoplus_{i=0}^{d}\Oh_{\Znumb_{p}^{\text{syn}}}\{d-i\}[2(i-d)] \to f_{\ast}\Oh_{(\ProS^{d})^{\text{syn}}}\{d\}[-2d]$$
   where we write $f\from \ProS^{d}\to \Spf(\Znumb_{p})$ for the projection. 
\end{constr}

\begin{prop}\label{Projective bundle formula for syn proposition}
   The morphism 
      $$\sum_{i=0}^{d} c^{\text{syn}}_{1}(\Oh(1))^{i}\{d-i\}[2(i-d)]\from \bigoplus_{i=0}^{d}\Oh_{\Znumb_{p}^{\text{syn}}}\{d-i\}[2(i-d)] \to f_{\ast}\Oh_{(\ProS^{d})^{\text{syn}}}\{d\}[-2d]$$
   is an isomorphism. 
\end{prop}

\begin{proof}
    By resolving $\Spf(\Znumb_{p})$ by a semiperfectoid (one can even find a perfectoid), we can assume 
    that the base is given by $\Spf(R)$ for $R$ a semiperfectoid. 
    Using the formula of mapping spectra recalled in \ref{dualizable objects in Fgauges}, it suffices to 
    check that the maps 
      \begin{itemize}
         \item[(a)] $\bigoplus_{i=0}^{d}\Oh_{R^{\Prism}}\{d-i\}[2(i-d)] \to f_{\ast}\Oh_{(\ProS^{d})^{\Prism}}\{d\}[-2d]$ in $\QCoh(R^{\Prism})$
         \item[(b)] $\bigoplus_{i=0}^{d}\Oh_{R^{\Nyg}}\{d-i\}[2(i-d)] \to f_{\ast}\Oh_{(\ProS^{d})^{\Nyg}}\{d\}[-2d]$ in $\QCoh(R^{\Nyg})$
      \end{itemize}
   are isomorphisms.
   As the lower-$\ast$ functor commutes with limits, both sides of these maps are 
   $(p,I)$-adically complete. Thus we can check these isomorphisms modulo $(I,p)$. We will 
   do so but omit the modulo notation in the following. 
   Now in $\QCoh(R^{\Prism,\Solid})$ mapping out of a free module on a profinite set $S\iso \limit_{n}S_{n}$
   is conservative and similar mapping out of shifted versions of such a free guy is 
   conservative on $\QCoh(R^{\Nyg,\Solid})$. We now claim that on both sides of the map we can compute 
   mapping out of such a free guy by 
      $$\Hom_{\QCoh(R^{\Prism,\Solid})}(S,\_) \iso \colimit_{n}\Hom_{\QCoh(R^{\Prism,\Solid})}(S_{n},\_)$$
    and similar for $\QCoh(R^{\Nyg,\Solid})$.
    For the domains of the maps, that is clear. For the targets, we can first choose the standard 
    Zariski presentation of projective spaces, which writes the target as a finite limit 
    of cohomologies of affine spaces and circles. Those we can now resolve by semiperfectoids, and the claim 
    follows as filtered colimits commute with totalisations in non-connective spectra. 
    Using this formula, the claims now follow from \cite{BhattLurieAPC}[9.1.4.(4)] and 
    \cite{BhattLurieAPC}[9.1.4.(5)] respectively. 
\end{proof}

\begin{non}\label{geometrically smooth p-adic morphisms}
   On the category of (derived) $p$-adic formal schemes, the collection of étale 
   and smooth morphisms define a class of geometrically étale and smooth morphisms in the sense of 
   \ref{definition of geometrically smooth morphisms}. To construct blow-ups, we can locally 
   $p$-adically complete the construction from (derived) algebraic geometry \cite{KhanVirtualCD}
   \cite{Tang2026TheG}[A] and then glue. All the other assertions in the definition are clear.    
\end{non}

\begin{prop}\label{blowup excision syn}
   $\FGauge_{\Prism}^{\Solid}$-cohomology satisfies blow-up excision in the sense of \ref{blowup excision definition}.
\end{prop}

\begin{proof}
   We have to see that the square 
\[\begin{tikzcd}
	\ProS_{\Znumb_{p}}^{n-1} & \V_{\ProS^{n-1}_{\Znumb_{p}}}(\Oh(1)) \\
	\{0\} & \A^{n}_{\Znumb_{p}}
	\arrow[from=1-1, to=1-2]
	\arrow[from=1-1, to=2-1]
	\arrow["p", from=1-2, to=2-2]
	\arrow["0"', from=2-1, to=2-2]
	\arrow["f", from=1-1, to=2-2]
\end{tikzcd}\]
   induces a (co)cartesian square on cohomology in $\FGauge_{\Prism}^{\Solid}(\A^{n}_{\Znumb_{p}})$.
   Using proper base change, we can check this after pulling back along an infinite 
   $p$-root cover. Thus by resolving $\Spf(\Znumb_{p})$ by a semiperfectoid $\Spf(R)$
   and then resolving $\A^{n}_{R}$ by a semiperfectoid, it suffices to check the analogous claim 
   for a blow-up square 
\[\begin{tikzcd}
	E & B \\
	Z & {\Spf(R)}
	\arrow[from=1-1, to=1-2]
	\arrow[from=1-1, to=2-1]
	\arrow[from=1-2, to=2-2]
	\arrow[from=2-1, to=2-2]
\end{tikzcd}\]
   where $R$ is semiperfectoid (even perfectoid if the reader likes). Note that such resolutions exist in such a way 
   that all $p$-adic formal schemes in the square are the $p$-adic completions of classical schemes. Now, similar to the 
   proof of \ref{Projective bundle formula for syn proposition} the claim follows from 
   \cite{BhattLurieAPC}[9.4.6].
\end{proof}

\begin{proof}[Proof of \ref{Main theorem}(E),(F)]
   Assertion $(E)$ easily follows from the projective bundle formula \ref{Projective bundle formula for syn proposition}.
   To deduce assertion $(F)$, we note that the construction of first Chern classes 
   \ref{additive orientation for syn} together with the projective bundle formula \ref{Projective bundle formula for syn proposition}
   and the class of étale and smooth morphisms \ref{geometrically smooth p-adic morphisms} give an additively 
   oriented six-functor formalism (for the other assertions we use the rest of \ref{Main theorem}). Furthermore, it 
   satisfies blow-up excision by \ref{blowup excision syn}. Thus we can apply 
   \ref{dualizing complex theorem} and \ref{dualizing complex corollary}.
\end{proof}

This finishes the proof of \ref{Main theorem}.

\subsection{Realizations}

We will now discuss some realizations. 

\begin{non}\label{realizations construction}
    We consider the following situation. Assume we have given an analytic stack $\BS$ 
    together with a map 
       $$\BS \to \Znumb_{p}^{\text{syn}\Solid}.$$
    Then we automatically obtain a well-behaved six-functor formalism for $p$-adic 
    formal schemes by the following construction. 

    For a $p$-adic formal scheme $X$, we first consider the cartesian square 
\[\begin{tikzcd}
	{X_{\BS}} & {X^{\text{syn},\Solid}} \\
	\BS & {\Znumb_{p}^{\text{syn},\Solid}}
	\arrow[from=1-1, to=1-2]
	\arrow[from=1-1, to=2-1]
	\arrow[from=1-2, to=2-2]
	\arrow[from=2-1, to=2-2]
\end{tikzcd}\]
   and then associate to $X$ the category 
      $$\QCoh(X_{\BS}).$$
    As all the assertions (A)-(F) of \ref{Main theorem} are stable under base change, this six-functor formalism will satisfy
    analogous assertions, where one defines the Tate-twist by base change. We will now discuss some examples of this phenomenon.
\end{non}

\begin{ex}
    For any $p$-adic formal scheme, we have a cartesian square 
\[\begin{tikzcd}
	{X^{\Prism,\Solid}} & {X^{\text{syn},\Solid}} \\
	{\Znumb_{p}^{\Prism,\Solid}} & {\Znumb_{p}^{\text{syn},\Solid}}
	\arrow["{j_{dR}}", from=1-1, to=1-2]
	\arrow[from=1-1, to=2-1]
	\arrow[from=1-2, to=2-2]
	\arrow["{j_{dR}}"', from=2-1, to=2-2].
\end{tikzcd}\]
    In particular, using the deRham point in \ref{realizations construction}, we obtain a well-behaved six-functor formalism 
    for Prismatic cohomology.
\end{ex}

\begin{ex}
    Similar to the last example, we can use the map 
       $$\Znumb_{p}^{\text{HT},\Solid} \to \Znumb_{p}^{\text{syn},\Solid}$$
    to obtain a six-functor formalism for Hodge-Tate cohomology. 
\end{ex}

\begin{constr}
    Recall that we had the structure map 
       $$\Znumb_{p}^{\Prism,\Solid} \to \Disc^{1}/\Gm^{\text{an}}$$
    which came from base changing a Cartier-Witt divisor along the map $\Witt\to \Ga$. 

    Note that the action of $\Gm^{\text{an}}$ restricts to the substack 
       $$\colimit_{n}\Spa(\Znumb_{p}[x]/(x^{n})) \subset \Disc^{1}$$
    such that we can base change this structure map along this map quotient 
    by $\Gm^{\text{an}}$. Further base changing along the inclusion 
      $$\colimit_{n}\Spa(\Znumb/p^{n}) \to \Spa(\Znumb_{p})$$
    we obtain an (open) substack 
      $$\Znumb_{p}^{\Prism} \subset \Znumb_{p}^{\Prism,\Solid}.$$ 
    To also modify the solid Nygaard filtered prismatisation, we consider the 
    cartesian square 
\[\begin{tikzcd}
	{\Znumb_{p}^{\Nyg}} & {\Znumb_{p}^{\Nyg,\Solid}} \\
	{\Znumb_{p}^{\Prism}} & {\Znumb_{p}^{\Prism,\Solid}}
	\arrow[from=1-1, to=1-2]
	\arrow[from=1-1, to=2-1]
	\arrow["\pi", from=1-2, to=2-2]
	\arrow[from=2-1, to=2-2]. 
\end{tikzcd}\]
   And finally, analogously to the solid Syntomification, we construct the pushout 
\[\begin{tikzcd}
	{\Znumb_{p}^{\Prism}\amalg \Znumb_{p}^{\Prism}} & {\Znumb_{p}^{\Nyg}} \\
	{\Znumb_{p}^{\Prism}} & {\Znumb_{p}^{\text{syn}}}
	\arrow[from=1-1, to=1-2]
	\arrow[from=1-1, to=2-1]
	\arrow[from=1-2, to=2-2]
	\arrow[from=2-1, to=2-2].
\end{tikzcd}\]
   Note that we have a map $\Znumb_{p}^{\text{syn}} \to \Znumb_{p}^{\text{syn},\Solid}$.
\end{constr}

\begin{defn}
    For a $p$-adic formal scheme $X$, we define the \emph{classical Syntomification} to be 
    the fiber product 
\[\begin{tikzcd}
	{X^{\text{syn}}} & {X^{\text{syn},\Solid}} \\
	{\Znumb_{p}^{\text{syn}}} & {\Znumb_{p}^{\text{syn},\Solid}}
	\arrow[from=1-1, to=1-2]
	\arrow[from=1-1, to=2-1]
	\arrow[from=1-2, to=2-2]
	\arrow[from=2-1, to=2-2].
\end{tikzcd}\]
\end{defn}

We now obtain the following corollary. 

\begin{thm}\label{Classsical syntomification theorem}
   The functor 
      $$\QCoh((\_)^{\text{syn}})\from \fSch_{\Spf(\Znumb_{p})}^{\text{op}} \to \PrL$$
   can be extended to a six-functor formalism on the category of (derived) $p$-adic formal 
   schemes satisfying (A)-(F) of \ref{Main theorem}. Furthermore, assertion (G) can be extended to:
   There is a symmetric monoidal and fully faithful natural transformation 
      $$\FGauge_{\Prism}(\_) \to \QCoh((\_)^{\text{syn}})$$
   where the left-hand side denotes the category of prismatic F-Gauges as defined in 
   \cite{FGauges}. 
\end{thm}

\begin{proof}
   As explained in \ref{realizations construction}, the first assertion follows from \ref{Main theorem}. 
   To prove the second one, by descent, it suffices to evaluate on a semiperfectoid $R$. Then, using the analogous 
   equalizer diagrams as in \ref{Qcoh of syntomic}, it suffices to check the analogous claims for 
      $$R^{\Nyg} \text{ and } R^{\Prism}.$$
   Note that $\QCoh(R^{\Nyg})$ identifies with the subcategory of $\QCoh(R^{\Nyg,\Solid})$ consisting of 
   $(p.I)$-adically complete objects. Thus the claim reduces to the assertion that for an animated ring $A$ 
   the functor 
      $$\QCoh(A)^{\text{cl}} \to \QCoh(\Spa(A)),$$
   where we denote by $\QCoh(\_)^{\text{cl}}$ the classical derived category, is symmetric monoidal and fully faithful. 
\end{proof}

Let us now consider an étale realization. For this we will work over $\Spf(\Znumb_{p}^{\text{cyc}})$.

\begin{constr}
   Recall from \cite{BhattLurieAPC}[8.5] that choosing a compatible system of $p$-th power roots of unity 
   $(1, \xi_{p},\xi_{p}^{2},\dots)$ in the Tate module of $\Znumb_{p}^{\text{cyc}}$
   determines a map 
      $$\epsilon \from \Oh_{(\Znumb_{p}^{\text{cyc}})^{\text{syn}}}\{-1\} \to \Oh_{(\Znumb_{p}^{\text{cyc}})^{\text{syn}}}.$$
   Thus from this we obtain a section 
      $$\upsilon_{1}= \epsilon^{p-1} \from (\Znumb_{p}^{\text{cyc}})^{\text{syn},\Solid} \to \overline{\V(\Oh\{-(p-1)\})}\footnote{The vector bundle here is interpreted purly algebraically. In particular, it is proper.}.$$
   Looking at the algebraic non-vanishing locus\footnote{Also this is a closed substack. That is, locally we just invert an element algebraically.} of this section 
   and further base changing along $\colimit_{n}\Spa(\Znumb/p^{n}) \to \Spa(\Znumb_{p})$, we obtain 
   a map 
     $$(\Qnumb_{p}^{\text{cyc}})^{\text{ét},\Solid} \to (\Znumb_{p}^{\text{cyc}})^{\text{syn},\Solid}.$$
\end{constr}

\begin{defn}
   For a $p$-adic formal scheme $X$ over $\Znumb_{p}^{\text{cyc}}$, we will refer to the fibre product 
\[\begin{tikzcd}
	{X_{\eta}^{\text{ét},\Solid}} & {X^{\text{syn},\Solid}} \\
	{(\Qnumb_{p}^{\text{cyc}})^{\text{ét},\Solid}} & {(\Znumb_{p}^{\text{cyc}})^{\text{syn},\Solid}}
	\arrow[from=1-1, to=1-2]
	\arrow[from=1-1, to=2-1]
	\arrow[from=1-2, to=2-2]
	\arrow[from=2-1, to=2-2]
\end{tikzcd}\]
   as the \emph{étale locus} of the solid Syntomification of $X$. 
\end{defn}

\begin{non}\label{computation of etale locus of semiperfectoid}
   Let us compute how $R_{\eta}^{\text{ét},\Solid}$ looks like for $R$ a semiperfectoid. 
   First, using \cite{BhattLurieAPC}[8.5.3+2.6.1], we see that the fibre product 
\[\begin{tikzcd}
	Q & {R^{\Nyg,\Solid}} \\
	{(\Qnumb_{p}^{\text{cyc}})^{\text{ét},\Solid}} & {(\Znumb_{p}^{\text{cyc}})^{\text{syn},\Solid}}
	\arrow[from=1-1, to=1-2]
	\arrow[from=1-1, to=2-1]
	\arrow[from=1-2, to=2-2]
	\arrow[from=2-1, to=2-2]
\end{tikzcd}\]
   is given by 
      $$Q\iso \colimit_{n}\Spa(\Prism_{R}[\frac{1}{I}]/p^{n},\Prism_{R})$$
   and we also get the same formula for the Prismatisation. 
   As we identify the Frobenius with the identity, we thus obtain the following 
   pushout formula 
\[\begin{tikzcd}
	{Q^{\text{perf}}\amalg Q^{\text{perf}}} & {Q^{\text{perf}}} \\
	{Q^{\text{perf}}} & {R_{\eta}^{\text{ét},\Solid}}
	\arrow["{(id,\varphi)}", from=1-1, to=1-2]
	\arrow["can"', from=1-1, to=2-1]
	\arrow[from=1-2, to=2-2]
	\arrow[from=2-1, to=2-2]
\end{tikzcd}\]
   where we write 
      $$Q^{\text{perf}}\iso \colimit_{n}\Spa(\Prism_{R,\text{perf}}[\frac{1}{I}]/p^{n},\Prism_{R,\text{perf}})$$
   using perfect prismatic cohomology \cite{Bhatt2019PrismsAP}[7+8].
\end{non}

We now obtain the following. 

\begin{thm}\label{Etale locus theorem}
   The functor 
      $$\QCoh((\_)_{\eta}^{\text{ét},\Solid})\from \fSch_{/\Spf(\Znumb_{p}^{\text{cyc}})}^{\text{op}} \to \PrL$$
   extends to a six-functor formalism on $p$-adic formal schemes over $\Znumb_{p}^{\text{cyc}}$, which satisfies 
   analogous assertions to (A)-(F). Furthermore, the dualizable objects naturally identify 
      $$\QCoh((\_)_{\eta}^{\text{ét},\Solid})^{\text{dual}}\iso \QCoh_{\text{lisse}}^{\text{b}}((\_)_{\eta},\Znumb_{p})$$
   with the category of lisse étale sheaves on the generic fiber. 
\end{thm}

\begin{proof}
   As explained in \ref{realizations construction}, the first assertion follows from \ref{Main theorem}. 
   To compute the dualizable objects, we first claim that it is enough to construct this equivalence 
   on integral perfectoids. On the left, this follows from descending to semiperfectoids and then using 
   \ref{computation of etale locus of semiperfectoid}. For the right, this follows from arc-descent 
   \cite{Bhatt2018TheA}[5.13] for constructible étale sheaves together with the fact that 
   perfectoids form a basis for the arc-topology \cite{Bhatt2019PrismsAP}[8.8].

   Now using \cite{GregPseudocoherentAP}[5.50] and \ref{computation of etale locus of semiperfectoid}, we see that on a perfectoid $R$
   the left-hand side is computed by 
      $$\Perf(\Ainf(R)[\frac{1}{d}]_{\widehat{p}})^{\varphi = id}.$$
   Thus the claim follows from Tilting and the Artin-Schreier sequence (see \cite{BhattPrismaticFcristals}[3.7]).
\end{proof}

\begin{rem}
   If one is happy to consider $\Fp$-coefficients, \ref{Etale locus theorem} can be extended to all $p$-adic 
   formal schemes \cite{Bhatt2022SyntomicCA}[2.7].
\end{rem}

\begin{rem}
   The étale realization considered here is somewhat suboptimal. For example 
      $$\QCoh((\_)_{\eta}^{\text{ét},\Solid})$$
   does not satisfy excision\footnote{This is also known as Localization or Kashiwara's lemma.} and thus cannot 
   capture all étale sheaves. This issue will be addressed elsewhere. 
\end{rem}

\bibliographystyle{alpha}
\bibliography{Paperbib}  

\end{document}